\documentclass[11pt,a4paper]{amsart}

\usepackage{hyperref}
\usepackage{color}
\usepackage{graphicx}
\usepackage{amsmath,amsfonts,amssymb, amstext}
\usepackage{mathrsfs,amsthm, amscd, url}
\usepackage{psfrag}
\usepackage[all,graph]{xy}

\title{A Second Order Algebraic Knot Concordance Group}

\author{Mark Powell}
\address{Department of Mathematics, Rawles Hall, 831 East Third Street, Bloomington, IN 47405, USA}
\email{macp@indiana.edu}

\xyoption{import}
\def\R{\mathbb{R}}
\def\Z{\mathbb{Z}}

\def\Q{{\mathbb{Q}}}
\def\C{{\mathcal{C}}}

\def\CP1{\mathbb{CP}^1}

\def\wt{\widetilde}
\def\ol{\overline}
\def\a{\alpha}

\def\g{\gamma}
\def\G{\Gamma}

\def\K{\mathcal{K}}
\def\Ups{\Upsilon}
\def\zh{\Z[\Z \ltimes H]}
\def\zhd{\Z[\Z \ltimes H']}
\def\zhdag{\Z[\Z \ltimes H^{\dag}]}
\def\zhddag{\Z[\Z \ltimes H^{\ddag}]}
\def\zhpc{\Z[\Z \ltimes H^\%]}
\def\COT{Cochran-Orr-Teichner }
\def\COTN{Cochran-Orr-Teichner}

\def\zpx{\Z[\pi_1(X)]}
\def\px{\pi_1(X)}

\def\P{\mathcal{P}}
\def\Y{\mathcal{Y}}
\def\ac2{\mathcal{AC}_2}

\def\ba{\begin{array}}
\def\ea{\end{array}}
\def\bn{\begin{enumerate}}
\def\en{\end{enumerate}}

\def\toiso{\xrightarrow{\simeq}}

\def\mc{\mathcal}

\theoremstyle{plain}
\newtheorem*{theorem*}{Theorem}
\newtheorem{theorem}{Theorem}[section]
\newtheorem{proposition}[theorem]{Proposition}
\newtheorem{lemma}[theorem]{Lemma}
\newtheorem{corollary}[theorem]{Corollary}

\newtheorem{question}[theorem]{Question}

\theoremstyle{definition}
\newtheorem{definition}[theorem]{Definition}
\newtheorem{remark}[theorem]{Remark}

\newtheorem*{fact*}{Fact}

\DeclareMathOperator{\Ext}{Ext}
\DeclareMathOperator{\Hom}{Hom}

\DeclareMathOperator{\im}{im}

\DeclareMathOperator{\coker}{coker}

\DeclareMathOperator{\Id}{Id}

\DeclareMathOperator{\Bl}{Bl}

\newcommand{\eps}{\varepsilon}
\begin{document}

\title{A Second Order Algebraic Knot Concordance Group}



\begin{abstract}
Let $\mathcal{C}$ be the topological knot concordance group of knots $S^1 \subset S^3$ under connected sum modulo slice knots.  Cochran, Orr and Teichner defined a filtration:
\[\C \supset \mathcal{F}_{(0)} \supset \mathcal{F}_{(0.5)} \supset \mathcal{F}_{(1)} \supset \mathcal{F}_{(1.5)} \supset \mathcal{F}_{(2)} \supset \dots \]
The quotient $\mathcal{C}/\mathcal{F}_{(0.5)}$ is isomorphic to Levine's algebraic concordance group; $\mathcal{F}_{(0.5)}$ is the algebraically slice knots.  The quotient $\mathcal{C}/\mathcal{F}_{(1.5)}$ contains all metabelian concordance obstructions.

Using chain complexes with a Poincar\'{e} duality structure, we define an abelian group $\mathcal{AC}_2$, our \emph{second order algebraic knot concordance group}.
We define a group homomorphism $\mathcal{C} \to \mathcal{AC}_2$ which factors through $\mathcal{C}/\mathcal{F}_{(1.5)}$, and we can extract the two stage Cochran-Orr-Teichner obstruction theory from our single stage obstruction group $\mathcal{AC}_2$.  Moreover there is a surjective homomorphism $\mathcal{AC}_2 \to \C/\mathcal{F}_{(0.5)}$, and we show that the kernel of this homomorphism is non--trivial.
\end{abstract}



\maketitle

\section{Introduction}

A \emph{knot} is an oriented, locally flat embedding of $S^1$ in the  3--sphere.  We say that two knots $K$ and $K'$ are \emph{concordant} if there exists an oriented, locally flat embedding of an annulus $C = S^1 \times I$ in $S^3 \times I$ with $C \cap S^3 \times \{0\} = K$ and $C \cap S^3 \times \{1\} = -K'$.  The monoid of knots under connected sum becomes a group when we factor out by the equivalence relation of concordance, called the {\em knot concordance group}, and denoted by $\C$.

This paper unifies previously known obstructions to the concordance of knots by using chain complexes with a Poincar\'e duality structure.  In particular, we attempt to find an algebraic formulation that computes portions of the knot concordance group as filtered by the work of T.~Cochran, K.~Orr and P.~Teichner.

We view this as an initial framework for extending the algebraic theory of surgery of A. Ranicki~\cite{Ranicki3} to classification problems involving $3$-- and $4$-- dimensional manifolds.  In order to apply Ranicki's machinery to low dimensional problems, we incorporate extra information which keeps track of the effect of duality on the fundamental groups involved.

The paper~\cite{COT} introduced a filtration of the classical knot concordance group $\C$ by subgroups:
\[
\C \supset \mathcal{F}_{(0)} \supset \mathcal{F}_{(0.5)} \supset \mathcal{F}_{(1)} \supset \mathcal{F}_{(1.5)} \supset \mathcal{F}_{(2)} \supset \dots.
\]
Knots in the subgroup $\mathcal{F}_{(n)}$ are called $(n)$-solvable knots, for $n \in \frac{1}{2} \mathbb{N} \cup \{0\}$.
The subgroups $\mathcal F_{(n)}$ are geometrically defined.  A knot is $(n)$-solvable if there is {\em some choice} of four manifold whose boundary is zero-framed surgery on the knot, and which is an $n^{th}$ order approximation to the exterior of a slice disk (See Definition~\ref{Defn:COTnsolvable}).

In this paper, we focus on the $(0.5),(1)$ and $(1.5)$ levels of this filtration, corresponding to abelian and metabelian quotients of knot groups and of the fundamental groups of appropriate 4--manifolds.  Our methods extend to the higher terms of the filtration, which will appear in a future paper.  (For an outline, see the appendix of \cite{Powellthesis}.)  As in \cite[Theorem~1.1~and~Remark~1.3.2]{COT}, the quotient $\mathcal{C}/\mathcal{F}_{(0.5)}$ is isomorphic to Levine's algebraic concordance group \cite{Levine}, which we denote $\mathcal{AC}_1$ (see Definition \ref{Defn:AC1}).  We produce a purely {\em algebraically defined} group of concordance invariants, $\mathcal {AC}_{2}$, and prove the following theorem.

\begin{theorem}\label{maintheorem}
There exists a second order algebraic knot concordance group $\ac2$, with a non-trivial homomorphism $\C \to \ac2$ which factors through $\C/\mathcal{F}_{(1.5)}$.  There is a commutative diagram
\[\xymatrix{\C \ar[r] \ar @{->>} [dr] & \ac2 \ar @{->>} [d] \\ & \mathcal{AC}_1}\]
with both of the maps to $\mathcal{AC}_1$ surjections.  A knot whose image in $\ac2$ is trivial has vanishing \COT $(1.5)$-level obstructions.  Moreover, the \COT obstructions can be extracted algebraically from an element of $\ac2$.  In particular the Cheeger-Gromov Von Neumann $\rho$-invariants used in \cite{COT} can be defined purely algebraically and used to detect non-triviality of elements of $\ac2$.  \end{theorem}

We will define (Definition \ref{defn:COTobstructionset_2}) a pointed set which encapsulates the \COT obstruction theory in a single object, which we denote $\mathcal{COT}_{(\C/1.5)}$.  We summarise Theorem \ref{maintheorem} in the following commutative diagram, where dotted arrows are used to denote morphisms of pointed sets.

\[\xymatrix{
 & & \ac2 \ar @{-->} [d] \ar[dr] &  \\
 \C \ar[r] & \C/\mathcal{F}_{(1.5)} \ar @{-->} [r] \ar[dr] \ar[ur] & \mathcal{COT}_{(\C/1.5)} & \mathcal{AC}_1.\\
 & & \C/\mathcal{F}_{(0.5)} \ar[ur]  & }\]

Our aim is to compute the group $\C/\mathcal F_{(1.5)}$ and we view Theorem \ref{maintheorem} as a first step toward this goal.

\begin{question}\label{question:is_it_an_iso}
How close is our homomorphism $\C/\mathcal F_{(1.5)} \to \mathcal {AC}_{2}$ to a (rational) isomorphism?  Can we identify elements in the kernel and cokernel?
\end{question}

The following corollary of Theorem \ref{maintheorem} is a consequence of \cite{Kim04} and \cite{COT2}.

\begin{corollary}\label{corol:infinite_rank_kernel}
The kernel of $\ac2 \to \mathcal {AC}_{1}$ is of infinite rank.
\end{corollary}

The first examples of knots which lie in the kernel of the map $\C \to \mathcal{AC}_1$ were given by A.~Casson and C.~Gordon in \cite{CassonGordon}.  Their seminal work was the basis for the work of Cochran, Orr and Teichner.  We expect it to be the case that a knot whose image in $\ac2$ is trivial also has vanishing Casson-Gordon slice obstructions, but we do not directly address this in the present work.

Cochran-Orr-Teichner concordance obstructions are a secondary obstruction theory in a similar manner to obstructions to lifting a map up a tower of fibrations, or extending a map over the skeleta of a CW-complex.  One uses the vanishing at each level of obstructions to define new obstructions, which if they in turn vanish, can be used to define further obstructions, and so on.  A knot being $(n)$-solvable implies that there is some path of vanishing \COT obstructions of length $\lceil n \rceil$.  By contrast, $\ac2$ contains well-defined knot concordance invariants, which do not need to be indexed by choices of lower level vanishing.

The approach is partially inspired by work of Gilmer~\cite{Gilmer}.  He defined an analogue of $\mathcal {AC}_{2}$ which attempted to capture invariants from $\mathcal {AC}_{1}$ together with Casson-Gordon invariants.  That influential, and still important paper, has errors relating to the universal coefficient theorem.  We avoid such problems by defining our group using chain complexes with symmetric structure instead of forms defined on homology.  A chain complex with symmetric structure is a purely algebraic analogue of a Poincar\'e duality space.  Consequently, our work has an altogether different character from Gilmer's.

By avoiding homology pairings and the associated universal coefficient theorem issues in the definition of our invariant, we avoid Ore localisation, the ad-hoc introduction of principal ideal domains, and we obtain a \emph{group} with a \emph{homomorphism} $\C \to \ac2$: the chain complexes behave well under connected sum.  Traditionally, cobordism groups use disjoint union to define their addition operation.  Our operation of addition mirrors much more closely the geometric operation of addition of knots.  The most important advantage derived from defining our obstruction in terms of chain complexes is that we have a \emph{single stage} obstruction which captures the first two main stages of the \COT obstruction theory.  Finally, since we keep the whole chain complex as part of our data, we potentially have more information than can be gleaned from the \COT obstructions, although computable invariants are elusive at present.

\subsection{Organisation of the paper}
The paper is devoted to the proof of Theorem \ref{maintheorem}.  Section \ref{Section:preliminaries} contains some definitions and constructions which will be central to the rest of the paper, including the definition of a symmetric Poincar\'{e} triad and the structure and behaviour of metabelian quotients of knot groups.  We define a monoid of chain complexes $\P$ in Section \ref{Section:monoid}, corresponding to the monoid of knots under connected sum.  In Section \ref{Chapter:algconcordance}, we impose an extra equivalence relation on $\P$ corresponding to concordance of knots, and so define the group $\ac2$.  Section \ref{Section:1.5solvable=>2nd_alg_slice} contains the proof that $(1.5)$-solvable knots map to the trivial element of $\ac2$.  Section \ref{Chapter:extracting_1st_order_obstructions} describes the homomorphism to the algebraic concordance groups and proves the facts about Blanchfield forms which will be required in subsequent sections.  Section \ref{Chapter:COTobstructionthy} defines the \COT obstruction set, before Section \ref{Chapter:extractingCOTobstructions} shows how to extract the \COT obstructions from an element of $\ac2$, showing that $\ac2$ is non-trivial.

\subsection*{Acknowledgements}

This work is a shortened version of my PhD thesis.  Most of all, I would like to thank my supervisor Andrew Ranicki, for all the help he has given me over the last three and a half years, in particular for suggesting this project, and for the ideas and advice which were instrumental in solving so many of the problems encountered.  I would also like to thank in particular Stefan Friedl, Kent Orr and Peter Teichner, and also Spiros Adams-Florou, Julia Collins, Diarmuid Crowley, Wolfgang L\"{u}ck, Tibor Macko, Daniel Moskovich, Paul Reynolds and Dirk Sch\"{u}tz for many helpful and stimulating conversations and generous advice.

\section{Preliminaries}\label{Section:preliminaries}

\subsection{Symmetric structures on chain complexes representing manifolds with boundary}

All of the chain complexes in this paper will come equipped with an algebraic Poincar\'{e} duality structure: the symmetric structure of Mischenko and Ranicki.  In this section we collect the basic constructions which we will need in order to define algebraic cobordisms.  For more details on the material presented here, see \cite[Part~I]{Ranicki3}, from which the definitions in this section are taken, and \cite{Powellthesis}, where I gave an extended explanation of the derivation of symmetric structures, and in particular of how to produce one explicitly for a knot exterior.

In the following we let $A$ be a ring with involution.  A symmetric chain complex over $A$ is a chain complex $C$ together with an element $\varphi \in Q^n(C)$: we refer to \cite[Part~I,~Page~101]{Ranicki3} for the definition of the symmetric $Q$-groups $Q^n(C)$.  A symmetric pair over $A$ is a chain map $f \colon C \to D$ with an element $(\delta\varphi,\varphi) \in Q^{n+1}(f)$.  Likewise, we refer to \cite[Part~I,~Pages~133--4]{Ranicki3} for the definition of the relative $Q$-groups.  Such complexes are said to be Poincar\'{e} if the symmetric structure induces, respectively, the Poincar\'{e} and Poincar\'{e}-Lefschetz duality isomorphisms between cohomology and homology.

We can represent a manifold with boundary in two ways: on the one hand, as a symmetric Poincar\'{e} pair, and on the other hand as a symmetric complex which is \emph{not} Poincar\'{e}.
The algebraic Thom and algebraic Poincar\'{e} thickening constructions of the following definition make the correspondence between these two representations of a manifold with boundary precise.

\begin{definition}[\cite{Ranicki3}]\label{Defn:algThomcomplexandthickening}
An $n$-dimensional symmetric complex $(C,\varphi \in Q^n(C,\eps))$ is \emph{connected} if $H_0(\varphi_0 \colon C^{n-*} \to C_*) = 0.$
The \emph{algebraic Thom complex} of an $n$-dimensional $\eps$-symmetric Poincar\'{e} pair over $A$
\[(f \colon C \to D, (\delta \varphi, \varphi) \in Q^n(f,\eps))\]
is the connected $n$-dimensional $\eps$-symmetric complex over $A$
\[(\mathscr{C}(f),\delta \varphi/\varphi \in Q^n(\mathscr{C}(f),\eps))\]
where $\mathscr{C}(f)$ is the algebraic mapping cone of $f$, and
\begin{eqnarray*}(\delta\varphi/\varphi)_s &:=& \left( \begin{array}{cc} \delta \varphi_s & 0 \\ (-1)^{n-r-1}\varphi_s f^* & (-1)^{n-r+s}T_{\eps} \varphi_{s-1} \end{array} \right) \colon \mathscr{C}(f)^{n-r+s} \\ &=& D^{n-r+s} \oplus C^{n-r+s-1} \to \mathscr{C}(f)_r = D_r \oplus C_{r-1} \;\; (s \geq 0).\end{eqnarray*}

The \emph{boundary} of a connected $n$-dimensional $\eps$-symmetric complex $(C,\varphi \in Q^n(C,\eps))$ over $A$, for $n \geq 1$, is the $(n-1)$-dimensional $\eps$-symmetric \emph{Poincar\'{e}} complex over $A$
\[(\partial C,\partial \varphi \in Q^{n-1}(\partial C,\eps))\]
given by:
\[d_{\partial C} = \left(\begin{array}{cc} d_C & (-1)^r \varphi_0 \\ 0 & \partial^*=d_{C^{n-*}} \end{array} \right) \colon \partial C^{r} = C_{r+1} \oplus C^{n-r} \to \partial C_r = C_{r} \oplus C^{n-r+1};\]
\begin{multline*} \partial \varphi_0 = \left(\begin{array}{cc} (-1)^{n-r-1}T_{\eps}\varphi_1 & (-1)^{r(n-r-1)}\eps \\ 1 & 0 \end{array} \right) \colon \partial C^{n-r-1} = C^{n-r} \oplus C_{r+1} \\ \to \partial C_r = C_{r+1} \oplus C^{n-r}; \end{multline*}
and, for $s \geq 1$,
\begin{multline*} \partial \varphi_s = \left(\begin{array}{cc} (-1)^{n-r+s-1}T_{\eps}\varphi_{s+1} & 0 \\ 0 & 0 \end{array} \right) \colon \partial C^{n-r+s-1} = C^{n-r+s} \oplus C_{r-s+1} \\ \to \partial C_r = C_{r+1} \oplus C^{n-r}. \end{multline*}
See \cite[Part~I,~Proposition~3.4~and~pages~141--2]{Ranicki3} for the full details on the boundary construction.

The \emph{algebraic Poincar\'{e} thickening} of a connected $\eps$-symmetric complex over $A$, $(C,\varphi \in Q^{n}(C,\eps)),$
is the $\eps$-symmetric Poincar\'{e} pair over $A$:
\[(i_C \colon \partial C \to C^{n-*}, (0,\partial \varphi ) \in Q^n(i_C,\eps)),\] where
$i_C = (0 , 1) \colon \partial C = C_{r+1} \oplus C^{n-r} \to C^{n-r}.$
The algebraic Thom complex and algebraic Poincar\'{e} thickening are inverse operations \cite[Part~I,~Proposition~3.4]{Ranicki3}.
\qed \end{definition}

Next, we give the definition of a symmetric Poincar\'{e} triad.  This is the algebraic version of a manifold with boundary where the boundary is split into two along a submanifold; in other words a cobordism of cobordisms which restricts to a product cobordism on the boundary.  Note that our notion is not quite as general as the notion in \cite[Sections~1.3~and~2.1]{Ranicki2}, since we limit ourselves to the case that the cobordism restricted to the boundary is a product.  We also circumvent the more involved definitions of \cite{Ranicki2}, and define the triads by means of \cite[Proposition~2.1.1]{Ranicki2}, with a sign change in the requirement of $i_-$ to be a symmetric Poincar\'{e} pair.

\begin{definition}[\cite{Ranicki2}]\label{Defn:symmPoincaretriad}
An \emph{$(n+2)$-dimensional (Poincar\'{e}) symmetric triad} is a triad of finitely generated projective $A$-module chain complexes:
\[\xymatrix @C+1cm{\ar @{} [dr] |{\stackrel{g}{\sim}}
C \ar[r]^{i_-} \ar[d]_{i_+} & D_- \ar[d]^{f_-} \\ D_+ \ar[r]_{f_+} & Y
}\]
with chain maps $i_{\pm},f_{\pm}$, a chain homotopy $g \colon f_- \circ i_- \simeq f_+ \circ i_+$ and structure maps $(\varphi,\delta\varphi_-,\delta\varphi_+,\Phi)$ such that:
$(C,\varphi)$ is an $n$-dimensional symmetric (Poincar\'{e}) complex,
\[(i_+ \colon C \to D_+,(\delta\varphi_+,\varphi)) \text{ and } (i_- \colon C \to D_-,(\delta\varphi_-,-\varphi))\]
are $(n+1)$-dimensional symmetric (Poincar\'{e}) pairs, and
\[(e \colon D_- \cup_{C} D_+ \to Y,(\Phi,\delta\varphi_- \cup_{\varphi} \delta\varphi_+)) \]
is a $(n+2)$-dimensional symmetric (Poincar\'{e}) pair, where:
\[e = \left(\begin{array}{ccc} f_- \, , & (-1)^{r-1}g \, , & -f_+ \end{array} \right) \colon (D_-)_r \oplus C_{r-1} \oplus (D_+)_r \to Y_r.\]
See \cite[Part~I,~pages~117--9]{Ranicki3} for the formulae which enable us to glue together two chain complexes along a common part of their boundaries with opposite orientations: the \emph{union construction}.  We write $(D'' = D \cup_{C'}D',\delta\varphi'' = \delta\varphi\cup_{\varphi'}\delta\varphi')$ for the union of $(D,\delta\varphi)$ and $(D',\delta\varphi')$ along $(C,\varphi')$.

A chain homotopy equivalence of symmetric triads is a set of chain equivalences:
\[\ba{rclcrcl} \nu_C\colon C &\to& C'; & \nu_{D_-} \colon D_- &\to& D'_- \\
\nu_{D_+} \colon D_+ &\to & D'_+;\text{ and}& \nu_E \colon Y &\to& Y', \ea\]
which commute with the chain maps of the triads up to chain homotopy, and such that the induced maps on $Q$-groups map the structure maps $(\varphi,\delta\varphi_-,\delta\varphi_+,\Phi)$ to the equivalence class of the structure maps $(\varphi',\delta\varphi'_-,\delta\varphi'_+,\Phi')$.  See \cite[Part I,~page~140]{Ranicki3} for the definition of the maps induced on relative $Q$-groups by an equivalence of symmetric pairs.
\qed \end{definition}

\begin{definition}\label{Defn:unionconstruction}
(\cite[Part I,~pages~134--6]{Ranicki3}) An \emph{$\eps$-symmetric cobordism} between symmetric complexes $(C,\varphi)$ and $(C',\varphi')$ is a $(n+1)$-dimensional $\eps$-symmetric Poincar\'{e} pair with boundary $(C \oplus C',\varphi \oplus -\varphi')$:
\[((f_C,f_{C'}) \colon C \oplus C' \to D,(\delta\varphi,\varphi \oplus -\varphi')\in Q^{n+1}((f_C,f_{C'}),\eps)).\]
\qed \end{definition}

The next lemma contains a fact which is key for constructing algebraic cobordisms corresponding to products $M \times I$.  We place it here so as not to interrupt the main text; we will have repeated cause to appeal to it.  Although this is well-known to the experts, I have not found a proof in the literature.

\begin{lemma}\label{lemma:productcobordism}
Given a homotopy equivalence $f \colon (C,\varphi) \to (C',\varphi')$
of $n$-dimensional symmetric Poincar\'{e} chain complexes such that $f^{\%}(\varphi) = \varphi'$, there is a symmetric cobordism $((f,1) \colon C \oplus C' \to C',(0,\varphi \oplus - \varphi')).$  This symmetric pair is also Poincar\'{e}.
\end{lemma}
\begin{proof}
We need to check that the symmetric structure maps $(0,\varphi\oplus -\varphi') \in Q^{n+1}((f,1))$ induce isomorphisms:
$H^r((f,1)) \xrightarrow{\simeq} H_{n+1-r}(C').$
We use the long exact sequence in cohomology of a pair, associated to the short exact sequence
$0 \to C' \xrightarrow{j} \mathscr{C}((f,1)) \to S(C \oplus C') \to 0$
to calculate the homology $H^r((f,1))$.  The sequence is:
\[ H^{r-1}(C') \xrightarrow{(f^*,1^*)^T} H^{r-1}(C \oplus C') \xrightarrow{\partial} H^r((f,1)) \xrightarrow{j^*} H^r(C') \xrightarrow{(f^*,1^*)^T} H^r(C \oplus C').\]
We have that $\ker((f^*,1^*)^T \colon H^r(C') \to H^r(C \oplus C')) \cong 0,$ so $j^*$ is the zero map, and therefore $\partial$ is surjective.  The image $\im((f^*,1^*)^T\colon H^{r-1}(C') \to H^{r-1}(C) \oplus H^{r-1}(C'))$
is the diagonal, so that the images of elements of the form $(0,y') \in H^{r-1}(C) \oplus H^{r-1}(C')$ generate $H^r((f,1))$.

The map $H^r((f,1)) \to H_{n-r+1}(C')$ generated by $(0,\varphi \oplus - \varphi')$, on the chain level, is
\[\left( 0, \left( \begin{array}{cc} f & 1 \end{array} \right)\left( \begin{array}{cc} \varphi_0 & 0 \\ 0 & -\varphi'_0 \end{array} \right)\right) \colon (C')^r \oplus (C \oplus C')^{r-1} \to C'_{n-r+1}\]
which sends $y' \in H^{r-1}(C')$ to $-\varphi'_0(y')$.  We therefore have an isomorphism on homology since $(C',\varphi')$ is a symmetric Poincar\'{e} complex, so we have a symmetric Poincar\'{e} pair
\[((f,1) \colon C \oplus C' \to C',(0,\varphi \oplus - \varphi')),\]
as claimed.
\end{proof}

\subsection{Second derived covers and connected sum}

Our obstructions, since they aim to capture second order information, work at the level of the second derived covers of the manifolds involved.  We therefore need to understand the behaviour of the second derived quotients of knot groups.  We denote the exterior of a knot $K \subset S^3$ by
\[X := S^3 \setminus \nu K.\]

\begin{proposition}\label{prop:2ndderivedsubgroup}
Let $\phi$ be the quotient map
\[\phi \colon \pi_1(X)/\pi_1(X)^{(2)} \to \pi_1(X)/\pi_1(X)^{(1)} \xrightarrow{\simeq} \Z.\]
Then for each choice of splitting homomorphism $\psi \colon \Z \to \pi_1(X)/\pi_1(X)^{(2)}$
such that $\phi \circ \psi = \Id$, let $t:= \psi(1)$.  There is an isomorphism:
\vspace{-0.1cm}
\[\ba{rcl} \theta \colon \pi_1(X)/\pi_1(X)^{(2)} &\xrightarrow{\simeq}& \Z \ltimes H; \\ g &\mapsto& (\phi(g), gt^{-\phi(g)}), \ea\]
where $H := H_1(X;\Z[\Z])$ is the Alexander module.
\end{proposition}
\begin{proof}
This is well--known, so we omit the proof.  See e.g. \cite[page~307]{Let00}.
\end{proof}

Although the following proposition is well--known, the careful treatment of inner automorphisms, used to take care of any ostensible dependence on the choice of splitting in Proposition \ref{prop:2ndderivedsubgroup}, will be invaluable in Section \ref{Section:monoid}.

\begin{proposition}\label{prop:2ndderivedsubgroup_adding}
Let $K$, $K^\dag$ and $K^\ddag := K\, \sharp\, K^\dag$ be oriented knots, with associated exteriors $X,X^\dag$ and $X^\ddag$, and denote $H^{\dag} := H_1(X^{\dag};\Z[\Z])$ and $H^{\ddag} := H_1(X^{\ddag};\Z[\Z])$.  The behaviour of the second derived quotients under connected sum is given by:
\[\pi_1(X^{\ddag})/\pi_1(X^{\ddag})^{(2)} \cong \Z \ltimes H^{\ddag} \cong \Z \ltimes (H \oplus H^{\dag}).\]
That is, we can take the direct sum of the Alexander modules.
\end{proposition}

\begin{proof}

First we observe that \vspace{-0.2cm}$$\pi_1(X^\ddag) \cong \pi_1(X) \ast_{\Z} \pi_1(X^\dag),$$ by the Seifert-Van-Kampen theorem: the knot exterior of a connected sum is given by gluing the exteriors of the summands together along neighbourhoods of meridians $S^1 \times D^1 \subset \partial X, \partial X^\dag$.  Note that $H$, $H^{\dag}$ and $H^{\ddag}$ are modules over the group ring $\Z[t,t^{-1}]$ for the same $t$, which comes from the preferred meridian of each of $X,X^{\dag}$ and $X^{\ddag}$ respectively; when the spaces are identified these meridians all coincide.
\[\Z \ltimes H^{\ddag}  \cong \pi_1(X^{\ddag})/\pi_1(X^{\ddag})^{(2)} \cong \pi_1(X) \ast_{\Z} \pi_1(X^{\dag})/(\pi_1(X) \ast_{\Z} \pi_1(X^{\dag}))^{(2)} \]\vspace{-0.5cm}
\begin{equation}\label{eqn:2nd_derived_sums}
  \cong  \left(\frac{\pi_1(X)}{\pi_1(X)^{(2)}} \ast_{\Z} \frac{\pi_1(X^{\dag})}{\pi_1(X^{\dag})^{(2)}}\right)\text{\huge{/}}[\pi_1(X)^{(1)},\pi_1(X^{\dag})^{(1)}] \cong  \frac{(\Z \ltimes H)\ast_{\Z}(\Z \ltimes H^{\dag})}{[\pi_1(X)^{(1)},\pi_1(X^{\dag})^{(1)}]}.
\end{equation}
We now need to be sure that the two group elements which we identify, which we call $g_1 \in \pi_1(X)$ and $g_1^{\dag} \in \pi_1(X^\dag)$, map to $(1,0) \in \Z \ltimes H$ and $(1,0^{\dag}) \in \Z \ltimes H^{\dag}$ respectively under the compositions
\[\pi_1(X) \to \px/\px^{(2)} \to \Z \ltimes H \text{ and }
\pi_1(X^{\dag}) \to \pi_1(X^{\dag})/\pi_1(X^{\dag})^{(2)} \to \Z \ltimes H^{\dag}.\]
If we had chosen $\psi(1) = g_1 \in \pi_1(X)/\pi_1(X)^{(2)} \text{ and } \psi^{\dag}(1) = g_1^{\dag} \in \pi_1(X^{\dag})/\pi_1(X^{\dag})^{(2)}$ then this would be the case and we would have:
\begin{eqnarray*}
\frac{(\Z \ltimes H)\ast_{\Z}(\Z \ltimes H^{\dag})}{[\pi_1(X)^{(1)},\pi_1(X^{\dag})^{(1)}]}
 \cong  \frac{\Z \ltimes (H \ast H^{\dag})}{[H,H^{\dag}]}  \cong  \Z \ltimes (H \oplus H^{\dag}),
\end{eqnarray*}
and the proof would be complete.  The point is that we can always arrange that the image of $g_1$ is $(1,0)$ by applying an inner automorphism of $\Z \ltimes H$, and similarly for $g_1^\dag$ and $\Z \ltimes H^\dag$.  Suppose that $\theta(g_1) = (1,h_1)$.  Recall \cite[Proposition~1.2]{Levine2} that $1-t$ acts as an automorphism of $H$.  We can therefore choose $h'_1 \in H$ such that $(1-t)h'_1 = h_1$.  Then we have that:
\begin{eqnarray*}
(0,h'_1)^{-1}(1,h_1)(0,h'_1) &=& (0,-h'_1)(1,h_1)(0,h'_1) = (1,-h'_1 + h_1)(0,h'_1)\\
= (1,-h'_1 + h_1 + th'_1) &=& (1,h_1 - (1-t)h'_1) = (1,h_1-h_1) = (1,0).
\end{eqnarray*}
So, as claimed, in the last isomorphism of (\ref{eqn:2nd_derived_sums}), we can compose $\theta$ and $\theta^\dag$ with suitable inner automorphisms and so achieve the desired conditions on the meridians which we identify.  Therefore the second derived quotients of the fundamental groups indeed add under connected sum as claimed.
\end{proof}

This concludes the preliminaries that we wish to collect prior to making our main definitions.

\section{A Monoid of Chain Complexes}\label{Section:monoid}

We shall define a set of purely algebraic objects which capture the necessary information to produce concordance obstructions at the metabelian level.  We define a set comprising 3-dimensional symmetric Poincar\'{e} triads over the group ring $\zh$ for certain $\Z[\Z]$-modules $H$.
In some sense, we are to forget that these chain complexes originally arose from geometry, and to perform operations on them purely with reference to the algebraic data which we store with each element.  The primary operation which we introduce in this section is a way to add these chain complexes, so that we obtain an abelian monoid.  On the other hand, we would not do well pedagogically to forget the geometry.  The great merit of the addition operation we put forward here is that it closely mirrors geometric addition of knots by connected sum.

A manifold triad is a manifold with boundary $(X,\partial X)$ such that the boundary splits along a submanifold into two manifolds with boundary, $\partial X = \partial X_{0} \cup_{\partial X_{01}} \partial X_1.$  In our case of interest where $X$ is a knot exterior we have a manifold triad:
\[\xymatrix @R-0.3cm{
S^1 \times S^0 \ar[r] \ar[d] & S^1 \times D^1 \ar[d]\\ S^1 \times D^1 \ar[r] & X,
}\]
where the longitude is divided into two copies of $D^1$.  Such a manifold triad gives rise to a corresponding triad of chain complexes: noting that the knot exterior has the homology of a circle and the inclusion of each of the boundary components $S^1 \times D^1$ induces an isomorphism on $\Z$-homology, we obtain a chain complex $\Z$-homology cobordism from $C_*(S^1\times D^1)$ to itself, which is a product along the boundary.

The chain complexes are taken over the group rings $\zh$ of the semi--direct products which arise, as in Proposition \ref{prop:2ndderivedsubgroup}, as the quotients of knot groups by their second derived subgroups, with $H$ an Alexander module (Theorem \ref{Thm:Levinemodule}).  The crucial extra condition is a consistency condition, which relates $H$ to the actual homology of the chain complex.  Since the Alexander module changes under addition of knots and in a concordance, this extra control is vital in order to formulate a concordance obstruction theory.

We quote the following theorem of Levine, specialised here to the case of knots in dimension 3, and use it to define the notion of an abstract Alexander module.  Recall that we denote the exterior of a knot $K$ by $X := S^3 \setminus \nu K$.

\begin{theorem}[\cite{Levine2}]\label{Thm:Levinemodule}
Let $K$ be a knot, and let  $H:= H_1(X;\Z[\Z]) \cong H_1(X_{\infty};\Z)$  be its Alexander module.  Take $\Z[\Z] = \Z[t,t^{-1}]$.  Then $H$ satisfies the following properties:
\begin{description}
\item[(a)] The Alexander module $H$ is of type $K$: that is, $H$ is finitely generated over $\Z[\Z]$, and multiplication by $1-t$ is a module automorphism of $H$.  These two properties imply that $H$ is $\Z[\Z]$-torsion.
\item[(b)] The Alexander module $H$ is $\Z$-torsion free.  Equivalently, for $\Z[\Z]$-modules of type $K$, the homological dimension\footnote{This is defined as the minimal possible length of a projective resolution.} of $H$ is 1.
\item[(c)] The Alexander module $H$ satisfies Blanchfield Duality: $$\overline{H} \cong \Ext^1_{\Z[\Z]}(H,\Z[\Z]) \cong \Ext^0_{\Z[\Z]}(H,\Q(\Z)/\Z[\Z]) \cong \Hom_{\Z[\Z]}(H,\Q(\Z)/\Z[\Z])$$ where $\overline{H}$ is the conjugate module defined using the involution defined by $t \mapsto t^{-1}$.
\end{description}
Conversely, given a $\Z[\Z]$-module $H$ which satisfies properties (a), (b) and (c), there exists a knot $K$ such that $H_1(X;\Z[\Z]) \cong H$.
\end{theorem}
\begin{definition}
We say that a $\Z[\Z]$-module which satisfies (a),(b) and (c) of Theorem \ref{Thm:Levinemodule} is an Alexander module, and denote the class of Alexander modules by $\mathcal{A}$.
\qed\end{definition}

Before we give the definition of our set of symmetric Poincar\'{e} triads, we exhibit some basic symmetric chain complexes which correspond to the spaces $S^0 \times S^1$ and $S^1 \times D^1$.

\begin{definition}\label{defn:basic_chain_complexes}
Let $H$ be an Alexander module.  Let $h_1 \in H$ and define $g_1 := (1,h_1)\in \zh$.  Moreover let $l_a \in \zh$, denote $g_q := l_a^{-1}g_1 l_a$ and let $l_b:= l_a^{-1}$.  The symmetric Poincar\'{e} chain complex $(C',\varphi_{C'} = \varphi \oplus -\varphi)$, of the form:
\[\xymatrix @C+2cm @R+0.5cm{C'^0 \ar[r]^{\delta_1} \ar[d]_{\varphi_0 \oplus -\varphi_0} & C'^1 \ar[d]^{\varphi_0 \oplus -\varphi_0} \ar[dl]^{\varphi_1 \oplus -\varphi_1}\\
C'_1 \ar[r]^{\partial_1} & C'_0,}\]
is given by:
\vspace{-0.3cm}
\[\xymatrix @C+3cm @R+2.5cm{\bigoplus_2\,\zh \ar[r]^{\left(\begin{array}{cc} g_1^{-1}-1 & 0 \\ 0 & g_q^{-1}-1 \\ \end{array} \right)} \ar[d]_{\left(\begin{array}{cc} 1 & 0 \\ 0 & -1 \\ \end{array} \right)} & \bigoplus_2\,\zh \ar[d]^{\left(\begin{array}{cc} g_1 & 0 \\ 0 & -g_q \\ \end{array} \right)} \ar[dl]_{\left(\begin{array}{cc} 1 & 0 \\ 0 & -1 \\ \end{array} \right)}\\
\bigoplus_2\,\zh \ar[r]^{\left(\begin{array}{cc} g_1-1 & 0 \\ 0 & g_q-1 \\ \end{array} \right)} & \bigoplus_2\,\zh.}\]
The \emph{annular chain complexes $D'_{\pm}$} fit into symmetric Poincar\'{e} pairs:
\[(i'_{\pm} \colon C' \to D'_{\pm}, (\delta\varphi_{\pm} = 0,\varphi_{C'}));\]
(they are Poincar\'{e} pairs by Lemma \ref{lemma:productcobordism}), defined as follows:
\[\xymatrix @R+0.3cm @C-0.4cm{D'_- & \zh \ar[rrrrr]^{\left(\begin{array}{c}g_1-1\end{array}\right)} &&&&& \zh \\
C' \ar[u]^{i'_-} \ar[d]_{i'_+} & \bigoplus_2\,\zh \ar[rrrrr]_{\left(\begin{array}{cc}g_1-1 & 0 \\ 0 & g_q-1 \end{array}\right)} \ar[u]^{\left(\begin{array}{c} 1 \\ l_a^{-1} \end{array}\right)} \ar[d]_{\left(\begin{array}{c} l_b^{-1} \\ 1 \end{array}\right)} &&&&& \bigoplus_2\,\zh \ar[u]_{\left(\begin{array}{c} 1 \\ l_a^{-1} \end{array}\right)} \ar[d]^{\left(\begin{array}{c} l_b^{-1} \\ 1 \end{array}\right)} \\
D'_+ & \zh \ar[rrrrr]_{\left(\begin{array}{c}g_q-1\end{array}\right)} &&&&& \zh,
}\]
The chain complexes $D'_{\pm}$ arise by taking the tensor products
$\Z[\Z \ltimes H] \otimes_{\Z[\Z]} C_*(S^1;\Z[\Z]),$
with homomorphisms $\Z[\Z] \to \zh$ given by $t \mapsto g_1$
for $D'_-$ and $t \mapsto g_q$ for $D'_+$.  There is therefore a canonical chain isomorphism $\varpi \colon D'_- \to D'_+$
given by
\vspace{-0.1cm}
\[\xymatrix @C+1cm {\zh \ar[r]^{(g_1-1)}  \ar[d]^{(l_a)} & \zh \ar[d]^{(l_a)}  \\ \zh \ar[r]^{(g_q - 1)} & \zh.}\]
\qed\end{definition}

\begin{definition}\label{Defn:algebraicsetofchaincomplexes}
We define the set $\mathcal{P}$ to be the set of equivalence classes of triples $(H,\Y,\xi)$ where: $H \in \mathcal{A}$ is an Alexander module; $\Y$ is a 3-dimensional symmetric Poincar\'{e} triad of finitely generated projective $\zh$-module chain complexes of the form:
\[\xymatrix @C+0.5cm {\ar @{} [dr] |{\stackrel{g}{\sim}}
(C,\varphi_C) \ar[r]^{i_-} \ar[d]_{i_+} & (D_-,\delta\varphi_-) \ar[d]^{f_-}\\ (D_+,\delta\varphi_+) \ar[r]^{f_+} & (Y,\Phi),
}\]
with the symmetric Poincar\'{e} pairs $(i_{\pm} \colon C \to D_{\pm}, (\delta\varphi_{\pm},\varphi_{C}))$ chain homotopy equivalent to $(i'_{\pm} \colon C' \to D'_{\pm}, (0,\varphi \oplus -\varphi))$ from Definition \ref{defn:basic_chain_complexes}, where the chain maps $f_{\pm}$ induce $\Z$-homology equivalences, and with a chain homotopy $g \colon f_- \circ i_- \sim f_+ \circ i_+ \colon C_* \to Y_{*+1}$; and $$\xi \colon H \toiso H_1(\Z[\Z] \otimes_{\zh} Y)$$ is a $\Z[\Z]$-module isomorphism.

Moreover we require that the maps $\delta\varphi_{\pm}$ have the property that $\varpi \delta\varphi_- \varpi^* = - \delta\varphi_+$, and that there is a chain homotopy $\mu \colon f_+ \circ \varpi \simeq f_-.$  This implies that objects of our set are independent of the choice of $f_-$ and $f_+$.

The maps $f_{\pm}$ must induce $\Z$-homology isomorphisms; note that $H_*(\Z \otimes_{\zh} D_{\pm}) \cong H_*(S^1;\Z)$:
\[(f_{\pm})_* \colon H_*(\Z \otimes_{\zh} D_{\pm}) \xrightarrow{\simeq} H_*(\Z \otimes_{\zh} Y).\]
We call the condition that the isomorphism $\xi \colon H \xrightarrow{\simeq} H_1(\Z[\Z] \otimes_{\zh} Y)$
exists, the \emph{consistency condition}, and we call $\xi$ the \emph{consistency isomorphism}.

We say that two triples $(H,\Y,\xi)$ and $(H^\%,\Y^\%,\xi^\%)$ are equivalent if there exists a $\Z[\Z]$-module isomorphism $\omega \colon H \toiso H^\%$, which induces a ring isomorphism $\zh \toiso \zhpc$, and if there exists a chain equivalence of triads $j \colon \Z[\Z \ltimes H^\%] \otimes_{\zh} \Y \to \Y^\%,$ such that the following diagram commutes:
\[\xymatrix @C+1cm{H \ar[r]_-{\xi}^-{\cong} \ar[d]_{\omega}^-{\cong} & H_1(\Z[\Z] \otimes_{\zh} Y) \ar[d]_-{j_*}^-{\cong} \\
H^\% \ar[r]_-{\xi^\%}^-{\cong} & H_1(\Z[\Z] \otimes_{\zhpc} Y^\%). }\]
The induced map $j_*$ on $\Z[\Z]$-homology makes sense, as there is an isomorphism $\Z[\Z] \cong \Z[\Z] \otimes_{\zhpc} \zhpc,$
so that $$H_1(\Z[\Z] \otimes_{\zh} Y) \toiso H_1(\Z[\Z] \otimes_{\zhpc} \zhpc \otimes_{\zh} Y).$$
It is easy to see that we have indeed described an equivalence relation: symmetry is seen using the inverses of the vertical arrows and transitivity is seen by vertically composing two such squares.
\qed\end{definition}

Given a knot $K$ with exterior $X$, we define a triple $(H,\Y,\xi)$ as follows.  Let $H := \pi_1(X)^{(1)}/\pi_1(X)^{(2)}$ considered as a $\Z[\Z]$-module via the action given by conjugation with a meridian. Let $\Y$ be the triad of handle chain complexes associated to the $\pi_1(X)^{(2)}$--cover of the manifold triad
\vspace{-0.2cm} \[\xymatrix @R-0.2cm{S^1 \times S^0 \ar[r] \ar[d] & S^1 \times D^1_+ \ar[d] \\ S^1 \times D^1_- \ar[r] & X,}\] with symmetric structures for $C_*(S^1 \times S^0)$ and $C_*(S^1\times D^1_{\pm})$ as given in Definition \ref{defn:basic_chain_complexes}, and with the symmetric structure for $C_*(X)$ given by the image under a chain level approximation to the diagonal map\[\Delta \colon C(X;\Z) \to C(X;\zh) \otimes_{\zh} C(X;\zh)\] of a relative fundamental class $[X,\partial X] \in C_3(X;\Z)$.  Lastly, let $\xi$ be the Hurewicz isomorphism \vspace{-0.25cm}$$\xi \colon H \toiso H_1(X;\Z[\Z]) \cong H_1(\Z[\Z] \otimes_{\Z[\Z \ltimes H]} Y).$$  Then we have:
\begin{proposition}\label{prop: fundtriaddefinesanelement}
Let $Knots$ be the set of isotopy classes of locally flat oriented knots.  The above association of $(H,\Y,\xi)$ to a knot $K$ defines a function: $$Knots \to \mathcal{P}.$$
\end{proposition}

\begin{proof}
We take \[Y := C(X;\zh) := \zh \otimes_{\zpx} C(X;\zpx),\] using the handle chain complex of $X$ with coefficients twisted by the group ring of the fundamental group.  We use a handle decomposition which contains a handle decomposition of a regular neighbourhood of the boundary $\partial X \times I$ as a subcomplex.  We split the boundary into two annular pieces $S^1 \times S^1 = S^1 \times D^1_+ \cup_{S^1 \times S^0} S^1 \times D^1_-$, with the longitude split in two.  We pick a meridian of $K$ and call it $g_1 \in \pi_1(X)$, and we let $l_a$ and $l_b$ be the images in $\pi_1(X)/\pi_1(X)^{(2)}$ of the two halves of the longitude, suitably based.  Take $(C,\varphi_C),(D_{\pm},\delta\varphi_{\pm})$ and $i_{\pm}$ to be the complexes defined in Definition \ref{defn:basic_chain_complexes}.  Define the maps $f_{\pm}$ and $g$ to be the maps induced by the inclusion of the boundary. The symmetric structure $\Phi$ on $Y_* = C_*(X;\zh)$ is given, as described, by the image of a relative fundamental class under a diagonal approximation chain map.  Note that for the model chain complexes, $\varpi = (l_a) \colon (D_-)_i \to (D_+)_i$ so $f_+ \circ \varpi = f_-$ and we can take $\mu = 0$.

It is important that our objects do not depend on choices, so that equivalent knots define equivalent triads.
Different choices of $l_a$ and $l_b$ affect these elements only up to a conjugation, or in other words an application of an inner automorphism, which means we can vary $C, D_{+}$ and $f_+$ by a chain isomorphism and obtain chain equivalent triads.  A different choice of element $g_1 = (1,h_1) \in \Z \ltimes H$ is related by a conjugation, or in other words an application of an inner automorphism, as in the proof of Proposition \ref{prop:2ndderivedsubgroup_adding}, so that we can change $C, D_{\pm}$ and $Y$ by chain isomorphisms and obtain chain equivalent triads.  The point is that we need to make choices of $g_1$ and of $l_a$ in order to write down a representative of an equivalence class of symmetric Poincar\'{e} triads, but still different choices yield equivalent triads.  We investigate the effect of such changes on the consistency isomorphism $\xi$.  A change in $l_a$ does not affect the isomorphism $\xi$.  A change in $g_1$ affects $\xi$ as follows.  When we wish to change the boundary maps and chain maps in a triad by applying an inner automorphism, conjugating by an element $h \in \Z \ltimes H$ say, we define the chain equivalence of triads $\Y \to \Y^\%$ which maps basis elements of all chain groups as follows: $e_i \mapsto he_i$: $\Y^\%$ has the same chain groups as $\Y$ but with the relevant boundary maps and chain maps conjugated by $h$.  This induces an isomorphism which by a slight abuse of notation we denote $h_* \colon H_1(\Z[\Z] \otimes_{\zh} Y) \toiso H_1(\Z[\Z] \otimes_{\zhpc} Y^\%)$.  We take $\omega \colon H \to H^\% = H$ as the identity.  In order to obtain an equivalent triple, we therefore take $\xi^\% = h_* \circ \xi$.

An isotopy of knots induces a homeomorphism of the exteriors $X \xrightarrow{\approx} X^\%$, fixing the boundary, which itself induces an isomorphism $$\omega \colon \pi_1(X)^{(1)}/\pi_1(X)^{(2)} = H \toiso \pi_1(X^\%)^{(1)}/\pi_1(X^\%)^{(2)} = H^\%.$$
Likewise the isotopy induces an equivalence of triads $\zhpc \otimes_{\zh} \Y \to \Y^\%$.  The geometrically defined maps $\xi$ and $\xi^\%$ fit into the commutative square as required in Definition \ref{Defn:algebraicsetofchaincomplexes}.



Finally, we should check that the conditions on homology for an element of $\P$ are satisfied.  First, $\Z \otimes_{\zh} D_{\pm}$ is given by $\Z \xrightarrow{0} \Z,$ which has the homology of a circle.  Alexander duality or an easy Mayer-Vietoris argument using the decomposition of $S^3$ as $X \cup_{\partial X \approx S^1 \times S^1} S^1 \times D^2$ shows that $H_*(C_*(X;\Z)) \cong H_*(S^1;\Z)$, with the generator of $H_1(X;\Z)$ being any of the meridians.  So the chain maps $\Id \otimes_{\zh} f_{\pm} \colon \Z \otimes D_{\pm} \to C_*(X;\Z)$ induce isomorphisms on homology.

The consistency condition is satisfied, since we have the canonical Hurewicz isomorphism $H \toiso H_1(X;\Z[\Z])$ as claimed. Therefore, we have indeed defined an element of $\mathcal{P}$.
\end{proof}

\begin{remark}
In \cite{Powellthesis}, I gave an algorithm to construct a symmetric Poincar\'{e} triad explicitly, given a diagram of a knot, using a handle decomposition of the knot exterior.  The novel part of this was to construct the symmetric structure maps explicitly, at the level of the universal cover.
\end{remark}

We now define the notion of addition of two triples $(H,\mathcal{Y},\xi)$ and $(H^\dag,\mathcal{Y}^{\dag},\xi^\dag)$ in $\mathcal{P}$.  In the following, the notation should be transparent: everything associated to $\mathcal{Y}^{\dag}$ will be similarly decorated with a dagger.

\begin{definition}\label{Defn:connectsumalgebraic}
We define the sum of two triples
\[(H^\ddag,\mathcal{Y}^{\ddag},\xi^\ddag) = (H,\mathcal{Y},\xi) \, \sharp \, (H^\dag,\mathcal{Y}^{\dag},\xi^\dag),\]
as follows.  The first step is to make sure that the two triads are over the same group ring.  Pick a representative in the equivalence class of each of the triples on the right hand side which satisfy $g_1 = (1,0)$ and $g_1^\dag = (1,0^\dag)$ respectively.  It was explained how to achieve this, with the application of inner automorphisms of $\Z \ltimes H$ and $\Z \ltimes H^\dag$, in the proofs of Propositions \ref{prop:2ndderivedsubgroup_adding} and \ref{prop: fundtriaddefinesanelement}.  Now define $H^{\ddag} := H \oplus H^{\dag}$.
We use the homomorphisms
\vspace{-0.1cm}
\[\ba{rcl} \Z \ltimes H & \to & \Z \ltimes (H \oplus H^\dag); \\ (n,h) & \mapsto &(n,(h,0^\dag)) \ea \vspace{-0.2cm}\]
and
\vspace{-0.1cm}
\[\ba{rcl} \Z \ltimes H^\dag & \to & \Z \ltimes (H \oplus H^\dag); \\ (n,h^\dag) & \mapsto &(n,(0,h^\dag)) \ea\]
to form the tensor products
$\zhddag \otimes_{\zh} \Y $
and
$\zhddag \otimes_{\zhdag} \Y^{\dag}$,
so that both symmetric Poincar\'{e} triads are over the same group ring.  This will be assumed for the rest of the present definition without further comment.

The next step is to exhibit a chain equivalence $\nu \colon C^{\dag} \xrightarrow{\sim} C.$
We show this for the models for each chain complex from Definition \ref{defn:basic_chain_complexes}, since any $C,C^{\dag}$ which can occur is itself chain equivalent to these models.  In fact, for the operation of connected sum which we define here, we describe how to add our two symmetric Poincar\'{e} triads $\Y$ and $\Y^{\dag}$ using the models given for $i_{\pm} \colon (C,\varphi_C) \to (D_{\pm},\delta\varphi_{\pm})$ and $i^{\dag}_{\pm} \colon (C^{\dag},\varphi_{C^{\dag}}) \to (D^{\dag}_{\pm},\delta\varphi^{\dag}_{\pm})$ in Definition \ref{defn:basic_chain_complexes}, since there is always an equivalence of symmetric triads mapping to one in which $C,C^{\dag}$ and $D^{\dag}_{\pm}$ have this form, by definition.  Note that, to achieve this with $g_1 = (1,0) = g_1^\dag$, we may have to change the isomorphisms $\xi$ and $\xi^\dag$ as in the proof of Proposition \ref{prop: fundtriaddefinesanelement}.

The chain isomorphism $\nu \colon C^{\dag}_* \to C_*$ is given by:
\[\xymatrix @R+1cm @C+2cm { \bigoplus_2\,\zhddag \ar[r]_{\left(\begin{array}{cc}  g^{\dag}_1-1 & 0 \\ 0 & g^{\dag}_q-1 \end{array}\right)} \ar[d]_{\left(\begin{array}{cc} 1 & 0 \\ 0 & (l_a^{\dag})^{-1}l_a \end{array}\right)} & \bigoplus_2\,\zhddag \ar[d]^{\left(\begin{array}{cc} 1 & 0 \\ 0 & (l_a^{\dag})^{-1}l_a \end{array}\right)} \\
\bigoplus_2\,\zhddag \ar[r]^{\left(\begin{array}{cc} g_1-1 & 0  \\ 0 & g_q-1 \end{array}\right)} & \bigoplus_2\,\zhddag.
}\]
In order to see that these are chain maps we need the relation $g_1^{\dag} = g_1 \in \Z \ltimes H^{\ddag}$ which, since by definition
$g_q = l_a^{-1}g_1l_a$ and $g_q^{\dag} = (l_a^{\dag})^{-1}g_1^{\dag}l_a^{\dag}$ implies that $g_q = l_a^{-1}l_a^{\dag}g_q^{\dag}(l_a^{\dag})^{-1}l_a.$
We can also use this to calculate that $\nu (\varphi^{\dag}\oplus -\varphi^{\dag}) \nu^{*} = \varphi \oplus -\varphi$.  Recall that we also have a chain isomorphism $\varpi \colon D_-^{\dag} = D_- \to D_+$.

We now glue the two symmetric triads together.  The idea is that we are following the geometric addition of knots, where the neighbourhoods of a chosen meridian of each knot get identified.  We have the following diagram:
\[\xymatrix @R+0.4cm @C-0.41cm {
(D_-,0 = \delta\varphi_-) \ar[d]^{f_-} & (C,\varphi \oplus -\varphi=\varphi_C) \ar[l]_-{i_-} \ar[d]_{i_+} \ar[dl]_{\stackrel{g}{\sim}} & (C^{\dag},\varphi^{\dag} \oplus -\varphi^{\dag}=\varphi_{C^{\dag}}) \ar[l]_-{\stackrel{\nu}{\simeq}} \ar[d]^{i_-^{\dag}} \ar[r]^-{i_+^{\dag}} \ar[dr]^{\stackrel{g^{\dag}}{\sim}} & (D_+^{\dag},0=\delta\varphi_+^{\dag}) \ar[d]^{f_+^{\dag}} \\
(Y,\Phi) & (D_+,0=\delta\varphi_+) \ar[l]_-{f_+} & (D_-^{\dag},0=\delta\varphi_-^{\dag}) \ar[r]^-{f_-^{\dag}} \ar[l]_-{\stackrel{\varpi}{\simeq}} & (Y^{\dag},\Phi^{\dag})
}\]
where the central square commutes.  We then use the union construction from \cite[Part~I,~pages~117--9]{Ranicki3} to define $\Y^{\ddag}$:
\[\xymatrix @C+0.5cm {\ar @{} [dr] |{\stackrel{g^{\ddag}}{\sim}}
(C^{\ddag},\varphi_{C^{\ddag}}) \ar[r]^{i^{\ddag}_-} \ar[d]_{i^{\ddag}_+} & (D^{\ddag}_-,\delta\varphi^{\ddag}_-) \ar[d]^{f^{\ddag}_-}\\ (D^{\ddag}_+,\delta\varphi^{\ddag}_+) \ar[r]_{f^{\ddag}_+} & (Y^{\ddag},\Phi^{\ddag});
}\]
\[(C^{\ddag},\varphi_{C^{\ddag}}):=(C^{\dag},\varphi_{C^{\dag}});\;i_+^{\ddag} := i_+^{\dag};\; i_-^{\ddag} := i_-\circ \nu;\]
\[(D_-^{\ddag},\delta\varphi_-^{\ddag}):= (D_-,\delta\varphi_-=0);\;(D_+^{\ddag},\delta\varphi_+^{\ddag}):= (D_+^{\dag},\delta\varphi_+^{\dag}=0);\]
\[(Y^{\ddag},\Phi^{\ddag}) := (\mathscr{C}((-f_+ \circ \varpi,f_-^{\dag})^T \colon D_-^{\dag} \to Y \oplus Y^{\dag}),\Phi \cup_{\delta\varphi_-^{\dag}} \Phi^{\dag}),\]
so that
\[Y^{\ddag}_r:= Y_r \oplus (D_-^{\dag})_{r-1} \oplus Y_r^{\dag};\]
\[d_{Y^{\ddag}} := \left(\begin{array}{ccc} d_Y & (-1)^{r}f_{+}\circ \varpi & 0 \\ 0 & d_{D_-^{\dag}} & 0 \\ 0 & (-1)^{r-1}f^{\dag}_{-} & d_{Y^{\dag}}
\end{array}\right)\colon Y^{\ddag}_r \to Y^{\ddag}_{r-1};\]
\[f^{\ddag}_{-} := \left(
              \begin{array}{ccc}
                f_- &
                0 &
                0
              \end{array}
            \right)^T \colon (D_-^{\ddag})_r = (D_-)_r \to Y^{\ddag}_r=Y_r \oplus (D_-^{\dag})_{r-1} \oplus Y^{\dag}_r;
\]
\[f^{\ddag}_+ = \left(
              \begin{array}{ccc}
                0 &
                0 &
                f^{\dag}_{+}
              \end{array}
            \right)^T \colon (D_+^{\ddag})_r = (D^{\dag}_+)_r \to Y^{\ddag}_r=Y_r \oplus (D_-^{\dag})_{r-1} \oplus Y^{\dag}_r;\]
\[\Phi^{\ddag}_s := (\Phi \cup_{\delta\varphi_-^{\dag}} \Phi^{\dag})_s = \left(
                        \begin{array}{ccc}
                          \Phi_s & 0 & 0 \\
                          0 & 0 & 0 \\
                          0 & 0 & \Phi^{\dag}_s \\
                        \end{array}
                      \right)\colon\]
\begin{multline*}(Y^{\ddag})^{3-r+s} = Y^{3-r+s} \oplus (D_-^{\dag})^{2-r+s} \oplus (Y^{\dag})^{3-r+s} \to Y^{\ddag}_r = Y_r \oplus (D_-^{\dag})_{r-1} \oplus Y^{\dag}_r\;\; (0 \leq s \leq 3);\end{multline*}
\[g^{\ddag}:= \left(
              \begin{array}{ccc}
                g \circ \nu &
                (-1)^{r+1} i_-^{\dag} &
                g^{\dag}
              \end{array}\right)^T \colon C^{\ddag}_r = C_r^{\dag} \to Y_{r+1}^{\ddag} = Y_{r+1} \oplus (D_-^{\dag})_r \oplus Y^{\dag}_{r+1}.\]

The mapping cone is of the chain map $(-f_+ \circ \varpi,f_-^{\dag})^T$, with a minus sign to reflect the geometric fact that when one adds together oriented knots; one must identify the boundaries with opposite orientations coinciding, so that the resulting knot is also oriented.

We therefore have the chain maps $i_{\pm}^{\ddag}$, given by:
\[\xymatrix @R+1cm @C-0.4cm{
D^{\ddag}_- = D_-  & \zhddag \ar[rrrrr]^{\left(\begin{array}{c}g_1-1\end{array}\right)}  &&&&& \zhddag \\
C^{\ddag}=C^{\dag} \ar[u]^{i_-^{\ddag}=i_- \circ \nu} \ar[d]_{i_+^{\ddag} = i_+^{\dag}} & \bigoplus_2\,\zhddag \ar[rrrrr]_{\left(\begin{array}{cc}g^{\dag}_1-1 & 0 \\ 0 & g^{\dag}_q-1 \end{array}\right)} \ar[u]^{\left(\begin{array}{c} 1 \\ (l_a^{\dag})^{-1} \end{array}\right)} \ar[d]_{\left(\begin{array}{c} (l_b^{\dag})^{-1} \\ 1 \end{array}\right)} &&&&& \bigoplus_2\,\zhddag \ar[u]_{\left(\begin{array}{c} 1 \\ (l_a^{\dag})^{-1} \end{array}\right)} \ar[d]^{\left(\begin{array}{c} (l_b^{\dag})^{-1} \\ 1 \end{array}\right)} \\
D_+^{\ddag}=D_+^{\dag} & \zhddag \ar[rrrrr]_{\left(\begin{array}{c}g^{\dag}_q-1\end{array}\right)} &&&&& \zhddag,
}\]
which means we can take $g_1^{\ddag}:= g_1^{\dag}=g_1 \in \Z \ltimes H^{\ddag} = \Z \ltimes (H \oplus H^{\dag}),$
$l_a^{\ddag}:= l_a^{\dag} \in \Z \ltimes H^{\ddag}$ and $l_b^{\ddag}:= l_b^{\dag} \in \Z \ltimes H^{\ddag},$
so that
$g_q^{\ddag} := g_q^{\dag} \in \Z \ltimes H^{\ddag}.$
We have a chain isomorphism $\varpi^{\dag} \colon D_- = D^{\dag}_- \to D_+^{\dag}$.  To construct a chain homotopy $\mu^{\ddag} \colon (0, 0,  f_+^{\dag} \circ \varpi^{\dag})^T \simeq  ( f_- , 0 , 0 )^T$ we first use
$\mu^{\dag} \colon (  0 , 0 , f_+^{\dag}\circ\varpi^{\dag} )^T \simeq (  0 , 0 , f_-^{\dag})^T.$
We then have a chain homotopy given by:
\begin{align*}&(0 ,\Id ,0  )^T \colon (D_-^{\dag})_0 \to Y^{\ddag}_1 = Y_1 \oplus (D_-^{\dag})_0 \oplus Y_1^{\dag} \text{ and}\\
&(0 ,-\Id , 0 )^T \colon (D_-^{\dag})_1 \to Y^{\ddag}_2 = Y_2 \oplus (D_-^{\dag})_1 \oplus Y_2^{\dag},\end{align*}
which shows that
\[( 0 , 0 ,f_-^{\dag})^T \simeq  (  f_+ \circ \varpi , 0 , 0  )^T \colon D_-^{\dag} \to \mathscr{C}((-f_+ \circ \varpi, f_-^{\dag})^T).\]
We finally have $\mu \colon ( f_+ \circ \varpi , 0 , 0  )^T \simeq (  f_- , 0 , 0 )^T.$  Combining these three homotopies yields
\[\mu^{\ddag} \colon ( 0 , 0 , f_+^{\dag} \circ \varpi^{\dag} )^T \simeq (  f_- , 0 , 0 )^T.\]
This completes our description of the symmetric Poincar\'{e} triad
\[\mathcal{Y}^{\ddag} := \mathcal{Y} \, \sharp \, \mathcal{Y}^{\dag}.\]
Finally, easy Mayer-Vietoris arguments show that $f^{\ddag}_{\pm} \colon H_*(D^\ddag_\pm;\Z) \toiso H_*(Y^\ddag;\Z)$ are isomorphisms and that there is a consistency isomorphism \[\xi^{\ddag} \colon  H^{\ddag} \toiso H_1(\Z[\Z] \otimes_{\zhddag} Y^{\ddag}),\]
which shows that the consistency condition is satisfied and defines the third element of the triple
\vspace{-0.2cm}
\[(H^\ddag,\mathcal{Y}^{\ddag},\xi^\ddag) = (H,\mathcal{Y},\xi) \, \sharp \, (H^\dag,\mathcal{Y}^{\dag},\xi^\dag) \in \P.\]
This completes the definition of the addition of two elements of $\P$.
\qed \end{definition}

\begin{proposition}\label{Prop:abelianmonoid}
The sum operation $\sharp$ on $\P$ is abelian, associative and has an identity, namely the triple containing the fundamental symmetric Poincar\'{e} triad of the unknot.  Therefore, $(\P,\sharp)$ is an abelian monoid.

Let ``$Knots$'' denote the abelian monoid of isotopy classes of locally flat oriented knots in $S^3$ under the operation of connected sum.  Then the function $Knots \to \P$ from Proposition \ref{prop: fundtriaddefinesanelement} becomes a monoid homomorphism.
\end{proposition}
\begin{proof}
The reader is referred to \cite[Proposition~6.8]{Powellthesis} for the proof of this proposition, which is too long for the present paper, and is relatively straight--forward.  It is hopefully intuitively plausible, given that our algebraic connected sum so closely mirrors the geometric connected sum, that our addition is associative, commutative, and that algebraic connected sum with the symmetric Poincar\'{e} triad $(\{0\},\Y^U,\Id)$ associated to the unknot produces an equivalent triad.
\end{proof}

\section{Algebraic Concordance}\label{Chapter:algconcordance}

In this section we introduce an algebraic concordance relation on the elements of $\P$ which closely captures the notion of $(1.5)$-solvability, in the sense that the \COT obstructions vanish if a knot is algebraically $(1.5)$-solvable (Definition \ref{defn:2ndorderconcordant}) which in turn holds if a knot is geometrically $(1.5)$-solvable.


We proceed as follows.  Given two triples $(H,\Y,\xi),(H^\dag,\Y^{\dag},\xi^\dag) \in \P$, we formulate an algebraic concordance equivalence relation, modelled on the concordance of knots and corresponding to $\Z$-homology cobordism of manifolds, with the extra control on the fundamental group which is evidently required, given the prominence of the Blanchfield form in \cite{COT} when controlling representations.  We take the quotient of our monoid $\P$ by this relation, and obtain a group $\mathcal{AC}_2 := \P/\sim$.  Our goal for this section is to complete the set up of the following commuting diagram, which has geometry in the left column and algebra in the right column:\vspace{-0.2cm}
\[\xymatrix @C+1cm{
Knots \ar[r] \ar @{->>} [d] & \P \ar @{->>} [d]  \\
\C \ar[r] & \mathcal{AC}_2, 
}\]
where $Knots$ is the monoid of geometric knots under connected sum and $\C$ is the concordance group of knots.
We shall first define our concordance relation, and show that it is an equivalence relation.  We will then define an inverse $-(H,\Y,\xi)$ of a triple $(H,\Y,\xi)$, and show that $(H,\Y,\xi) \,\sharp\, -(H,\Y,\xi) \sim (\{0\},\Y^U,\Id_{\{0\}})$, where $(\{0\},\Y^U,\Id_{\{0\}})$ is the triple of the unknot, so that we obtain a group $\mathcal{AC}_2$.

\begin{proposition}\label{lemma:basicfactconcordance}
Two knots $K$ and $K^{\dag}$ are topologically concordant if and only if the 3-manifold \vspace{-0.2cm}
\[Z := X \cup_{\partial X = S^1 \times S^1} S^1 \times S^1 \times I \cup_{S^1 \times S^1 = \partial X^{\dag}} -X^{\dag}\]
is the boundary of a topological 4-manifold $W$ such that
\begin{description}
\item[(i)] the inclusion $i \colon Z \hookrightarrow W$ restricts to $\Z$-homology equivalences \vspace{-0.1cm}$$H_*(X;\Z) \xrightarrow{\simeq} H_*(W;\Z) \xleftarrow{\simeq} H_*(X^\dag;\Z); \text{ and}$$
\item[(ii)] the fundamental group $\pi_1(W)$ is normally generated by a meridian of (either of) the knots.
\end{description}
\end{proposition}
We omit the proof of this proposition, which is well-known to the experts, and refer the interested reader to \cite[Proposition~8.1]{Powellthesis}

We need to construct the algebraic version of $Z$ from two symmetric Poincar\'{e} triads $\Y$ and $\Y^{\dag}$ so that we can impose conditions on the algebraic 4-dimensional complexes which have it as their boundary.  As part of the definition of a symmetric Poincar\'{e} triad $\Y$ over $\zh$ (Definition \ref{Defn:symmPoincaretriad}),\vspace{-0.2cm}
\[\xymatrix @C+0.5cm {\ar @{} [dr] |{\stackrel{g}{\sim}}
(C,\varphi_C) \ar[r]^{i_-} \ar[d]_{i_+} & (D_-,\delta\varphi_-) \ar[d]^{f_-}\\ (D_+,\delta\varphi_+) \ar[r]^{f_+} & (Y,\Phi),
}\]
we can construct a symmetric Poincar\'{e} pair
\[(\eta \colon E := D_- \cup_{C} D_+ \to Y ,(\Phi, \delta\varphi_- \cup_{\varphi_C} \delta\varphi_+))\]
where \vspace{-0.2cm}
\[\eta = \left(\begin{array}{ccc} f_-\, , & (-1)^{r-1}g \, , & -f_+ \end{array} \right) \colon E_r = (D_-)_r \oplus C_{r-1} \oplus (D_+)_r \to Y_r.\]
In our case of interest, $E$, for the standard models of $C, D_{\pm}$, is given by:
\[\xymatrix{ E_2 \cong \bigoplus_2 \, \zh \xrightarrow{\partial_2} E_1 \cong \bigoplus_4 \, \zh  \xrightarrow{\partial_1} E_0 \cong \bigoplus_2 \, \zh,
}\]
where:\vspace{-0.4cm}
\[\partial_1 = \left(\begin{array}{cc} g_1-1 & 0  \\ 1 & l_a \\ l_a^{-1} & 1 \\ 0 & g_q-1 \end{array} \right) ; \text{ and } \partial_2 = \left(\begin{array}{cccc} -1 & g_1-1 & 0 & -l_a \\ -l_a^{-1} & 0 & g_q-1 & -1 \end{array} \right),\]
with $\phi_0 \colon E^{2-r} \to E_r$:\vspace{-0.2cm}
\[\xymatrix @C+1cm{ E^0 \ar[r]^{\delta_1} \ar[d]^{\phi_0} & E^1 \ar[r]^{\delta_2} \ar[d]^{\phi_0} & E^2 \ar[d]^{\phi_0}\\
E_2 \ar[r]^{\partial_2} & E_1 \ar[r]^{\partial_1} & E_0
}\]
given by:
\[\xymatrix @R+1.5cm @C+1.5cm{ \bigoplus_2\,\zh \ar[r]^{\delta_1} \ar[d]^{\left(\begin{array}{cc} -1 & l_a \\ 0 & 0 \end{array} \right)} & \bigoplus_4\,\zh \ar[r]^{\delta_2} \ar[d]^{\left(\begin{array}{cccc} 0 & g_1 & -l_ag_q & 0 \\ 0 & 0 & 0 & l_a \\ 0 & 0 & 0 & -1 \\ 0 & 0 & 0 & 0 \end{array} \right)} & \bigoplus_2\,\zh \ar[d]^{\left(\begin{array}{cc} 0 & g_1 l_a \\ 0 & -g_q \end{array} \right)}\\
\bigoplus_2\,\zh \ar[r]^{\partial_2} & \bigoplus_4\,\zh \ar[r]^{\partial_1} & \bigoplus_2\,\zh.
}\]
We have replaced $l_b^{-1}$ with $l_a$ here.  Note that the boundary and symmetric structure maps still depend on the group element $l_a$.  The next lemma shows that, over the group ring $\Z[\Z \ltimes (H \oplus H^{\dag})] = \zhddag$, the chain complexes $E,E^{\dag}$ of the boundaries of two different triads $\Y,\Y^{\dag}$ are isomorphic.  It is used to construct the top row of the triad in Definition \ref{defn:2ndorderconcordant}.

\begin{lemma}\label{lemma:torusindependantofl_a}
There is a chain isomorphism of symmetric Poincar\'{e} complexes:
\[\varpi_E \colon \zhddag \otimes_{\zh} E \to \zhddag \otimes_{\zhdag} E^{\dag}, \]
\[\xymatrix @C+2cm{
E_2 \ar[r]^{\partial_2} \ar[d]_{\varpi_E} & E_1 \ar[r]^{\partial_1} \ar[d]_{\varpi_E} & E_0 \ar[d]_{\varpi_E} \\
E^{\dag}_2 \ar[r]^{\partial_2^{\dag}} & E^{\dag}_1 \ar[r]^{\partial_1^{\dag}} & E^{\dag}_0 \\
}\]
omitting $\zhddag \otimes_{\zh}$ and $\zhddag \otimes_{\zhdag}$ from the notation of the diagram, given by: \vspace{-0.2cm}
\[\xymatrix @R+1.5cm @C+2cm{
\bigoplus_2\,\Ups^{\ddag} \ar[rr]^{\left(\begin{array}{cccc} -1 & g_1-1 & 0 & -l_a \\ -l_a^{-1} & 0 & g_q-1 & -1 \end{array} \right)} \ar[d]^{\left(\begin{array}{cc} 1 & 0 \\ 0 & l_a^{-1}l^{\dag}_a \end{array} \right)} && \bigoplus_4\,\Ups^{\ddag} \ar[r]^{\left(\begin{array}{cc} g_1-1 & 0  \\ 1 & l_a \\ l_a^{-1} & 1 \\ 0 & g_q-1 \end{array} \right)} \ar[d]_{\left(\begin{array}{cccc} 1 & 0  & 0 & 0\\ 0 & 1 & 0 & 0 \\ 0 & 0 & l_a^{-1}l^{\dag}_a  & 0 \\ 0 & 0 & 0 & l_a^{-1}l_a^{\dag}\end{array} \right)} & \bigoplus_2\,\Ups^{\ddag} \ar[d]_{\left(\begin{array}{cc} 1 & 0 \\ 0 & l_a^{-1}l^{\dag}_a \end{array} \right)} \\
\bigoplus_2\,\Ups^{\ddag} \ar[rr]_{\left(\begin{array}{cccc} -1 & g^{\dag}_1-1 & 0 & -l^{\dag}_a \\ -(l^{\dag}_a)^{-1} & 0 & g^{\dag}_q-1 & -1 \end{array} \right)} && \bigoplus_4\,\Ups^{\ddag} \ar[r]_{\left(\begin{array}{cc} g^{\dag}_1-1 & 0  \\ 1 & l^{\dag}_a \\ (l^{\dag}_a)^{-1} & 1 \\ 0 & g^{\dag}_q-1 \end{array} \right)} & \bigoplus_2\,\Ups^{\ddag} \\
}\]
where $\Ups^{\ddag} := \zhddag$.
\end{lemma}
\begin{proof}
To see that $\varpi_E$ is a chain map, as usual one needs the identities:
\[l_ag_ql_a^{-1} = g_1 = g_1^{\dag} = l_a^{\dag}g_q^{\dag}(l_a^{\dag})^{-1}.\]
The maps of $\varpi_E$ are isomorphisms, and the reader can calculate that $\varpi_E \phi \varpi_E^* = \phi^{\dag}.$  Note that this proof relies on the fact that $l_al_b=1$ and would require extra control over the longitude if we were not working modulo the second derived subgroup, but instead were only factoring out further up the derived series.
\end{proof}
\begin{definition}\label{defn:2ndorderconcordant}
We say that two triples $(H,\Y,\xi),(H^\dag,\Y^{\dag},\xi^\dag) \in \P$ are \emph{second order algebraically concordant} or \emph{algebraically $(1.5)$-equivalent}, written $\sim$, if there is a $\Z[\Z]$ module $H'$ of type $K$, that is $H'$ satisfies the properties of (a) of Theorem \ref{Thm:Levinemodule}, with a homomorphism \[(j_{\flat},j_{\flat}^{\dag}) \colon H \oplus H^{\dag} \to H'\] which induces homomorphisms
\[\zh \to \zhd
\text{ and } \zhdag \to \zhd,\]
along with a finitely generated projective $\zhd$-module chain complex $V$ with structure maps $\Theta$, the requisite chain maps $j,j^{\dag},\delta$, and chain homotopies $\g,\g^{\dag}$, such that there is a 4-dimensional symmetric Poincar\'{e}\footnote{The top row is a symmetric Poincar\'{e} pair by Lemma \ref{lemma:productcobordism})} triad:
\[\xymatrix @R+0.5cm @C+1cm {\ar @{} [dr] |{\stackrel{(\gamma,\g^{\dag})}{\sim}}
(\zhd \otimes (E,\phi)) \oplus (\zhd \otimes (E^{\dag},-\phi^{\dag})) \ar[r]^-{(\Id, \Id \otimes \varpi_{E^{\dag}})} \ar[d]_{\left( \begin{array}{cc} \Id \otimes \eta & 0 \\ 0 & \Id \otimes \eta^{\dag}\end{array}\right)} & \zhd \otimes (E,0) \ar[d]^{\delta}
\\ (\zhd \otimes (Y,\Phi)) \oplus (\zhd \otimes (Y^{\dag},-\Phi^{\dag})) \ar[r]^-{(j,j^{\dag})} & (V,\Theta),
}\]
which satisfies two homological conditions.  The first is that:
\[\ba{ccl}j \colon H_*(\Z \otimes_{\zhd} (\zhd \otimes_{\zh} Y)) &\xrightarrow{\simeq}& H_*(\Z \otimes_{\zhd} V) \text{ and} \\
j^{\dag} \colon H_*(\Z \otimes_{\zhd} (\zhd \otimes_{\zhdag} Y^{\dag})) &\xrightarrow{\simeq}& H_*(\Z \otimes_{\zhd} V)\ea\]
are isomorphisms, so that
$H_*(\Z \otimes_{\zhd} V) \cong H_*(S^1;\Z).$
The second homological condition is the consistency condition, that there is a consistency isomorphism:
\[\xi' \colon H' \xrightarrow{\simeq} H_1(\Z[\Z] \otimes_{\zhd} V),\]
such that the diagram below commutes:\vspace{-0.2cm}
\[\xymatrix @R+0.5cm @C+1cm{H \oplus H^{\dag} \ar[r]^{(j_{\flat},j_{\flat}^{\dag})} \ar[d]^{\left(\begin{array}{cc} \xi & 0 \\ 0 & \xi^{\dag} \end{array} \right)}_{\cong} & H' \ar[d]^{\xi'}_{\cong} \\
H_1(\Z[\Z] \otimes_{\zh} Y) \oplus H_1(\Z[\Z] \otimes_{\zhdag} Y^{\dag}) \ar[r]^-{\Id_{\Z[\Z]} \otimes (j_*,j^{\dag}_*)} & H_1(\Z[\Z] \otimes_{\zhd} V).
}\]
We say that two knots are second order algebraically concordant if their triples are, and we say that a knot is \emph{second order algebraically slice} or \emph{algebraically $(1.5)$-solvable} if it is second order algebraically concordant to the unknot.
\qed \end{definition}

\begin{remark}
In what follows we frequently omit the tensor products when reproducing versions of the diagram of the triad in Definition \ref{defn:2ndorderconcordant}, taking as understood that all chain complexes are tensored so as to be over $\zhd$ and all homomorphisms act with an identity on the $\zhd$ component of the tensor products.

\end{remark}

\begin{definition}
The quotient of $\P$ by the relation $\sim$ of Definition \ref{defn:2ndorderconcordant} is the \emph{second order algebraic concordance group} $\ac2$.  See Proposition \ref{prop:equivrelation} for the proof that $\sim$ is an equivalence relation and Proposition \ref{prop:inverseswork} for the proof that $\ac2$ is a group.
\qed \end{definition}

\begin{proposition}\label{prop:conc_implies_2ndorderconc}
Two concordant knots $K$ and $K^{\dag}$ are second order algebraically concordant.
\end{proposition}
We postpone the proof of this result: Proposition \ref{prop:conc_implies_2ndorderconc} is a corollary of Theorem \ref{Thm:1.5solvable=>2nd_alg_slice}.  See \cite[Proposition~8.6]{Powellthesis} for a proof of this special case.

\begin{proposition}\label{prop:equivrelation}
The relation $\sim$ of Definition \ref{defn:2ndorderconcordant} is an equivalence relation.
\end{proposition}
\begin{proof}
We begin by showing that $\sim$ is well--defined and reflexive: that $(H,\Y,\xi) \sim (H^\%,\Y^\%,\xi\%),$ where $(H,\Y,\xi)$ and $(H^\%,\Y^\%,\xi\%)$ are equivalent in the sense of Definition \ref{Defn:algebraicsetofchaincomplexes}.  This is the algebraic equivalent of the geometric fact that isotopic knots are concordant.  Suppose that we have an isomorphism $\omega \colon H \to H^\%$, and a chain equivalence of triads
$j \colon \Z[\Z \ltimes H^\%] \otimes_{\zh} \Y \to \Y^\%,$ such that the relevant square commutes, as in Definition \ref{Defn:algebraicsetofchaincomplexes} (see below). To show reflexivity, we take $H':= H^\%$, and take $(j_{\flat},j_{\flat}) = (\omega,\Id) \colon H \oplus H^\% \to H^\%$ and $(V,\Theta):= (Y^\%,0)$.  We tensor all chain complexes with $\zhpc$, which do not already consist of $\zhpc$-modules.  We have, induced by $j$, an equivalence of symmetric Poincar\'{e} pairs:
 \begin{multline*} (j_E,j_Y;\, k) \colon (\Id \otimes \eta \colon \zhpc \otimes_{\zh} E \to \zhpc \otimes_{\zh} Y) \to (\eta^\% \colon E^\% \to Y^\%),\end{multline*}
where $k \colon \eta^\% j_E \sim j_Y \eta$ is a chain homotopy (see \cite[Part~I,~page~140]{Ranicki3}).
We therefore have the symmetric triad:
\[\xymatrix @R+0.5cm @C+1cm {\ar @{} [dr] |{\stackrel{(k,0)}{\sim}}
\zhpc \otimes_{\zh} (E,\phi) \oplus (E^\%,-\phi^\%) \ar[r]^-{(j_E,\Id)} \ar[d]_{\left( \begin{array}{cc} \Id \otimes \eta & 0 \\ 0 & \eta^\% \end{array} \right)} & (E^\%,0) \ar[d]^{\eta^\%} \\ (Y,\Phi) \oplus (Y^\%,-\Phi^\%) \ar[r]^-{(j_Y,\Id)} & (Y^\%,0).
}\]
The proof of Lemma \ref{lemma:productcobordism} shows that this is a symmetric Poincar\'{e} triad.  Applying the chain isomorphism $\varpi_{E^\%} \colon E^\% \toiso \zhpc \otimes_{\zh} E$ to the top right corner produces the triad:
\[\xymatrix @R+0.5cm @C+1cm {\ar @{} [dr] |{\stackrel{(k,0)}{\sim}}
\zhpc \otimes_{\zh} (E,\phi) \oplus (E^\%,-\phi^\%) \ar[r]^-{(\varpi_{E^\%}\circ j_E,\varpi_{E^\%})} \ar[d]_{\left( \begin{array}{cc} \Id \otimes \eta & 0 \\ 0 & \eta^\% \end{array} \right)} & (\zhpc \otimes_{\zh} E,0) \ar[d]^{\eta^\% \circ(\varpi_{E^\%})^{-1}} \\ (Y,\Phi) \oplus (Y^\%,-\Phi^\%) \ar[r]^-{(j_Y,\Id)} & (Y^\%,0),
}\]
as required.  The homological conditions are satisfied since the maps $j,j^{\dag}$ from Definition \ref{defn:2ndorderconcordant} are chain equivalences and the chain complex $V=Y^\%$.  The consistency condition is satisfied since the commutativity of the square \[\xymatrix @C+1cm{H \ar[r]_-{\xi}^-{\cong} \ar[d]_{\omega}^-{\cong} & H_1(\Z[\Z] \otimes_{\zh} Y) \ar[d]_-{j_*}^-{\cong} \\
H^\% \ar[r]_-{\xi^\%}^-{\cong} & H_1(\Z[\Z] \otimes_{\zhpc} Y^\%), }\]which shows that $(H,\Y,\xi)$ and $(H^\%,\Y^\%,\xi\%)$ are equivalent in the sense of Definition \ref{Defn:algebraicsetofchaincomplexes}, extends to show that the square
\[\xymatrix @R+0.5cm @C+1cm{H \oplus H^{\%} \ar[r]^{(\omega,\Id)} \ar[d]^{\left(\begin{array}{cc} \xi & 0 \\ 0 & \xi^{\%} \end{array} \right)} & H^\% \ar[d]^{\xi^\%} \\
H_1(\Z[\Z] \otimes_{\zh} Y) \oplus H_1(\Z[\Z] \otimes_{\zh} Y^{\%}) \ar[r]^-{(j_*,\Id_*)} & H_1(\Z[\Z] \otimes_{\zhd} Y^\%)
}\]
is also commutative.  Therefore Definition \ref{defn:2ndorderconcordant} is satisfied, so $\sim$ is indeed a reflexive relation.  It is easy to see that $\sim$ is symmetric; we leave the straight-forward check to the reader.

To show transitivity, suppose that $(H,\Y,\xi) \sim (H^\dag,\Y^{\dag},\xi^\dag)$ using $(j_\flat,j_\flat^\dag) \colon H \oplus H^\dag \to H'$, and also that $(H^\dag,\Y^{\dag},\xi^\dag) \sim (H^\ddag,\Y^{\ddag},\xi^\ddag)$, using $(\ol{j_{\flat}^\dag},\ol{j_{\flat}^{\ddag}}) \colon H^{\dag} \oplus H^{\ddag} \to \ol{H'},$ so that there is a diagram of $\Z[\Z \ltimes \ol{H'}]$-module chain complexes:
\[\xymatrix @R+0.5cm @C+1.5cm {\ar @{} [dr] |{\stackrel{(\ol{\gamma^{\dag}},\ol{\g^{\ddag}})}{\sim}}
(E^{\dag},\phi^{\dag}) \oplus (E^{\ddag},-\phi^{\ddag}) \ar[r]^-{(\Id, \wt{\varpi}_{E^{\ddag}})} \ar[d]_{\left( \begin{array}{cc} \eta^{\dag} & 0 \\ 0 & \eta^{\ddag}\end{array}\right)} & (E^{\dag},0) \ar[d]^{\ol{\delta}}
\\ (Y^{\dag},\Phi^{\dag}) \oplus (Y^{\ddag},-\Phi^{\ddag}) \ar[r]^-{(\ol{j^{\dag}},\ol{j^{\ddag}})} & (\ol{V},\ol{\Theta}).
}\]
In this proof the bar is a notational device and has nothing to do with involutions.  To show that $(H,\Y,\xi) \sim (H^\ddag,\Y^{\ddag},\xi^\ddag)$, first we must define a $\Z[\Z]$-module $\ol{\ol{H'}}$ so that we can tensor everything with $\Z[\Z \ltimes \ol{\ol{H'}}]$.  We will glue the symmetric Poincar\'{e} triads together to show transitivity; first we must glue together the $\Z[\Z]$-modules.  Define:
\[(j_{\flat},\ol{j_{\flat}^{\ddag}}) \colon H \oplus H^{\ddag} \to \ol{\ol{H'}} := \coker((j_{\flat}^{\dag},-\ol{j_{\flat}^{\dag}}) \colon H^{\dag} \to H' \oplus \ol{H}').\]
Now, use the inclusions followed by the quotient maps: $$H' \to H' \oplus \ol{H'} \to \ol{\ol{H'}} \text{ and } \ol{H'} \to H' \oplus \ol{H'} \to \ol{\ol{H'}}$$ to take the tensor product of both the 4-dimensional symmetric Poincar\'{e} triads which show that $(H,\Y,\xi) \sim (H^\dag,\Y^{\dag},\xi^\dag)$, and that $(H,\Y^{\dag},\xi^\dag) \sim (H^\ddag,\Y^{\ddag},\xi^\ddag)$, with $\Z[\Z \ltimes \ol{\ol{H'}}]$, so that both contain chain complexes of modules over the same ring  $\Z[\Z \ltimes \ol{\ol{H'}}]$.  Then algebraically gluing the triads together, as in \cite[pages~117--9]{Ranicki2}, we obtain the 4-dimensional symmetric Poincar\'{e} triad:
\[\xymatrix @R+1cm @C+2cm {\ar @{} [dr] |{\ol{\ol{\g}} = \left(\begin{array}{cc} \g & 0 \\  0 & 0 \\  0 & \ol{\g^{\ddag}} \\ \end{array} \right)}
(E,\phi) \oplus (E^{\ddag},-\phi^{\ddag}) \ar[r]^-{\left(\begin{array}{cc} \Id & 0 \\  0 & 0 \\  0 & \wt{\varpi}_{E^{\ddag}} \\ \end{array} \right)} \ar[d]_{\left( \begin{array}{cc} \eta & 0 \\ 0 & \eta^{\ddag}\end{array}\right)} & (\ol{\ol{E}},- 0 \cup_{\phi^{\dag}} 0) \ar[d]^{\ol{\ol{\delta}} = \left( \begin{array}{ccc} \delta & (-1)^{r-1}\g^{\dag} & 0 \\ 0 & \eta^{\dag} & 0 \\ 0 & (-1)^{r-1}\ol{\g^{\dag}} & \ol{\delta} \\ \end{array} \right)}
\\ (Y,\Phi) \oplus (Y^{\ddag},-\Phi^{\ddag}) \ar[r]_-{\left(\begin{array}{cc} j & 0 \\  0 & 0 \\  0 & \ol{j^{\ddag}} \\ \end{array} \right)} & (\ol{\ol{V}},\ol{\ol{\Theta}}).
}\]
where:\vspace{-0.2cm}
\[\ol{\ol{E}} := \mathscr{C}((\varpi_{E^{\dag}},\Id)^T \colon E^{\dag} \to E \oplus E^{\dag});\]
\[\ol{\ol{V}} := \mathscr{C}((j^{\dag},\ol{j^{\dag}})^T \colon Y^{\dag} \to V \oplus \ol{V}); \text{ and } \ol{\ol{\Theta}} := \Theta \cup_{\Phi^{\dag}} \ol{\Theta}.\]
We need to show that this is equivalent to a triad where the top right term is $(E,0)$.  First, to see that $E \simeq \ol{\ol{E}}$, the chain complex of $\ol{\ol{E}}$ is given by:
\[\xymatrix {E_2^{\dag} \ar[rrrr]^-{\left(
                               \begin{array}{ccc}
                                 \varpi_{E^{\dag}}, &
                                 \partial_{E^{\dag}},&
                                 \Id
                               \end{array}
                             \right)^T} & & & & E_2 \oplus E_1^{\dag} \oplus E_2^{\dag} \ar[r]^-{\partial^{\ol{\ol{E}}}_2}& E_1 \oplus E_0^{\dag} \oplus E_1^{\dag} \ar[r]^-{\partial^{\ol{\ol{E}}}_1} & E_0 \oplus E_0^{\dag},}\]
where:\vspace{-0.2cm}
\[\partial^{\ol{\ol{E}}}_2 = \left(
                               \begin{array}{ccc}
                                 \partial_{E} & -\varpi_{E^{\dag}} & 0 \\
                                 0 & \partial_{E^{\dag}} & 0 \\
                                 0 & -\Id & \partial_{E^{\dag}} \\
                               \end{array}
                             \right); \text{ and } \partial^{\ol{\ol{E}}}_1 = \left(
                               \begin{array}{ccc}
                                 \partial_E & \varpi_{E^{\dag}} & 0 \\
                                 0 & \Id & \partial_{E^{\dag}} \\
                               \end{array}
                             \right).\]
It is easy to see that the chain map:
\[\nu' := \left(
     \begin{array}{ccc}
       \Id\, , & 0\, , & -\varpi_{E^{\dag}} \\
     \end{array}
   \right) \colon E_r \oplus E_{r-1}^{\dag} \oplus E_{r}^{\dag} \to E_r,\]
is a chain equivalence, with chain homotopy inverse:
\[ \nu'^{-1} := \left( \begin{array}{ccc} \Id\, , & 0\, , & 0 \end{array}\right)^T \colon E_r \to E_r \oplus E_{r-1}^{\dag} \oplus E_r^{\dag}.\]

We therefore have the diagram:
\[\xymatrix @R+1cm @C+2cm {
 & & (E,0) \ar@/^5pc/[ddl]^{\ol{\ol{\delta}}\circ \nu'^{-1}}\\
(E,\phi) \oplus (E^{\ddag},-\phi^{\ddag}) \ar@/^5pc/[rru]^>>>>>>>>>>>>>>>{\left(\Id,- \varpi_{E^{\dag}} \circ \wt{\varpi}_{E^{\ddag}} \right)
} \ar[r]^-{\left(\begin{array}{cc} \Id & 0 \\  0 & 0 \\  0 & \wt{\varpi}_{E^{\ddag}} \\ \end{array} \right)} \ar @{} [dr] |{\stackrel{\ol{\ol{\g}}}{\sim}} \ar[d] & (\ol{\ol{E}},- 0 \cup_{\phi^{\dag}} 0) \ar[ru]^{\simeq}_{\nu'} \ar[d]^{\ol{\ol{\delta}}} \ar @{} [dr] |{\stackrel{k'}{\sim}} &
\\ (Y,\Phi) \oplus (Y^{\ddag},-\Phi^{\ddag}) \ar[r] & (\ol{\ol{V}},\ol{\ol{\Theta}}). &
}\]
The top triangle commutes, while the bottom triangle commutes up to a chain homotopy $k'$: $k'$ gets composed with $\ol{\ol{\gamma}}$ to make the new triad.  Furthermore, $\nu' (- 0 \cup_{\phi^{\dag}} 0) \nu'^* = 0,$ so that we indeed have an equivalent triad with the top right as $(E,0)$.

To complete the proof, we need to see that the consistency condition holds.  The following commutative diagram has exact columns, the right hand column being part of the Mayer-Vietoris sequence.  The horizontal maps are given by consistency isomorphisms.  Recall that $\ol{\ol{H'}} := \coker((j_{\flat}^{\dag},-\ol{j_{\flat}^{\dag}}) \colon H^{\dag} \to H' \oplus \ol{H}').$  All homology groups in this diagram are taken with $\Z[\Z]$-coefficients.
\[\xymatrix @R+0.1cm @C+0.6cm {H^\dag \ar[d] \ar[rrr]^-{\xi^\dag}_-{\cong} & & & H_1(Y^\dag) \ar[d] \\
 H' \oplus \ol{H'} \ar[dd] \ar[rrr]_-{\cong}^-{\left(\ba{cc}\xi' & 0 \\ 0 & \ol{\xi'} \ea\right)} & & & H_1(V) \oplus H_1(\ol{V}) \ar[dd]\\
 & H \oplus H^\ddag \ar[ul] \ar@{-->}[dl] \ar[r]_-{\cong}^-{\left(\ba{cc}\xi & 0 \\ 0 & \xi^\ddag \ea\right)} & H_1(Y) \oplus H_1(Y^\ddag) \ar[ur] \ar@{-->}[dr] & \\
\ol{\ol{H'}} \ar[d] \ar@{-->}[rrr]_-{\cong}^-{\ol{\ol{\xi'}}} & & & H_1(\ol{\ol{V}}) \ar[d]\\
   0 & & & 0}\]
The diagonal dotted arrows are induced by the diagram, so as to make it commute.  The horizontal dotted arrow $\ol{\ol{H'}} \to H_1(\Z[\Z] \otimes_{\Z[\Z \ltimes \ol{\ol{H'}}]}\ol{\ol{V}})$ is induced by a diagram chase: the quotient map $H' \oplus \ol{H'} \to \ol{\ol{H'}}$ is surjective.  We obtain a well--defined isomorphism
\[\ol{\ol{\xi'}} \colon \ol{\ol{H'}} \toiso H_1(\Z[\Z] \otimes_{\Z[\Z \ltimes \ol{\ol{H'}}]}\ol{\ol{V}}).\]
The commutativity of the diagram above implies the commutativity of the induced diagram:
\[\xymatrix @R+0.5cm @C+1cm{H \oplus H^{\ddag} \ar[r] \ar[d]^-{\left(\begin{array}{cc} \xi & 0 \\ 0 & \xi^{\ddag} \end{array} \right)} & \ol{\ol{H'}} \ar[d]^-{\ol{\ol{\xi'}}} \\
H_1(\Z[\Z] \otimes_{\zh} Y) \oplus H_1(\Z[\Z] \otimes_{\zh} Y^{\ddag}) \ar[r] & H_1(\Z[\Z] \otimes_{\zhd} \ol{\ol{V}}).
}\]
This completes the proof that $\sim$ is transitive and therefore completes the proof that $\sim$ is an equivalence relation.
\end{proof}

\begin{definition}\label{defn:inverses}
Given an element $(H,\Y,\xi) \in \P$, choose a representative with the boundary given by the model chain complexes.
\[\xymatrix @C+0.5cm {\ar @{} [dr] |{\stackrel{g}{\sim}}
(C,\varphi \oplus -\varphi) \ar[r]^{i_-} \ar[d]_{i_+} & (D_-,0) \ar[d]^{f_-}\\ (D_+,0) \ar[r]^{f_+} & (Y,\Phi).
}\]
The following is also a symmetric Poincar\'{e} triad:
\[\xymatrix @C+0.5cm {\ar @{} [dr] |{\stackrel{g}{\sim}}
(C,-\varphi \oplus \varphi) \ar[r]^{i_-} \ar[d]_{i_+} & (D_-,0) \ar[d]^{f_-}\\ (D_+,0) \ar[r]^{f_+} & (Y,-\Phi),
}\]
which define as the element $-\Y$.  This is the algebraic equivalent of changing the orientation of the ambient space and of the knot simultaneously. The chain equivalence:
\[\varsigma = \left(
    \begin{array}{cc}
      0 & l_a \\
      l_a^{-1} & 0 \\
    \end{array}
  \right)
 \colon C_i \to C_i\]
for $i=0,1$ sends $\varphi \oplus -\varphi$ to $-\varphi \oplus \varphi$ and satisfies $i_{\pm} \circ \varsigma = i_{\pm}$.  We can therefore define the inverse $-(H,\Y,\xi) \in \P$ to be the triple $(H,-\Y,\xi)$, where $-\Y$ is the symmetric Poincar\'{e} triad:
\[\xymatrix @C+0.5cm {\ar @{} [dr] |{\stackrel{g\circ \varsigma}{\sim}}
(C,\varphi \oplus -\varphi) \ar[r]^{i_-} \ar[d]_{i_+} & (D_-,0) \ar[d]^{f_-}\\ (D_+,0) \ar[r]^{f_+} & (Y,-\Phi).
}\]
Summarising, to form an inverse we replace $g$ with $g \circ \varsigma$, and change the sign on the symmetric structures everywhere but on $C$ in the top left of the triad.
\qed \end{definition}

\begin{figure}[h]
    \begin{center}
 {\psfrag{A}{$(Y,\Phi)$}
 \psfrag{B}{$(D_+,0)$}
 \psfrag{C}{$(D_-,0)$}
 \psfrag{D}{$(V,\Theta)$}
 \psfrag{E}{$(D_-,0)$}
 \psfrag{F}{$(D_+,0)$}
 \psfrag{G}{$(D^{\dag}_-,0)$}
 \psfrag{H}{$(Y^{\dag},-\Phi^{\dag})$}
 \psfrag{J}{$(D_+^{\dag},0)$}
 \psfrag{K}{$(C,\varphi \oplus - \varphi)$}
 \psfrag{L}{$(C^{\dag},\varphi^{\dag} \oplus -\varphi^{\dag})$}
\includegraphics[width=8cm]{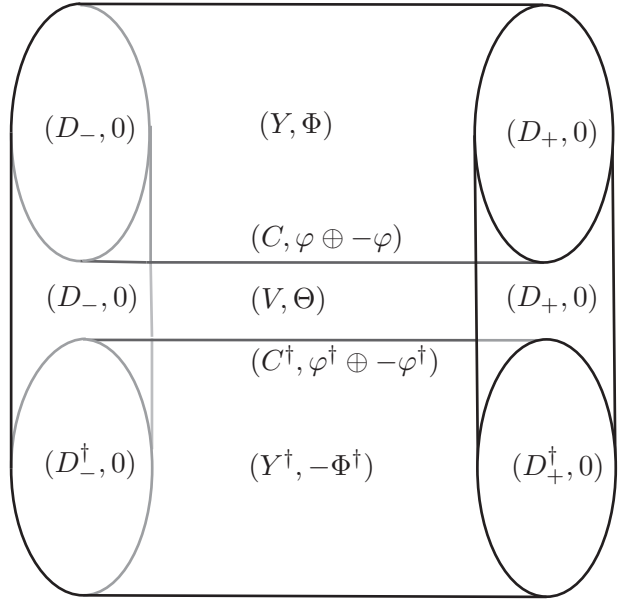}
 }
 \caption{The cobordism which shows that $\mathcal{Y} \sim \mathcal{Y}^{\dag}$.}
 \label{Fig:existenceofinverse2}
 \end{center}
\end{figure}

\begin{figure}[h]
    \begin{center}
 {\psfrag{A}{$(Y,\Phi)$}
 \psfrag{B}{$D_-^{\dag}$}
 \psfrag{C}{$D_-$}
 \psfrag{D}{$(V,\Theta)$}
 \psfrag{E}{$D_-$}
 \psfrag{F}{$D_-^{\dag}$}
 \psfrag{G}{$D_-$}
 \psfrag{H}{$D_- = Y^U$}
 \psfrag{J}{$D_-$}
 \psfrag{K}{$C$}
 \psfrag{L}{$C$}
 \psfrag{M}{$D_+$}
 \psfrag{N}{$D_+^{\dag}$}
 \psfrag{P}{$D_+^{\dag}$}
 \psfrag{Q}{$(Y^{\dag},-\Phi^{\dag})$}
 \psfrag{S}{$C^{\dag}$}
 \includegraphics[width=9cm]{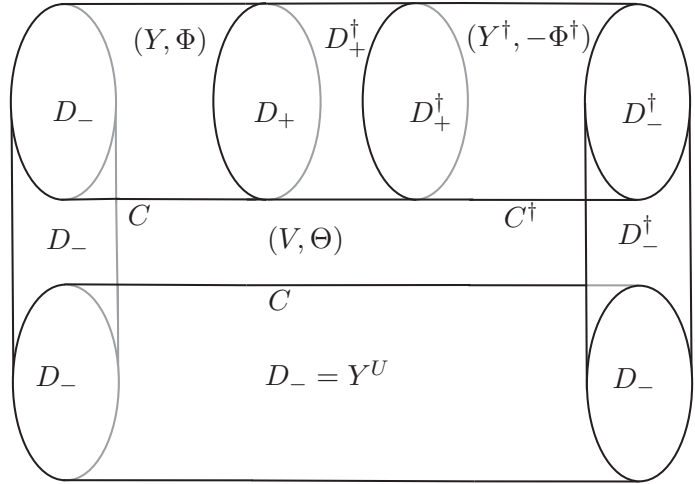}
 }
 \caption{The cobordism which shows that $\mathcal{Y} \, \sharp \, - \mathcal{Y}^{\dag} \sim \mathcal{Y}^U$.}
 \label{Fig:existenceofinverse3}
 \end{center}
\end{figure}

\begin{proposition}\label{prop:inverseswork}
Recall that $(\{0\},\Y^U,\Id_{\{0\}})$ is the triple of the unknot, and let $(H,\Y,\xi)$ and $(H^\dag,\Y^{\dag},\xi^\dag)$ be two triples in $\P$.  Then $$(H,\Y,\xi) \, \sharp \, -(H^\dag,\Y^{\dag},\xi^\dag) \sim (\{0\},\Y^U,\Id_{\{0\}})$$ if and only if $(H,\Y,\xi) \sim (H^\dag,\Y^{\dag},\xi^\dag).$
\end{proposition}

\begin{proof}
We omit the proof of this result, and instead refer the reader to \cite[Proposition~8.10]{Powellthesis}.  It is hopefully intuitively plausible, given that two knots $K, K^\dag$ are concordant if and only if $K \,\sharp\, -K^\dag$is slice.  See Figures \ref{Fig:existenceofinverse2} and \ref{Fig:existenceofinverse3}.
\end{proof}
Proposition \ref{prop:inverseswork} completes the proof that we have defined an abelian group.

\section{$(1.5)$-Solvable Knots are Algebraically $(1.5)$-Solvable}\label{Section:1.5solvable=>2nd_alg_slice}

This section contains the proof of the following theorem.

\begin{theorem}\label{Thm:1.5solvable=>2nd_alg_slice}
A $(1.5)$-solvable knot is algebraically $(1.5)$-solvable.
\end{theorem}

We begin by recalling the definition of $(n)$-solubility.  We denote the zero framed surgery on a knot $K$ by $M_K$.

\begin{definition}\label{Defn:COTnsolvable} \cite[Definition~1.2]{COT}
A \emph{Lagrangian} of a symmetric form $\lambda \colon P \times P \to R$ on a free $R$-module $P$ is a submodule $L \subseteq P$ of half-rank on which $\lambda$ vanishes.

For $n \in \mathbb{N}_0 := \mathbb{N} \cup \{0\}$, let $\lambda_n$ be the equivariant intersection pairing, and $\mu_n$ the self-intersection form, on the middle dimensional homology $H_2(W;\Z[\pi_1(W)/\pi_1(W)^{(n)}])$  of the covering space $W^{(n)}$ corresponding to the subgroup $\pi_1(W)^{(n)} \leq \pi_1(W)$:
\begin{multline*}\lambda_n \colon H_2(W;\Z[\pi_1(W)/\pi_1(W)^{(n)}]) \times H_2(W;\Z[\pi_1(W)/\pi_1(W)^{(n)}]) \to \Z[\pi_1(W)/\pi_1(W)^{(n)}].\end{multline*}
An $(n)$-\emph{Lagrangian} is a submodule of $H_2(W;\Z[\pi_1(W)/\pi_1(W)^{(n)}])$, on which $\lambda_n$ and $\mu_n$ vanish, which maps via the covering map onto a Lagrangian of $\lambda_0$.

We say that a knot $K$ is \emph{$(n)$-solvable} if the zero framed surgery $M_K$ bounds a topological spin 4-manifold $W$ such that the inclusion induces an isomorphism on first homology and such that $W$ admits two dual \emph{$(n)$-Lagrangians}.  In this setting, dual means that $\lambda_n$ pairs the two Lagrangians together non-singularly and their images freely generate $H_2(W;\Z)$.

We say that $K$ is \emph{$(n.5)$-solvable} if in addition one of the $(n)$-Lagrangians is the image of an $(n+1)$-Lagrangian.
\qed \end{definition}

An $(n)$-solution $W$ is an approximation to a slice disc complement; if $K$ is slice then it is $(n)$-solvable for all $n$, so if we can obstruct a knot from being $(n)$- or $(n.5)$-solvable then in particular we show that it is not slice.

It is an interesting question (Question \ref{question:is_it_an_iso}) to wonder whether the converse of Theorem \ref{Thm:1.5solvable=>2nd_alg_slice} holds.  At present, $\ac2$ does not capture the subtle quadratic refinement information, encoded in $\mu_2$, which is part of Definition \ref{Defn:COTnsolvable}.  Until the construction of $\ac2$ is improved so as to take the self intersection form into account it is unlikely that the converse to Theorem \ref{Thm:1.5solvable=>2nd_alg_slice} should hold.  Perhaps rationally there is more hope.

The idea of the proof of Theorem \ref{Thm:1.5solvable=>2nd_alg_slice} is as follows.  The Cappell-Shaneson technique \cite{CappellShaneson74} looks for obstructions to being able to perform surgery on a 4-manifold $W$ whose boundary is the zero framed surgery $M_K$, in order to excise the second $\Z$-homology and create a homotopy slice disc exterior.  The main obstruction to being able to do this surgery is the middle-dimensional intersection form of $W$, as in the \COT definition of $(n)$-solubility.  However, even if the Witt class of the intersection form vanishes, with coefficients in $\Z[\pi_1(W)/\pi_1(W)^{(2)}]$ for testing $(1.5)$-solubility, this does not imply that we have a half basis of the second homology  $H_2(W;\Z[\pi_1(W)/\pi_1(W)^{(2)}])$ representable by disjointly embedded spheres, as our data for surgery: typically the homology classes will be represented as embedded surfaces of non-zero genus, whose fundamental group maps into $\pi_1(W)^{(2)}$.  We cannot do surgery on such surfaces.

However, the conditions on a $(1.5)$-solution are, as we shall see, precisely the conditions required for being able to perform \emph{algebraic surgery on the chain complex} of the $(1.5)$-solution.  The $(1.5)$-level algebra cannot see the differences between $(2)$-surfaces and spheres, so that we can obtain an \emph{algebraic $(1.5)$-solution} $V$.

In particular, the existence of the dual $(1)$-Lagrangian allows us to perform algebraic surgery \emph{without changing the first homology} at the $\Z[\Z]$ level, therefore maintaining the consistency condition.  When performing geometric surgery on a 4-manifold $W$ along a 2-sphere, we remove $S^2 \times D^2$ and glue in $D^3 \times S^1$.  Removing the thickening $D^2$ potentially creates new elements of $H_1(W;\Z[\Z])$.  However, the existence of a dual surface to the $S^2$ which we remove guarantees that the boundary $S^1$ of the thickening $D^2$ bounds a surface on the other side, so that we do not create extra 1-homology.  This phenomenon will also be seen when performing algebraic surgery; as ever, the degree of verisimilitude provided by the chain level approach is somewhat remarkable.


\begin{definition}
An $n$-dimensional symmetric complex $(C,\varphi \in Q^n(C,\eps))$ is \emph{connected} if
$H_0(\varphi_0 \colon C^{n-*} \to C_*) = 0.$
An $n$-dimensional symmetric pair
$(f \colon C \to D,(\delta\varphi,\varphi) \in Q^n(f,\eps))$
is \emph{connected} if
$H_0( ( \delta\varphi_0 , \varphi_0 f^* )^T \colon D^{n-*} \to \mathscr{C}(f)_*) = 0.$
\qed \end{definition}

\begin{definition}\label{Defn:algebraicsurgerymatrices}
\cite[Part~I,~page~145]{Ranicki3} Given a connected $n$-dimensional symmetric chain complex over a ring $A$, $(C,\varphi \in Q^n(C,\eps))$, an \emph{algebraic surgery} on $(C,\varphi)$ takes as data a connected $(n+1)$-dimensional symmetric pair:
\[(f \colon C \to D,(\delta\varphi,\varphi) \in Q^{n+1}(f,\eps)).\]
The output, or effect, of the algebraic surgery is the connected $n$-dimensional symmetric chain complex over $A$, $(C',\varphi' \in Q^n(C',\eps))$, given by:
\begin{eqnarray*}
d_{C'} &= &\left(
             \begin{array}{ccc}
               d_C & 0 & (-1)^{n+1}\varphi_0f^* \\
               (-1)^rf & d_D & (-1)^r\delta\varphi_0 \\
               0 & 0 & (-1)^r\delta_D \\
             \end{array}
           \right)
\colon \\& &C'_r = C_r \oplus D_{r+1} \oplus D^{n-r+1} \to C'_{r-1} = C_{r-1} \oplus D_r \oplus D^{n-r+2},\end{eqnarray*}

with the symmetric structure given by:
\begin{eqnarray*}
\varphi'_0 &=& \left(
                 \begin{array}{ccc}
                   \varphi_0 & 0 & 0 \\
                   (-1)^{n-r}fT_{\eps}\varphi_1 & (-1)^{n-r}T_{\eps}\delta\varphi_1 & (-1)^{r(n-r)}\eps \\
                   0 & 1 & 0 \\
                 \end{array}
               \right)
               \colon\\& & C'^{n-r} = C^{n-r} \oplus D^{n-r+1} \oplus D_{r+1} \to C'_r = C_r \oplus D_{r+1} \oplus D^{n-r+1};\text{ and}\\
\varphi'_s &=&  \left(
                 \begin{array}{ccc}
                   \varphi_s & 0 & 0 \\
                   (-1)^{n-r}fT_{\eps}\varphi_{s+1} & (-1)^{n-r}T_{\eps}\delta\varphi_{s+1} & 0 \\
                   0 & 0 & 0 \\
                 \end{array}
               \right)
               \colon\\& & C'^{n-r+s} = C^{n-r+s} \oplus D^{n-r+s+1} \oplus D_{r-s+1} \to C'_r = C_r \oplus D_{r+1} \oplus D^{n-r+1}
\end{eqnarray*}
for $s \geq 1$.\qed \end{definition}

The reader can check that $d_{C'}^2 = 0$ and that $\{\varphi'_s\} \in Q^n(C',\eps)$.
Algebraic surgery on a chain complex which is symmetric but not Poincar\'{e} preserves the homotopy type of the boundary: see \cite[Part~I,~Proposition~4.1~(i)]{Ranicki3} for the proof.

\begin{definition}
The \emph{suspension morphism} $S$ on chain complexes raises the degree: $(SC)_r = C_{r-1};\; d_{SC} = d_C$.
\qed \end{definition}

\begin{proof}[Proof of Theorem \ref{Thm:1.5solvable=>2nd_alg_slice}]
We need to show that the triple $(H^K,\Y^K,\xi^K)$ of a $(1.5)$-solvable knot $K$, with a $(1.5)$-solution $W$, is equivalent to the identity element of $\mathcal{AC}_2$, which is represented by the triple $(\{0\},\Y^U,\Id_{\{0\}})$ corresponding to the unknot.

The chain complex $N_K := E^K \cup_{E^K \oplus E^U} Y^K \oplus Y^U$ is chain equivalent to the chain complex $C_*(M_K;\Z[\Z \ltimes H_1(M_K;\Z[\Z])])$ of the second derived cover of the zero framed surgery on $K$.  Our first attempt for a chain complex which fits into a 4-dimensional symmetric Poincar\'{e} triad as required in Definition \ref{defn:2ndorderconcordant} is the chain complex of the second derived cover of the $(1.5)$ solution $W$ \vspace{-0.2cm}
\[(V',\Theta') := (C_*(W;\Z[\Z \ltimes H_1(W;\Z[\Z])]),\backslash\Delta([W,M_K])),\] so that  \vspace{-0.2cm} \[H' := \pi_1(W)^{(1)}/\pi_1(W)^{(2)} \toiso H_1(W;\Z[\Z]),\] and we have the triad: \vspace{-0.3cm}
\[\xymatrix @R+0.5cm @C+1cm {\ar @{} [dr] |{\stackrel{(\gamma^K,\gamma^{U})}{\sim}}
(E^K,\phi^K) \oplus (E^U,-\phi^U) \ar[r]^-{(\Id, \Id \otimes \varpi_{E^{K}})} \ar[d]_{\left( \begin{array}{cc} \eta^K & 0 \\ 0 & \eta^{U}\end{array}\right)} & (E^K,0) \ar[d]^{\delta}
\\ (Y^K,\Phi^K) \oplus  (Y^{U},-\Phi^{U}) \ar[r]^-{(j^K,j^{U})} & (V',\Theta'),
}\]
with a geometrically defined consistency isomorphism $$H' \toiso H_1(W;\Z[\Z]) = H_1(\Z[\Z] \otimes_{\zhd} V).$$

The problem is that $H_2(W;\Z)$ is typically non-zero: if it were zero, we would have our topological concordance exterior and in particular $K$ would be second order algebraically slice.  We therefore need, as indicated above, to perform algebraic surgery on $V'$ to transform it into a $\Z$-homology circle.  We form the algebraic Thom complex (Definition \ref{Defn:algThomcomplexandthickening}):
\begin{multline*} C_*(W,M_K;\Z[\Z \ltimes H']) \simeq \ol{V} := \mathscr{C}((\delta,(-1)^{r-1}\gamma^K, (-1)^{r-1}\gamma^U, -j^K,-j^U) \colon \\ (N_K)_r = E^K_r \oplus E_{r-1}^K \oplus E_{r-1}^U \oplus Y^K_r \oplus Y^U_r \to V'_r),\end{multline*}
with symmetric structure $\ol{\Theta}:= \Theta' / (0 \cup_{\phi^K \oplus -\phi^U} \Phi^K \oplus -\Phi^U)$.  In this section the bar is again a notational device and has nothing to do with involutions.

This gives us the input for surgery, since the input for algebraic surgery must be a symmetric chain complex.  Next, we need the data for surgery.

As in the proof of \cite[Proposition~4.3]{COT}, any compact topological 4-manifold has the homotopy type of a finite simplicial complex: see \cite[Annex~B~III,~page~301]{KS}.  In particular this means that $H_2(W;\Z)$ is finitely generated.  We therefore have homology classes $l'_1,\dots,l'_k \in H_2(W;\Z[\Z \ltimes H'])$ which generate the $(2)$-Lagrangian whose existence is guaranteed by definition of a $(1.5)$-solution $W$.  There are therefore dual cohomology classes $l_1,\dots,l_k \in H^2(W,M_K;\Z[\Z \ltimes H'])$, by Poincar\'{e}-Lefschetz duality.  Taking cochain representatives for these, we have maps $l_i \colon \ol{V}_2 \to \Z[\Z \ltimes H']$.  We then take as our data for algebraic surgery the symmetric pair:
\[(\ol{f} \colon \ol{V} \to B:= S^2(\bigoplus_k \, \Z[\Z \ltimes H']),(0,\ol{\Theta})),\]
where \vspace{-0.2cm}
\[\ol{f} = (l_1,\dots,l_k)^T \colon \ol{V}_2 \to B_2 = \bigoplus_k \, \Z[\Z \ltimes H'].\]
The fact that the $l_i$ are cohomology classes means that $l_i d_{\ol{V}} = 0 $, so that $\ol{f}$ is a chain map.  The requirement that the $l_i'$ generate a submodule of $H_2(W;\Z[\Z \ltimes H']) = H_2(V')$ on which the intersection form vanishes means that the duals $l_i$ generate a submodule of $H^2(\ol{V})$ on which the cup product vanishes.  The cup product of any two $l_i,l_j$ is given by:
\[\Delta_0^*(l_i \otimes l_j)([W,M_K]) = (l_i\otimes l_j)(\Delta_0([W,M_K])) = (l_i \otimes l_j)\ol{\Theta}_0,\] which under the slant isomorphism is  $l_i \ol{\Theta}_0 l_j^*,$ and so we see that each of these composites vanishes.

The only possibility for non-zero symmetric structure in the data for surgery would arise when $s = n-2r-1 = 4-2\cdot 2-1 = -1$, so no such non-zero structure maps exist.  Therefore the condition for our data for surgery to be a symmetric pair is that $\ol{f} \,\ol{\Theta}_0 \ol{f}^* = 0;$
which is the condition that the $k \times k$ matrix with $(i,j)$th entry $l_i \ol{\Theta}_0 l_j^*$, is zero.  This is satisfied as we saw above, since $l_i \ol{\Theta}_0 l_j^* \colon \Z[\Z \ltimes H'] \to \Z[\Z \ltimes H']$ is a module homomorphism given by multiplication by the same group ring element as the evaluation on the relative fundamental class $[W,M_K]$ of the cup product of two cohomology classes dual to the $(2)$-Lagrangian, and so equals the value of $\lambda_2(l'_i,l'_j)$.  This means that we can proceed with the operation of algebraic surgery to form the symmetric chain complex $(V,\Theta)$, which is the effect of algebraic surgery, shown below.  We may assume, since $W$ is a 4-manifold with boundary, that we have a chain complex $V'$ whose non-zero terms are $V'_0, V'_1, V'_2$ and $V'_3$.  The non-zero terms in $\ol{V}$ will therefore be of degree less than or equal to four.

The output of algebraic surgery, which we denote as $(V,\Theta)$ is then given, from Definition \ref{Defn:algebraicsurgerymatrices}, by:
\[\xymatrix @R+2cm @C-0.4cm{ \ol{V}^0 \ar[rr]^-{\left( \begin{array}{c} d^*_{\ol{V}} \\ 0 \end{array}\right)} \ar[d]^{\left(\begin{array}{c} \ol{\Theta}_0 \end{array} \right)} &&
\ol{V}^1 \oplus B^2 \ar[rrrr]^-{\left(\begin{array}{cc} d^*_{\ol{V}} & \ol{f}^* \end{array}\right)} \ar[d]^<<<<<<<<<{\left(\begin{array}{cc} \ol{\Theta}_0 & 0 \\ 0 & 1 \end{array} \right)} &&&&
\ol{V}^2 \ar[rrr]^-{\left( \begin{array}{c} d^*_{\ol{V}} \\ -\ol{f}\,\ol{\Theta}_0^*\end{array} \right)} \ar[d]_>>>>>>>{\left(\begin{array}{c} \ol{\Theta}_0 \end{array} \right)} &&&
\ol{V}^3 \oplus B_2 \ar[rrrr]^-{\left( \begin{array}{cc} d^*_{\ol{V}} & 0 \end{array}\right)} \ar[d]^>>>>>>>>{\left(\begin{array}{cc} \ol{\Theta}_0 & 0 \\ -f T\ol{\Theta}_1 & -1 \end{array} \right)} &&&&
\ol{V}^4 \ar[d]_<<<<<<<<<<{\left(\begin{array}{c} \ol{\Theta}_0 \end{array} \right)} \\
\ol{V}_4 \ar[rr]_-{\left( \begin{array}{c} d_{\ol{V}} \\ 0 \end{array}\right)} &&
\ol{V}_3 \oplus B^2 \ar[rrrr]_-{\left(\begin{array}{cc} d_{\ol{V}} & -\ol{\Theta}_0\, \ol{f}^* \end{array}\right)} &&&&
\ol{V}_2 \ar[rrr]_-{\left( \begin{array}{c} d_{\ol{V}} \\ \ol{f}\end{array} \right)} &&&
\ol{V}_1 \oplus B_2 \ar[rrrr]_-{\left( \begin{array}{cc} d_{\ol{V}} & 0 \end{array}\right)} &&&&
\ol{V}_0.
}\]
The higher symmetric structures $\Theta_s$ are just given by the maps $\ol{\Theta}_s$ for $s=1,2,3,4$ except for the map:\vspace{-0.2cm}
\[\Theta_1 = \left(\begin{array}{cc} \ol{\Theta}_1\, , & -\ol{f} T \ol{\Theta}_2 \end{array} \right)^T \colon \ol{V}^4 \to \ol{V}_1 \oplus B_2.\]
Next, we take the algebraic Poincar\'{e} thickening (Definition \ref{Defn:algThomcomplexandthickening}) of $V$ to get:
\[i_V \colon \partial V \to V^{4-*},\]
where, as in Section \ref{Section:preliminaries}, we define the complex $V^{4-*}$ by:
\[(V^{4-*})_r = \Hom_{\zhd}(V_{4-r},\zhd),\]
with boundary maps
$\partial^* \colon (V^{4-*})_{r+1} \to (V^{4-*})_{r}$
given by
$\partial^* = (-1)^{r+1}d^*_V,$ where $d^*_V$ is the coboundary map.
By \cite[Part~I,~Proposition~4.1~(i)]{Ranicki3}, the operation of algebraic surgery does not change the homotopy type of the boundary.  There is therefore a chain equivalence:
\[(N_K, 0 \cup_{\phi^K \oplus -\phi^U} \Phi^K \oplus -\Phi^U) \xrightarrow{\sim} (\partial V, \partial \Theta), \]
so that using the composition of the relevant maps in:
\[N_K = E^K \cup_{E^K \oplus E^U} Y^K \oplus Y^U \xrightarrow{\sim} \partial V \to V^{4-*}\]
we again have a 4-dimensional symmetric Poincar\'{e} triad:
\[\xymatrix @C+1cm {\ar @{} [dr] |{\sim}
(E^K,\phi^K) \oplus (E^U,-\phi^U) \ar[r] \ar[d] & (E^K,0) \ar[d]
\\ (Y^K,\Phi^K) \oplus  (Y^{U},-\Phi^{U}) \ar[r] & (V^{4-*},0).
}\]
To complete the proof we need to check the homology conditions of Definition \ref{defn:2ndorderconcordant}, namely that $V^{4-*}$ has the $\Z$-homology of a circle and the consistency condition that there is an isomorphism $\xi' \colon H' \toiso H_1(\Z[\Z] \otimes_{\zhd} V^{4-*})$.  We have:
\[H_4(\Z \otimes_{\zhd} V^{4-*}) \cong H^0(\Z \otimes_{\zhd} \ol{V}) \cong H^0(W,M_K;\Z) \cong H_4(W;\Z) \cong 0,\text{ and}\]
\[H_0(\Z \otimes_{\zhd} V^{4-*}) \cong H^4(\Z \otimes_{\zhd} \ol{V}) \cong H^4(W,M_K;\Z) \cong H_0(W;\Z) \cong \Z,\]
as required.
For each basis element $(0,\dots,0,1,0,\dots,0) \in B^2$, where the $1$ is in the $i$th entry, we have, for $v \in \ol{V}_2$, \[\ba{rcl} \ol{f}^*(0,\dots,0,1,0,\dots,0)(v) & = & (0,\dots,0,1,0,\dots,0)\ol{f}(v) \\ & = & (0,\dots,0,1,0,\dots,0)(l_1,\dots,l_k)^T(v) = l_i(v). \ea\]
This means, since no $l_i$ lies in the image of $d_{\ol{V}}^* \colon \ol{V}^1 \to \ol{V}^2$, that the kernel $\ker((d_{\ol{V}}^*, \ol{f}^*)\colon \ol{V}^1 \oplus B^2 \to \ol{V}^2)$ is isomorphic to $\ker(d_{\ol{V}}^* \colon \ol{V}^1 \to \ol{V}^2)$, so that:\vspace{-0.2cm}
\[H_3(\Z \otimes_{\zhd} V^{4-*}) \cong H^1(\Z \otimes_{\zhd} \ol{V}) \cong H^1(W,M_K;\Z) \cong 0.\]
Also, since the $l_i$ are in the image of $\ol{f}^*$, they are no longer cohomology classes of $V^{4-*}$ as they were of $\ol{V}$.

At this point we need the dual classes; recall that we have, from Definition \ref{Defn:COTnsolvable}, classes $d'_1,\dots,d'_k \in H_2(W;\Z[\Z])$, whose images in $H_2(W;\Z)$ we also denote by $d'_1,\dots,d'_k$, which satisfy
$\lambda_1(l'_i,d'_j) = \delta_{ij}.$
We therefore have, by Poincar\'{e}--Lefschetz duality, classes: \vspace{-0.2cm}\[d_1,\dots,d_k \in H^2(W,M_K;\Z[\Z]),\]
with representative cochains which we also denote $d_1,\dots,d_k \in \ol{V}^2$.

Since, as above, the intersection form is defined in terms of the cup product, we have, over $\Z[\Z]$ and $\Z$, that:\vspace{-0.2cm}
\[l_i\ol{\Theta}_0^* d_j^* = \delta_{ij}.\]
We can use $\ol{\Theta}_0^* = T\ol{\Theta}_0$ instead of $\ol{\Theta}_0$ to calculate the cup products due to the existence of the higher symmetric structure chain homotopy $\ol{\Theta}_1$.  Then
\begin{eqnarray*} -\ol{f}\,\ol{\Theta}_0^*(d_j) &=& -\ol{f}\,\ol{\Theta}_0^*d_j^*(1) = -(l_1\ol{\Theta}_0^*d_j^*(1),\dots,l_k\ol{\Theta}_0^*d_j^*(1))^T \\ &=& -(0,\dots,0,1,0,\dots,0)^T = -e_j,\end{eqnarray*}
where the $1$ is in the $j$th position, and for $j = 1,\dots,k$ we denote the standard basis vectors by $e_j := (0,\dots,0,1,0,\dots,0)^T \in B_2.$  This means that the $d_j$ are not in the kernel of $-\ol{f}\ol{\Theta}_0^*$.  Then, since $d_{\ol{V}}^*(d_j) = 0$ as the $d_j$ are cocycles in $\ol{V}$, we know that the $d_j$ are no longer cohomology classes in $H_2(\Z \otimes_{\zhd} V^{4-*})$.  The group $H^2(\Z \otimes_{\zhd} \ol{V})$ was generated by the classes $l_1,\dots,l_k,d_1,\dots,d_k$, which means that we now have $H_2(\Z \otimes_{\zhd} V^{4-*}) \cong 0.$

Moreover, over both $\Z[\Z]$ and $\Z$, taking the element $D := \sum_{i=1}^k \, a_j d_j$, for any elements $a_1,\dots,a_k \in \Z[\Z]$, we have that:
\[-\ol{f}\,\ol{\Theta}_0^*(-D) = \sum_{j=1}^k \, a_j (\ol{f}\,\ol{\Theta}_0^*d_j^*(1)) = \sum_{j=1}^k \, a_j e_j \in B_2.\]
This means that $-\ol{f}\,\ol{\Theta}_0^*$ is onto $B_2$.  Therefore:
\[H_1(\Z \otimes_{\zhd} V^{4-*}) \cong H^3(\Z \otimes_{\zhd} \ol{V}) \cong H^3(W,M_K;\Z) \cong H_1(W;\Z) \cong \Z,\]
so the first homology remains unchanged at the $\Z$ level as required.  Similarly, with $\Z[\Z]$ coefficients, we have the isomorphisms:
\begin{eqnarray*} H' & \toiso & H_1(W;\Z[\Z]) \toiso H^3(W,M_K;\Z[\Z]) \\ & \toiso & H^3(\Z[\Z] \otimes_{\zhd} \ol{V}) \toiso H_1(\Z[\Z] \otimes_{\zhd} V^{4-*}), \end{eqnarray*}
which define the map \vspace{-0.2cm}
\[\xi' \colon H' \toiso H_1(\Z[\Z] \otimes_{\zhd} V^{4-*}),\]
so that the consistency condition is satisfied.  Since $H'$ is isomorphic to the $\Z[\Z]$-homology of a finitely generated projective module chain complex which is a $\Z$-homology circle, we can apply Levine's arguments \cite[Propositions~1.1~and~1.2]{Levine2}, to see that $H'$ is of type $K$.  This completes the proof that $(1.5)$-solvable knots are second order algebraically slice, or algebraically $(1.5)$-solvable.
\end{proof}

Theorem \ref{Thm:1.5solvable=>2nd_alg_slice} shows that the homomorphism from $\C$ to $\mathcal{AC}_2$ factors through $\mathcal{F}_{(1.5)}$ as claimed.

\section{Extracting first order obstructions}\label{Chapter:extracting_1st_order_obstructions}

In this section we obtain a surjective homomorphism from $\ac2$ to Levine's algebraic concordance group $\mathcal{AC}_1$.  In itself this is an important property which a respectable notion of a second order concordance group ought to have; moreover, this is the first step in defining the \COT obstructions algebraically.

We give the definition of $\mathcal{AC}_1$ in terms of Blanchfield forms.  For proofs of its equivalence to the standard definition in terms of Seifert forms, see \cite{Kearton} and \cite{Ranicki4}.

\begin{definition}\label{Defn:Blanchfieldform}
The \emph{Blanchfield form} \cite{Blanchfield}  of a knot $K$ is the non-singular Hermitian sesquilinear pairing
\[\Bl \colon H_1(M_K;\Z[\Z]) \times H_1(M_K;\Z[\Z]) \to \Q(\Z)/\Z[\Z] = \Q(t)/\Z[t,t^{-1}]\]
adjoint to the sequence of isomorphisms
\begin{multline*}\ol{H_1(M_K;\Z[\Z])} \xrightarrow{\simeq} H^2(M_K;\Z[\Z]) \xrightarrow{\simeq} H^1(M_K;\Q(\Z)/\Z[\Z]) \\ \xrightarrow{\simeq} \Hom_{\Z[\Z]}(H_1(M_K;\Z[\Z]),\frac{\Q(\Z)}{\Z[\Z]}),\end{multline*}
given by Poincar\'{e} duality, the inverse of a Bockstein homomorphism and the universal coefficient spectral sequence (see \cite{Levine2}).

We say that a Blanchfield form is \emph{metabolic} if it has a metaboliser.  A \emph{metaboliser} for the Blanchfield form is a submodule $P \subseteq H_1(M_K;\Z[\Z])$ such that:
\[P=P^{\bot}:= \{v \in H_1(M_K;\Z[\Z])\,|\, \Bl(v,w) = 0 \text{ for all } w \in P\}.\]
\qed \end{definition}

\begin{definition}\label{Defn:AC1}
The algebraic concordance group, first defined in \cite{Levine} and which we denote $\mathcal{AC}_1$, is defined as follows.  A Blanchfield form \cite{Blanchfield} is an Alexander $\Z[\Z]$-module $H$ (Theorem \ref{Thm:Levinemodule}) with a $\Z[\Z]$-module isomorphism:
\[\Bl \colon H \toiso H^{\wedge}:= \ol{\Hom_{\Z[\Z]}(H, \Q(\Z)/\Z[\Z])},\]
which satisfies $\Bl = \Bl^{\wedge}$.
We define the Witt group of equivalence classes of Blanchfield forms, with addition by direct sum and the inverse of $(H,\Bl)$ given by $(H,-\Bl)$. We call an element $(H,\Bl)$ metabolic if there exists a metaboliser $P \subseteq H$ such that $P = P^{\bot}$ with respect to $\Bl$.  We say that $(H,\Bl)$ is equivalent to $(H',\Bl')$ if $(H \oplus H',\Bl \oplus -\Bl')$ is metabolic.
Lemma \ref{Lemma:Cancellation_Blanchfield} states the rational version of the fact that this is transitive and is therefore an equivalence relation.  The integral version is harder, but follows from the proof (see e.g. \cite[Theorems~3.10~and~4.2]{Ranicki4}) of the fact that the Witt group of Seifert forms and the Witt group of Blanchfield forms are isomorphic.
\qed\end{definition}

We only prove the rational version of the following lemma, since this is what we will need in Proposition \ref{prop:COTobstr_equiv_rln} to see that the equivalence relation used to define $\mc{COT}_{(\C/1.5)}$ is transitive.  In particular, in the proof of Proposition \ref{prop:COTobstr_equiv_rln}, we will need an explicit description of the new metaboliser, as provided by Lemma \ref{Lemma:Cancellation_Blanchfield}.

The proof given is, in the author's opinion, the correct way to prove such a statement, since it shows most clearly the correspondence of the algebra to the underlying geometry.

\begin{lemma}\label{Lemma:Cancellation_Blanchfield}
Let $(H,\Bl)$ and $(H',\Bl')$ be rational Blanchfield forms.  Suppose that $(H \oplus H',\Bl \oplus \Bl')$ is metabolic with metaboliser $P = P^\bot \subseteq H \oplus H'$, and that $(H',\Bl')$ is metabolic with metaboliser $Q = Q^\bot \subseteq H'$.  Then $(H,\Bl)$ is also metabolic, and a metaboliser is given by
\[R := \{h \in H \, | \, \exists \, q \in Q \text{ with } (h,q) \in P\} \subseteq H.\]
\end{lemma}
\begin{proof}
A Blanchfield form is the same as a $0$-dimensional symmetric Poincar\'{e} complex in the category of finitely generated $\Q[t,t^{-1}]$-modules with $1-t$ acting as an automorphism.
By \cite[Propositions~3.2.2~and~3.4.5~(ii)]{Ranicki2}, a metaboliser $P$ for a Blanchfield form $(H,\Bl)$ is the same as a $1$-dimensional symmetric Poincar\'{e} pair
\[(f \colon C \to D, (0,\Bl^{\wedge})),\]
where $C = S^0 H^\wedge$ and $D = S^0P^\wedge$, in the category of finitely generated $\Q[t,t^{-1}]$-modules with $1-t$ acting as an automorphism.  This is an algebraic null--cobordism of $(H^\wedge,\Bl^\wedge)$.
Let
\[\left(\ba{c}g \\ g'\ea\right) \colon P \to H \oplus H' \text{ and } h \colon Q \to H'\]
be the inclusions of the metabolisers.  We therefore have symmetric Poincar\'{e} pairs:
\[(\left(\ba{cc}g^\wedge & g'^\wedge\ea\right) \colon H^\wedge \oplus H'^\wedge \to P^\wedge = D_0, (0,\Bl^\wedge \oplus \Bl'^\wedge))\]
and
\[(h^\wedge \colon H'^\wedge \to Q^\wedge = D'_0,(0,-\Bl'^\wedge)).\]
We have introduced a minus sign in front of $\Bl'^\wedge$, so that we can glue the two algebraic cobordisms together along $H'^\wedge$ to yield another algebraic cobordism:
\[\xymatrix @R+0.5cm @C+1cm{ & H'^\wedge = D''_1 \ar[d]^-{\left(\ba{c} g'^\wedge \\ h^\wedge \ea \right)}\\
H^\wedge = C_0 \ar[r]^-{\left(\ba{c} g^\wedge \\ 0 \ea \right)} & P^\wedge \oplus Q^\wedge = D''_0.
}\]
From this we deduce that:
\[\ol{R}:= \im\Big(H^0(D'') \to H^0(C) \Big)\]
is a metaboliser for $\Bl^\wedge \colon H^0(C) = H^{\wedge\wedge} \times H^{\wedge\wedge} \to \Q(t)/\Q[t,t^{-1}]$, where the over--line indicates the use of the involution.
Since the identification $H^{\wedge\wedge} \cong H$ involves an involution, we have that
\[\ol{\ol{R}} = R = \im \Big( \left( \ba{cc}g  & 0 \ea\right) \colon \ker \big( \left(\ba{cc}g'  & h \ea\right)\colon P \oplus Q \to H'\big) \to H \Big),\]
is a metaboliser for $\Bl$.  Finally, this is indeed equal to
\[\{h \in H \, | \, \exists \, q \in Q \text{ with } (h,q) \in P\},\]
as required.
\end{proof}

To define the map $\ac2 \to \mc{AC}_1$, we begin by taking an element $(H,\Y,\xi) \in \ac2$, and forming the algebraic equivalent of the zero surgery $M_K$.  Recall that we denote the triple associated to the unknot by $(\{0\},\Y^U,\Id_{\{0\}})$.  We construct the symmetric Poincar\'{e} complex:
\[(N,\theta) := ((Y \oplus (\zh \otimes_{\Z[\Z]} Y^U)) \cup_{E \oplus (\zh \otimes_{\Z[\Z]} E^U)} E, (\Phi \oplus 0) \cup_{\phi \oplus -\phi^U} 0).\]
In the case that $\Y = \Y^K$ is the fundamental symmetric Poincar\'{e} triad of a knot $K$, we have that
$N = N_K \simeq C_*(M_K;\zh).$  The key observation is that the Blanchfield form can be defined purely in terms of the symmetric Poincar\'{e} complex $(\Z[\Z] \otimes_{\zh} N, \Id \otimes \theta)$.

In the following, recall the standard notation $$(\Z[\Z] \otimes_{\zh} N)^i = \ol{\Hom_{\Z[\Z]}(\Z[\Z] \otimes_{\zh} N_i, \Z[\Z])}.$$

\begin{proposition}\label{Prop:chainlevelBlanchfield}
Given $[x],[y] \in H_1(\Z[\Z] \otimes_{\zh} N)$, the rational Blanchfield pairing of $[x]$ and $[y]$ is given by:
\[\Bl([x],[y]) = \frac{1}{s} \ol{z(x)}\]
where: \[x,y \in (\Z[\Z] \otimes_{\zh} N)_1,\; z \in (\Z[\Z] \otimes_{\zh} N)^1;\] \[\partial^*(z) = s\theta'_0 (y) \text{ for some } s \in \Z[\Z]- \{0\},\] and \[\theta'_0 \colon (\Z[\Z] \otimes_{\zh} N)_1 \to (\Z[\Z] \otimes_{\zh} N)^2\] is part of a chain homotopy inverse \[\theta_0' \colon (\Z[\Z] \otimes_{\zh} N)_r \to (\Z[\Z] \otimes_{\zh} N)^{3-r},\] so that \[\theta_0 \circ \theta'_0 \simeq \Id,\, \theta'_0 \circ \theta_0 \simeq \Id.\]  The Blanchfield pairing is non-singular, sesquilinear and Hermitian.
\end{proposition}

We omit the proof, since it is long but essentially comprises straight-forward computations.  See \cite[Proposition~10.2]{Powellthesis}.

\begin{proposition}\label{Thm:zeroAC2_goesto_zeroAC1}
There is a surjective homomorphism
$\ac2 \to \mathcal{AC}_1,$
which makes following diagram commute:
\[\xymatrix @C+1cm{
\C \ar[r] \ar @{->>} [d] & \mathcal{AC}_2 \ar @{->>} [d]  \\
\C/\mathcal{F}_{(0.5)} \ar[r]^{\simeq}  \ar[ur] & \mathcal{AC}_1.}\]
\end{proposition}

The bottom map is an isomorphism: see \cite[Remark~1.3.2]{COT}.

\begin{proof}
Given an element $(H,\Y,\xi) \in \ac2$, we can find the Blanchfield form on the $\Z[\Z]$-module:
\[\Bl \colon H_1(\Z[\Z] \otimes_{\zh} N) \times H_1(\Z[\Z] \otimes_{\zh} N) \to \Q(\Z)/\Z[\Z],\]
as in Proposition \ref{Prop:chainlevelBlanchfield}.  To see that addition commutes with the map $\ac2 \to \mathcal{AC}_1$, note that the Alexander modules add as in Proposition \ref{prop:2ndderivedsubgroup_adding}.  The symmetric structures also have no mixing between the chain complexes of $Y$ and $Y^{\dag}$ in the formulae in Definition \ref{Defn:connectsumalgebraic}, so that, noting that there is a Mayer--Vietoris sequence isomorphism $H_1(\Z[\Z] \otimes_{\zh} Y) \toiso H_1(\Z[\Z] \otimes_{\zh} N)$, the Blanchfield form of a connected sum in $\ac2$ is the direct sum of the two Blanchfield forms in the Witt group.  Surjectivity follows from the fact (see \cite{Levine2}) that every Blanchfield form is realised as the Blanchfield form of a knot, and therefore as the Blanchfield form of the fundamental symmetric Poincar\'{e} triad of a knot.

We will show the following, which we state as a separate result, and prove after the rest of the proof of Proposition \ref{Thm:zeroAC2_goesto_zeroAC1}:
\begin{theorem}\label{thm:COT4.4chainversion}
For triple $(H,\Y,\xi) \in \ac2$ which is second order algebraically concordant to the unknot, via a 4-dimensional symmetric Poincar\'{e} pair:
\[(j \colon \zhd \otimes_{\zh} N \to V,\, (\Theta,\theta)),\]
if we define:
\[P:= \ker(j_* \colon H_1(\Q[\Z]\otimes_{\zhd}\zhd \otimes_{\zh} N) \to H_1(\Q[\Z] \otimes_{\zhd} V)),\]
then $P$ is a metaboliser for the rational Blanchfield form on $H_1(\Q[\Z] \otimes_{\zh} N)$.
\end{theorem}
Before proving Theorem \ref{thm:COT4.4chainversion}, we will first show how it implies Proposition \ref{Thm:zeroAC2_goesto_zeroAC1}.  The Witt group of rational Blanchfield forms is defined as in Definitions \ref{Defn:Blanchfieldform} and \ref{Defn:AC1} and Proposition \ref{Prop:chainlevelBlanchfield}, but with the coefficient ring $\Z$ replaced by $\Q$.  Now recall that the Witt group of integral Blanchfield forms injects into the Witt group of rational Blanchfield forms. To see this, first note that:
\[H_1(\Z[\Z] \otimes_{\zh} N) \rightarrowtail H_1(\Q[\Z] \otimes_{\zh} N) \cong \Q \otimes_{\Z} H_1(\Z[\Z] \otimes_{\zh} N).\]
The first map is an injection since $H_1(\Z[\Z] \otimes_{\zh} N)$ is $\Z$-torsion free (Theorem \ref{Thm:Levinemodule}), while the second map is an isomorphism as $\Q$ is flat as a $\Z$-module.  Then suppose that we have a metaboliser $P_{\Q}$ for the rational Blanchfield form.  This restricts to a metaboliser \[P_{\Z} := P_{\Q} \cap (\Z \otimes_{\Z} H_1(\Z[\Z] \otimes_{\zh} N))\] for the integral Blanchfield form, since the calculation, restricted to the image of $H_1(\Z[\Z] \otimes_{\zh} N)$, is the same for the two forms.  The symmetric structure map in the rational case is just the integral map tensored up with the rationals; $(\theta'_0)_{\Q} = \Id_{\Q} \otimes_{\Z} (\theta'_0)_{\Z}$.

Therefore, the only place that the two calculations could differ is if one took $s \in \Q[\Z] \setminus \Z[\Z]$ or $z \in (\Q[\Z] \otimes_{\zh} N)^1 \setminus (\Z[\Z] \otimes_{\zh} N)^1.$  Note that we can consider $(\Z[\Z] \otimes_{\zh} N)^1$ as a subset of $(\Q[\Z] \otimes_{\zh} N)^1$ since $\Q[\Z] \otimes_{\zh} N \cong \Q \otimes_{\Z} \Z[\Z] \otimes_{\zh} N$, and $\Q[\Z] \cong \Q \otimes_{\Z} \Z[\Z]$.  In the cases that such an $s$ or such a $z$ are chosen, we can clear denominators in the equation
$\partial^* (z) = s\theta'_0(y)$
to get
$\partial^*(nz) = ns\theta'_0(y),$
for some $n \in \Z$, so that now $ns \in \Z[\Z]$ and $nz \in (\Z[\Z] \otimes_{\zh} N)^1$. Then:
\[\frac{1}{ns}\ol{(nz)(x)} = \frac{n}{ns}\ol{z(x)} = \frac{1}{s}\ol{z(x)},\]
which is the same outcome.
By Theorem \ref{thm:COT4.4chainversion}, second order algebraically slice triples map to metabolic rational Blanchfield forms, which we have now seen restrict to metabolic integral Blanchfield forms.  By applying Proposition \ref{prop:inverseswork}, we see that we have a well--defined homomorphism as claimed.  This completes the proof of Proposition \ref{Thm:zeroAC2_goesto_zeroAC1}, modulo Theorem \ref{thm:COT4.4chainversion}.
\end{proof}

Next, we will prove Theorem \ref{thm:COT4.4chainversion}.  This theorem is an algebraic reworking of \cite[Theorem~4.4]{COT}, which we state here (for $n=1$).

\begin{theorem}[\cite{COT} Theorem 4.4]\label{Lemma:COT4.4}
Suppose $M_K$ is $(1)$-solvable via $W$.  Then the rational Blanchfield form of $M_K$ is metabolic, and in fact if we define:
\[P:= \ker(i_* \colon H_1(M_K;\Q[\Z]) \to H_1(W;\Q[\Z])),\]
then $P=P^{\bot}$ with respect to $\Bl$.
\end{theorem}

In Section \ref{Chapter:extractingCOTobstructions}, Theorem \ref{thm:COT4.4chainversion} will be crucial for the control which the Blanchfield form provides on which 1-cycles of $\Q[\Z] \otimes_{\zh} N$ bound in some 4-dimensional pair, which in turn controls which representations extend over putative algebraic slice disc exteriors.  The proof will require the following proposition (\ref{Prop:COT2.10chainversion}) of \cite{COT}.  Since we will also require the use of Proposition \ref{Prop:COT2.10chainversion} when extracting the \COT obstructions, we give the statement here in the non--commutative setting, even though this is not required for the proof of Theorem \ref{thm:COT4.4chainversion}.  Before we can do this, we need two definitions.

\begin{definition}\label{defn:PTFA}
A \emph{Poly--Torsion--Free--Abelian}, or PTFA, group $\G$ is a group which admits a finite sequence of normal subgroups $\{1\} = \G_0 \lhd \G_1 \lhd ... \lhd \G_k = \G$ such that the successive quotients $\G_{i+1}/\G_i$ are torsion-free abelian for each $i \geq 0$.
\qed\end{definition}
\begin{definition}\label{Defn:OreLocalisation}
The \emph{Ore condition} determines whether a multiplicative subset $S$ of a non-commutative ring without zero-divisors can be formally inverted. A ring $A$ satisfies the Ore condition if, given $s \in S$ and $a \in A$, there exists $t \in S$ and $b \in A$ such that $at=sb$.  Then the Ore localisation $S^{-1}A$ exists.  If $S = A-\{0\}$ then $S^{-1}A$ is a skew-field which we denote by $\mathcal{K}(A)$, or sometimes just $\K$ if $A$ is understood.
\qed \end{definition}
Note that if $A=\Z[\Z]$, then $\K(A) = \Q(\Z)$.  The rational group ring of a PTFA group satisfies the Ore condition \cite[Proposition~2.5]{COT}.  See \cite[Chapter~2]{Stenstrom} for more details on the Ore condition, such as for the fact that the Ore localisation $\K(A)$ is flat as a module over $A$.

\begin{proposition}\label{Prop:COT2.10chainversion} \cite[Proposition~2.10]{COT}
Let $\G$ be a PTFA group.  If $C_*$ is a nonnegative chain complex over $\Q\G$ which is finitely generated projective in dimensions $0 \leq i \leq n$ and such that $H_i(\Q \otimes_{\Q\G} C_*) \cong 0$ for $0 \leq i \leq n$, then $H_i(\K \otimes_{\Q\G} C_*) \cong 0$.
\end{proposition}
The statement of \cite[Proposition~2.10]{COT} is made with the hypothesis that the chain complex is finitely generated free.  We note that the statement can be relaxed to $C$ being a finitely generated projective module chain complex, since this still allows the lifting of the partial chain homotopies.

\begin{proof}[Proof of Theorem \ref{thm:COT4.4chainversion}]
A large part of this proof can be carried over verbatim from the proof of \cite[Theorem~4.4]{COT}, subject to a manifold--chain complex dictionary, as follows.  The homology of $M_K$ with coefficients in a ring $R$ should be replaced with the homology of:
$R \otimes_{\zh} N;$
the (co)homology of $W$ with coefficients in $R$ should be replaced with the (co)homology of:
$R \otimes_{\zhd} V;$ and the homology of the pair $(W,M_K)$ with coefficients in $R$ should be replaced with the homology of:
\[R \otimes_{\zhd} \mathscr{C}(j \colon \zhd \otimes_{\zh} N \to V).\]
To complete the proof we need to show that:
\begin{description}
         \item[(i)] The relative linking pairings $\beta_{rel}$ are non-singular.  This will follow from the argument in the proof of \cite[Theorem~4.4]{COT} once we show, for an algebraic $(1.5)$-solution $V$, that $H_*(\Q(\Z) \otimes_{\zhd} V) \cong 0.$  Note that this also implies by universal coefficients that $H^*(\Q(\Z) \otimes_{\zhd} V) \cong 0,$ and that $H_*(\Q[\Z] \otimes_{\zhd} V)$ is torsion, since $\Q(\Z)$ is flat over $\Q[\Z]$.
         \item[(ii)] The sequence \[TH_2(\Q[\Z] \otimes_{\zhd} \mathscr{C}(j)) \xrightarrow{\partial} H_1(\Q[\Z] \otimes_{\zh} N) \xrightarrow{j_*} H_1(\Q[\Z] \otimes_{\zhd} V)\] is exact.
       \end{description}
To prove (i) we apply Proposition \ref{Prop:COT2.10chainversion} to the chain complex \[\Q[\Z] \otimes_{\zhd} \mathscr{C}(j \circ f_- \colon \zhd \otimes_{\zh} D_- \to V).\]  Since $j \circ f_-$ induces isomorphisms on rational homology, the relative homology groups vanish:
\[H_*(\Q \otimes_{\Q[\Z]} \Q[\Z] \otimes_{\zhd} \mathscr{C}(j \circ f_-)) \cong 0.\]
Proposition \ref{Prop:COT2.10chainversion} then says that:
\[H_*(\Q(\Z) \otimes_{\Q[\Z]} \Q[\Z] \otimes_{\zhd} \mathscr{C}(j \circ f_-)) \cong 0,\]
which implies the second isomorphism of:
\begin{eqnarray*} H_*(\Q(\Z) \otimes_{\zhd} V) & \cong & H_*(\Q(\Z) \otimes_{\Q[\Z]} \Q[\Z] \otimes_{\zhd} V) \\ & \cong & H_*(\Q(\Z) \otimes_{\Q[\Z]} \Q[\Z] \otimes_{\zh} D_-).\end{eqnarray*}
Then since $\Q(\Z) \otimes_{\zh} D_-$ is given by the contractible chain complex $\Q(t) \xrightarrow{t-1} \Q(t)$, we see that $H_*(\Q(\Z) \otimes_{\zhd} V) \cong 0$.

The definitions of the relative linking pairings can be made purely algebraically using chain complexes, using the corresponding sequences of isomorphisms:
\begin{eqnarray*}\ol{TH_2(\Q[\Z] \otimes_{\zhd} \mathscr{C}(j))} &\xrightarrow{\simeq}& TH^2(\Q[\Z] \otimes_{\zhd} V) \xrightarrow{\simeq} \\ H^1(\Q(\Z)/\Q[\Z] \otimes_{\zhd} V) &\xrightarrow{\simeq}& \Hom_{\Q[\Z]}(H_1(\Q[\Z] \otimes_{\zhd} V),\Q(\Z)/\Q[\Z]); \end{eqnarray*} and
\begin{eqnarray*}
\ol{TH_1(\Q[\Z] \otimes_{\zhd} V)} &\toiso& TH^3(\Q[\Z] \otimes_{\zhd} V) \toiso \\ H^2(\Q(\Z)/\Q[\Z] \otimes_{\zhd} V)
 & \xrightarrow{\simeq}& \Hom_{\Q[\Z]}(H_2(\Q[\Z] \otimes_{\zhd} V),\Q(\Z)/\Q[\Z]).
\end{eqnarray*}
There are also explicit chain level formulae for the pairings $\beta_{rel}$ in a similar vein to that for $\Bl$ in Proposition \ref{Prop:chainlevelBlanchfield}; for us, the important point is that the above maps are indeed isomorphisms.

To prove (ii), we show that in fact $H_2(\Q[\Z] \otimes_{\zhd} \mathscr{C}(j))$ is entirely torsion.
This follows from the long exact sequence of the pair \[\Id_{\Q(\Z)} \otimes j \colon \Q(\Z) \otimes_{\zh} N \to \Q(\Z) \otimes_{\zhd} V.\]
We have the following excerpt:
\[H_2(\Q(\Z) \otimes_{\zhd} V) \to H_2(\Q(\Z) \otimes_{\zhd} \mathscr{C}(j)) \to H_1(\Q(\Z) \otimes_{\zh} N).\]
We have already seen in (i) that $H_2(\Q(\Z) \otimes_{\zhd} V) \cong 0$.  We claim that $$H_1(\Q(\Z) \otimes_{\zh} N) \cong 0,$$ which then implies by exactness that the central module $H_2(\Q(\Z) \otimes_{\zhd} \mathscr{C}(j))$ is also zero.  Then note, since $\Q(\Z)$ is flat over $\Q[\Z]$, that
\[H_2(\Q(\Z) \otimes_{\zhd} \mathscr{C}(j)) \cong \Q(\Z) \otimes_{\Q[\Z]} H_2(\Q[\Z] \otimes_{\zhd} \mathscr{C}(j)).\]  That this last module vanishes means that $H_2(\Q[\Z] \otimes_{\zhd} \mathscr{C}(j))$ is $\Q[\Z]$-torsion.  To see the claim that $H_1(\Q(\Z) \otimes_{\zh} N) \cong 0$, recall that:
\[H_1(\Q[\Z] \otimes_{\zh} N) \cong H_1(\Q[\Z] \otimes_{\zh} Y) \cong \Q \otimes_{\Z} H_1(\Z[\Z] \otimes_{\zh} Y) \cong \Q \otimes_{\Z} H,\]
and that an Alexander module $H$ is $\Z[\Z]$-torsion, so that the $\Q[\Z]$-module $\Q \otimes_{\Z} H$ is $\Q[\Z]$-torsion.  This completes the proof of (ii); and therefore completes the proof of all the points that the chain complex argument for Theorem \ref{thm:COT4.4chainversion} is not directly analogous to the geometric argument in the proof of \cite[Theorem~4.4]{COT}, completing the present proof and therefore also the proof of Proposition \ref{Thm:zeroAC2_goesto_zeroAC1}.
\end{proof}

\section{The Cochran-Orr-Teichner obstruction theory}\label{Chapter:COTobstructionthy}

Before explaining how to extract the \COT obstructions, first we need to define them. In this section we not only define but also repackage the \COT metabelian obstructions, to put them into a single pointed set, which we denote $\mathcal{COT}_{(\C/1.5)}$.  This construction involves taking large disjoint unions over all of the possible choices which are implicit in defining the \COT obstructions.  By contrast, the construction of $\ac2$ is significantly simpler, as well as having the advantage of being a group.

\COT \cite{COT} use their obstruction theory to detect that certain knots are not $(1.5)$- and $(2.5)$-solvable.  In \cite{CochranTeichner} it is shown that certain knots are $(n)$-solvable but not $(n.5)$-solvable for any $n \in \mathbb{N}_0$.  We focus on the $(1.5)$-level obstructions for this exposition.  Following \cite{Let00}, who worked on the metabelian case, \COT define representations of the fundamental group of the zero framed surgery $\rho \colon \pi_1(M_K) \to \G$,
where $\Gamma = \Gamma_1 := \Z \ltimes \Q(t)/\Q[t,t^{-1}],$ their \emph{universally $(1)$-solvable group}.  To define the semi-direct product in $\G$, $n \in \Z$ acts by left multiplication by $t^n$.  The representation:
\[\rho \colon \pi_1(M_K) \to \pi_1(M_K)/\pi_1(M_K)^{(2)} \to \Z \ltimes H_1(M_K;\Q[t,t^{-1}]) \to \Z \ltimes \Q(t)/\Q[t,t^{-1}]\]
is given by: $g \mapsto (n := \phi(g),h := gt^{-\phi(g)}) \mapsto (n,\Bl(p,h)),$
where $\phi \colon \pi_1(M_K) \to \Z$ is the abelianisation homomorphism and $t$ is a preferred meridian in $\pi_1(M_K)$, the pairing $\Bl$ is the Blanchfield form, and $p$ is an element of $H_1(M_K;\Q[t,t^{-1}])$.





Now suppose that there is $(1)$-solution $W$.  As in Theorem \ref{Lemma:COT4.4}, define $$P := \ker(i_* \colon H_1(M_K;\Q[\Z]) \to H_1(W;\Q[\Z])).$$  Then for each $p \in P$, by \cite[Theorem~3.6]{COT}, we have a representation $\wt{\rho} \colon \pi_1(W) \to \G$, which enables us to define the intersection form:
\[\lambda_2 \colon H_2(W;\Q\Gamma) \times H_2(W;\Q\Gamma) \to \Q\Gamma.\]
Since $W$ is a manifold with boundary, this will in general be a singular intersection form.  To define a non-singular form we localise coefficients: \COT use the non-commutative \emph{Ore localisation} to formally invert all the non-zero elements in $\Q\G$ to obtain a skew-field $\K$, as in Definition \ref{Defn:OreLocalisation}; note that $\G$ is a PTFA group, so the Ore localisation exists by \cite[Proposition~2.5]{COT}.



As is proved in \cite[Propositions~2.9,~2.10~and~2.11 and Lemma~2.12]{COT}, the homology of $M_K = \partial W$ vanishes with $\K$ coefficients. Therefore the intersection form on the middle dimensional homology of $W$ becomes non-singular over $\K$, so we have an element in the Witt group of non-singular Hermitian forms over $\K$.  Moreover, using Proposition \ref{Prop:COT2.10chainversion}, control over the size of the $\Z$-homology translates into control over the size of the $\K$-homology of $W$.  To explain how this gives us a well--defined obstruction, which does not depend on the choice of 4-manifold, and how this obstruction lives in a group, we define $L$-groups and the localisation exact sequence in $L$-theory.

\begin{definition}[\cite{Ranicki3} I.3]\label{Defn:Lgroups}
Two $n$--dimensional $\eps$--symmetric Poincar\'{e} finitely generated projective $A$-module chain complexes $(C,\varphi)$ and $(C',\varphi')$ are \emph{cobordant} if there is an $(n+1)$-dimensional $\eps$-symmetric Poincar\'{e} pair:
\[(f,f') \colon C \oplus C' \to D, (\delta\varphi,\varphi \oplus -\varphi').\]
The union operation of \cite[Part~I,~pages~117--9]{Ranicki3} shows that cobordism of chain complexes is a transitive relation.  The equivalence classes of symmetric Poincar\'{e} chain complexes under the cobordism relation form a group $L^n(A,\eps)$, with
\[(C,\varphi) + (C',\varphi') = (C \oplus C',\varphi \oplus \varphi'); \; -(C,\varphi) = (C,-\varphi).\]
As usual if we omit $\eps$ from the notation we assume that $\eps=1$.  In the case $n=0$, $L^0(A)$ coincides with the Witt group of non-singular Hermitian forms over $A$.
\qed \end{definition}

Note that an element of an $L$-group is in particular a symmetric \emph{Poincar\'{e}} chain complex.  This means that the intersection forms of $(1)$-solutions typically give elements of $L^0(\K)$ but not of $L^0(\Q\Gamma)$.

\begin{definition}[\cite{Ranicki2} Chapter 3]\label{defn:localisationexactsequence}
The \emph{Localisation Exact Sequence in $L$-theory} is given, for a ring $A$ without zero divisors and a multiplicative subset $S = A -\{0\}$, which satisfies the Ore condition, as follows:
\[ \cdots \to L^n(A) \to L^n(S^{-1}A) \to L^n(A,S) \to L^{n-1}(A) \to \cdots.\]
The relative $L$-groups $L^n(A,S)$ are defined to be the cobordism classes of $(n-1)$-dimensional symmetric Poincar\'{e} chain complexes over $A$ which become contractible over $S^{-1}A$, where the cobordisms are also required to be contractible over $S^{-1}A$.  For $n=2$ this is equivalent to the Witt group of $S^{-1}A/A$-valued linking forms on $H^1$ of the chain complex.


The first map $L^n(A) \to L^n(S^{-1}A)$ in the localisation sequence is given by considering a chain complex over the ring $A$ as a chain complex over $S^{-1}A$, by tensoring up using the inclusion $A \to S^{-1}A$.  The salient effect of this is that some maps become invertible which previously were not.  We say that a symmetric chain complex is $\K$-Poincar\'{e} if it is Poincar\'{e} after tensoring with $\K$.

The second map $L^n(S^{-1}A) \to L^n(A,S)$ is the boundary construction.  Let $(C_*,\varphi)$ represent an element of $L^n(S^{-1}A)$.  By clearing denominators, there is a chain complex which is chain equivalent to $(C_*,\varphi)$, in which all the maps are given in terms of $A$.  We may therefore assume that we have a symmetric but typically not Poincar\'{e} complex $(C_*,\varphi)$ over $A$, and take the mapping cone $\mathscr{C}(\varphi_0 \colon C^{n-*} \to C_*)$.  This gives, as in Definition \ref{Defn:algThomcomplexandthickening}, an $(n-1)$-dimensional symmetric Poincar\'{e} chain complex over $A$ which becomes contractible over $S^{-1}A$, since $\varphi_0$ is a chain equivalence over $S^{-1}A$, i.e. we have an element of $L^n(A,S)$.

On the level of Witt groups, this map sends a Hermitian $S^{-1}A$-non-singular intersection form over $A$, $(L ,\lambda \colon L \to L^*),$  to the linking form on $\coker (\lambda \colon L \to L^*)$ given by: $(x,y) \mapsto z(x)/s$,
where $x,y \in L^*, z \in L, sy=\lambda(z)$ \cite[pages~242--3]{Ranicki2}.

The third map $L^n(A,S) \to L^{n-1}(A)$ is the forgetful map on the equivalence relation; it forgets the requirement that the cobordisms be contractible over $S^{-1}A$, simply asking for algebraic cobordisms over $A$.
\qed \end{definition}



The obstruction theory of \COTN, for suitable representations $\pi_1(M_K) \to \G$, detects the class of $(C_*(M_K;\Q\Gamma), \backslash\Delta([M_K]))$ in $L^4(\Q\Gamma,S)$, where $S := \Q\G - \{0\}$; we have an invariant of the 3-manifold $M_K$.  The first question we ask, corresponding to $(1)$-solvability, is whether the chain complex of $M_K$ bounds over $\Q\G$.  Suppose that $K$ is a $(1)$-solvable knot.  Then  we have a symmetric Poincar\'{e} complex
$$(C_*(M_K;\Q\Gamma),\backslash\Delta([M_K])) \in \ker(L^4(\Q\G,S) \to L^3(\Q\G)).$$
The obstruction which detects that there is no $\K$-contractible null-cobordism of $C_*(M_K;\Q\G)$ therefore lies in $L^4(\K)/\im(L^4(\Q\G)).$

A $(1)$-solution $W$ defines an element of $L^4(\K)$ by taking the symmetric $\K$-Poincar\'{e} chain complex: \[(C_*(W,M_K;\K) = \K \otimes_{\Q\G} C_*(W,M_K;\Q\G), \backslash\Delta([W,M_K])).\]  The image of $L^4(\Q\G)$ represents the change corresponding to a different choice of $(1)$-solution $W$: the obstruction defined must be independent of this choice.
Since 2 is invertible in the rings $\K$ and $\Q\G$, we can do surgery below the middle dimension \cite[Part~I,~3.3~and~4.3]{Ranicki3} to see that our obstruction lives in
$L^0(\K)/\im(L^0(\Q\G)).$
Taking two choices of 4-manifold $W,W'$ with boundary $M_K$ and gluing to form
$V:= W \cup_{M_K} -W',$
we obtain a 4-manifold whose image in $L^4(\Q\G) \cong L^0(\Q\G)$ gives the difference between the Witt classes of the intersection forms of $W$ and $W'$, showing that the invariant in $L^0(\K)/\im(L^0(\Q\G))$ is well-defined.  If this obstruction does not vanish then $K$ cannot be $(1.5)$-solvable and therefore in particular is not slice.

The main obstruction theorem of \COTN, at the $(1.5)$ level, is the following:

\begin{theorem}\label{Thm:COTmaintheorem}\cite[Theorem~4.2]{COT}
Let $K$ be a knot, and define, for each $p \in H_1(M_K;\Q[\Z])$:
\[B:= (C_*(M_K;\Q\G),\backslash\Delta([M_K])) \in L^4(\Q\G,\Q\G-\{0\}).\]
Suppose that $K$ is $(1)$-solvable via a $(1)$-solution $W$.  Then there exists a metaboliser $P =P^\bot \subseteq H_1(M_K;\Q[\Z])$ such that for all $p \in P$,
\[B \in \ker(L^4(\Q\G,\Q\G-\{0\}) \to L^3(\Q\G)).\]
Suppose that $K$ is $(1.5)$-solvable via a $(1.5)$-solution $W$.  Then there exists a metaboliser $P =P^\bot \subseteq H_1(M_K;\Q[\Z])$ such that for all $p \in P$, $B=0$.
\end{theorem}

\begin{proof}
We give a sketch proof.  The fact that a meridian of $K$ maps non--trivially under $\rho$ is sufficient, as in \cite[Section~2]{COT}, to see that $C_*(M_K;\K) \simeq 0$, so that indeed $B \in L^4(\Q\G,\Q\G-\{0\})$.  The $(1)$-solvable condition ensures, by Theorem \ref{Lemma:COT4.4} and \cite[Theorem~3.6]{COT}, that certain representations extend over $\pi_1(W)$, for $(1)$-solutions $W$, so that $B \mapsto 0 \in L^3(\Q\G)$.  If $W$ is also a $(1.5)$-solution, there is a metaboliser for the intersection form on $H_2(W;\K)$: as mentioned above the fact that we have control over the rank of the $\Z$-homology translates into control on the rank of the $\K$-homology.  We have a half-rank summand on which the intersection form vanishes: the intersection form is therefore trivial in the Witt group $L^0(\K)$.  Since $L^4(\K) \cong L^0_S(\K)$ by surgery below the middle dimension, we indeed have $B=0$.
\end{proof}

We now define a pointed set, which is algebraically defined, which we call the \emph{\COT obstruction set}, and denote $(\mathcal{COT}_{(\C/1.5)},U)$.  The above exposition then enables us to define a map of pointed sets $\C/\mathcal{F}_{(1.5)} \to \mathcal{COT}_{(\C/1.5)}$: the \COT obstructions do not necessarily add well, so we are only able to consider pointed sets, requiring that $(1.5)$-solvable knots map to $U$, the marked point of $\mathcal{COT}_{(\C/1.5)}$.  The reason for this definition is that the second order \COT obstructions depend for their definitions on certain choices of the way in which the first order obstructions vanish.  More precisely, for each element $p \in H_1(M_K;\Q[\Z])$ we obtain a different representation $\pi_1(M_K) \to \G$ and therefore, if it is defined, a potentially different obstruction $B$ from Theorem \ref{Thm:COTmaintheorem}.  The following definition gives an algebraic object, $\mathcal{COT}_{(\C/1.5)}$, which encapsulates the choices in a single set.  Our second order algebraic concordance group $\ac2$ gives a single stage obstruction group from which an element of $\mathcal{COT}_{(\C/1.5)}$ can be extracted; for this see Section \ref{Section:extractingCOT_subsection}.  I would like to thank Peter Teichner for pointing out that I ought to make Definition \ref{defn:COTobstructionset_2}.

In the following definition, for intuition, $(N,\theta)$ should be thought of as corresponding to the symmetric Poincar\'{e} chain complex of the zero surgery $M_K$ on a knot in $S^3$, $\G := \Z \ltimes \Q(t)/\Q[t,t^{-1}]$, and $H$ should be thought of as corresponding to $H_1(M_K;\Q[\Z])$.  There is no requirement that $(N,\theta)$ actually is the chain complex associated to a knot: we are working more abstractly.

\begin{definition}\label{defn:COTobstructionset_2}
Let $H$ be a rational Alexander module, that is a $\Q[\Z]$-module such that $H = \Q \otimes_{\Z} H'$ for some $H' \in \mathcal{A}$.  We denote the class of such $H$ by $\Q \otimes_{\Z} \mathcal{A}$.  Let \[\Bl\colon H \times H \to \Q(t)/\Q[t,t^{-1}]\] be a non-singular, sesquilinear, Hermitian pairing, and let $p \in H$.  We define the set:
\[L^4_{H,\Bl,p}(\Q\G,\Q\G-\{0\})\]
to comprise pairs $((N,\theta \in Q^3(N)),\xi)$,
where $(N,\theta)$ is a 3-dimensional symmetric Poincar\'{e} complex over $\Q\G$ which is contractible when tensored with the Ore localisation $\K$ of $\Q\G$:
\[\K \otimes_{\Q\G} N \simeq 0,\]
which satisfies:
\[H_*(\Q \otimes_{\Q\G} N) \cong H_*(S^1 \times S^2 ; \Q);\]
and where $\xi$ is an isomorphism
\[\xi \colon H \xrightarrow{\simeq} H_1(\Q[\Z] \otimes_{\Q\G} N).\]
Using the 3-dimensional symmetric Poincar\'{e} chain complex $(\Q[\Z] \otimes_{\Q\G} N,\Id \otimes \theta)$, we can define the rational Blanchfield form (see Proposition \ref{Prop:chainlevelBlanchfield}):
\[\wt{\Bl} \colon H_1(\Q[\Z] \otimes_{\Q\G} N) \times H_1(\Q[\Z] \otimes_{\Q\G} N) \to \Q(t)/\Q[t,t^{-1}].\]
We require that:
$\Bl(x,y) = \wt{\Bl}(\xi(x),\xi(y))$
for all $x,y \in H$.  In the case that $p=0 \in H$, we have a further condition that:
\begin{equation}\label{Eqn:common_boundary}((N,\theta),\xi)_0 \cong ((\Q[\Z] \otimes_{\Q\G} N,\Id \otimes \theta),\xi) \in L^4_{H,\Bl,0}(\Q\G,\Q\G-\{0\})\end{equation}
We consider the union, for a fixed $H \in \Q \otimes_{\Z} \mathcal{A}$ and a fixed $\Bl \colon H \times H \to \Q(t)/\Q[t,t^{-1}]$:
\[\mathcal{AF}_{(\C/1.5)}(H,\Bl) := \bigsqcup_{p \in H} \, L^4_{H,\Bl,p}(\Q\G,\Q\G-\{0\}),\]
over all $p \in H$.
Next, we consider the union over all possible $H$ and $\Bl$ of a class of certain subsets of $\mathcal{AF}_{(\C/1.5)}(H,\Bl)$, namely the subsets which have one element of $L^4_{H,\Bl,p}(\Q\G,\Q\G-\{0\})$ for each $p \in H$:
\[\bigcup_{\stackrel{H \in \Q \otimes_{\Z} \mathcal{A}}{ \Bl \colon \ol{H} \xrightarrow{\simeq} \Ext_{\Q[\Z]}^1(H,\Q[\Z])} }\, \Big\{ \, \bigsqcup_{p \in H} \,\{((N,\theta),\xi)_{p}\} \subset \mathcal{AF}_{(\C/1.5)}(H,\Bl)\Big\}.\] By defining a partial ordering on this class we can make it into a set by taking an inverse limit.  For each $\Q[\Z]$-module isomorphism $\a \colon H \xrightarrow{\simeq} H^{\%}$, we define a map
\[\a_* \colon L^4_{H,\Bl,p}(\Q\G,\Q\G-\{0\}) \to L^4_{H^{\%},\Bl^{\%},\a(p)}(\Q\G,\Q\G-\{0\}),\]
where $\Bl^{\%}(x,y) := \Bl(\a^{-1}(x),\a^{-1}(y))$ by
\[((N,\theta \in Q^3(N)),\xi) \mapsto ((N,\theta \in Q^3(N)), \xi \circ \a^{-1}).\]
This defines a map:
\[\a_* \colon \mathcal{AF}_{(\C/1.5)}(H,\Bl) \to \mathcal{AF}_{(\C/1.5)}(H^{\%},\Bl^{\%}),\]
which we use to map subsets to subsets.  We say that a subset:
\[\bigsqcup_{p \in H} \,\{((N,\theta),\xi)_{p}\} \subset \mathcal{AF}_{(\C/1.5)}(H,\Bl),\]
is less than or equal to
\[\bigsqcup_{q \in H^\%} \,\{((N,\theta),\xi^{\%})_{q}\} \subset \mathcal{AF}_{(\C/1.5)}(H^{\%},\Bl^{\%}),\]
if the latter is the image of the former under $\a_*$.  We then define:
\begin{eqnarray*}\mathcal{AF}_{(\C/1.5)} := \underleftarrow{\lim} \bigg\{\bigsqcup_{p \in H} \,\{((N,\theta),\xi)_{p}\} \subset \mathcal{AF}_{(\C/1.5)}(H,\Bl)\, | \, H \in \Q \otimes_{\Z} \mathcal{A},\\ \Bl \colon \ol{H} \xrightarrow{\simeq} \Ext_{\Q[\Z]}^1(H,\Q[\Z])\bigg\}.\end{eqnarray*}

Finally, we must say what it means for two elements of $\mathcal{AF}_{(\C/1.5)}$ to be equivalent, in such a way that isotopic and concordant knots map to equivalent elements of $\mathcal{AF}_{(\C/1.5)}$, and we must define the class of the zero object, so that we have a pointed set.

The distinguished point is the equivalence class of the 3-dimensional symmetric Poincar\'{e} chain complex: \begin{multline*} U := \Big(\big(\Q\G \otimes_{\Q[\Z]} C_*(S^1 \times S^2; \Q[\Z]),\backslash\Delta([S^1 \times S^2])\big), \xi = \Id \colon \{0\} \to \{0\} \Big) \\ \in \mathcal{AF}_{(\C/1.5)}(\{0\},\Bl_{\{0\}}).\end{multline*}
We declare two elements of $\mathcal{AF}_{(\C/1.5)}$ to be equivalent, denoted $\sim$, if we can choose a representative class for the inverse limit construction of each i.e. pick representatives:
\[\bigsqcup_{p \in H} \,\{((N,\theta),\xi)_{p}\} \subset \mathcal{AF}_{(\C/1.5)}(H,\Bl) \text{ and }\bigsqcup_{q \in H^\dag} \,\{((N^{\dag},\theta^{\dag}),\xi^{\dag})_{q}\} \subset \mathcal{AF}_{(\C/1.5)}(H^{\dag},\Bl^{\dag})\]
for some $H,H^{\dag} \in \Q \otimes_{\Z} \mathcal{A}$, such that there is a metaboliser $P \subseteq H \oplus H^\dag$ of
\[\Bl \oplus - \Bl^\dag \colon H \oplus H^\dag \times H \oplus H^\dag \to \Q(\Z)/\Q[\Z]\]
for which all the elements of $L^4(\Q\G,\Q\G -\{0\})$ in the disjoint union:
\[\bigsqcup_{(p,q) \in P}\, \{((N_p \oplus N_q^{\dag},\theta_p \oplus -\theta^{\dag}_q),\xi_p \oplus \xi_q^{\dag})\} \subset \mathcal{AF}_{(\C/1.5)}(H\oplus H^{\dag},\Bl \oplus -\Bl^\dag)\]
bound a 4-dimensional symmetric Poincar\'{e} pair
\[(j_p \oplus j_q^{\dag} \colon N_p \oplus N_q^{\dag} \to V_{(p,q)},(\delta \theta_{(p,q)}, \theta_p \oplus -\theta^{\dag}_q) \in Q^4(j_p \oplus j_q^{\dag}))\]
over $\Q\G$ such that
$$H_1(\Q \otimes_{\Q\G} N_p) \toiso H_1(\Q \otimes_{\Q\G} V_{(p,q)}) \xleftarrow{\simeq} H_1(\Q \otimes_{\Q\G} N_q^\dag),$$
such that the isomorphism
\[\xi_p \oplus \xi_q^\dag \colon H \oplus H^\dag \toiso H_1(\Q[\Z] \otimes_{\Q\G} N_p) \oplus H_1(\Q[\Z] \otimes_{\Q\G} N_q^\dag)\]
restricts to an isomorphism
\[P \toiso \ker \big(H_1(\Q[\Z] \otimes_{\Q\G} N_p) \oplus H_1(\Q[\Z] \otimes_{\Q\G} N_q^\dag) \to H_1(\Q[\Z] \otimes_{\Q\G} V_{(p,q)})\big),\]
and such that the algebraic Thom complex (Definition \ref{Defn:algThomcomplexandthickening}), taken over the Ore localisation, is algebraically null-cobordant in $L^4_S(\K) \cong L^0_S(\K)$:
\[[(\K \otimes_{\Q\G} \mathscr{C}((j_p \oplus j_q^\dag)),\Id \otimes \delta \theta_{(p,q)}/(\theta_{p}\oplus-\theta^\dag_q))] = [0] \in L^4_S(\K).\]
The relation $\sim$ is an equivalence relation: see Proposition \ref{prop:COTobstr_equiv_rln}.

Taking the quotient of $\mathcal{AF}_{(\C/1.5)}$ by this equivalence relation defines the second order \COT obstruction pointed set $(\mathcal{COT}_{(\C/1.5)},U)$: there is a well--defined map from concordance classes of knots modulo $(1.5)$-solvable knots to this set, which maps $(1.5)$-solvable knots to the equivalence class of $U$, as follows.

Define $H:= H_1(M_K;\Q[\Z])$.  For each $p\in H$, we use the corresponding representation $\rho \colon \pi_1(M_K) \to \G$ to form the complex:
$$((N,\theta),\xi)_{p} := ((\Q\G \otimes_{\Z[\pi_1(M_K)]} C_*(M_K;\Z[\pi_1(M_K)]), \backslash \Delta ([M_K])),\xi)  \in L^4_{H,\Bl,p}(\Q\G,\Q\G - \{0\}).$$
This gives a well--defined map: see Proposition \ref{Prop:COTobstr_well_defined}.  This completes our description of the \COT pointed set.
\qed\end{definition}

\begin{proposition}\label{prop:COTobstr_equiv_rln}
The relation $\sim$ of Definition \ref{defn:COTobstructionset_2} is indeed an equivalence relation.
\end{proposition}
\begin{proof}
To see reflexivity, note that the diagonal $H \subseteq H \oplus H$ is a metaboliser for $\Bl \oplus - \Bl$.  Then take $V_{(p,p)} := N_p$ and $\delta\theta_{(p,p)} := 0$.  It is straight--forward to see that $\sim$ is symmetric.  For transitivity, suppose that
\[\bigsqcup_{p \in H} \,((N,\theta),\xi)_{p} \sim \bigsqcup_{q \in H^\dag} \,((N^\dag,\theta^\dag),\xi^\dag)_{q}\]
with a metaboliser $P \subseteq H \oplus H^\dag$ and chain complexes $(V_{(p,q)},\delta\theta_{(p,q)})$, and that
\[\bigsqcup_{q \in H^\dag} \,((N^\dag,\theta^\dag),\xi^\dag)_{q} \sim \bigsqcup_{r \in H^\ddag} \,((N^\ddag,\theta^\ddag),\xi^\ddag)_{r}.\]
with a metaboliser $Q \subseteq H^\dag \oplus H^\ddag$ and chain complexes $(\ol{V}_{(q,r)},\ol{\delta\theta}_{(q,r)})$.

We define the metaboliser $R \subseteq H \oplus H^\ddag$ by
\[R:= \{(p,r) \in H \oplus H^\ddag \, | \, \exists \, q \in H^\dag \text{ with } (p,q) \in P \text{ and } (q,r) \in Q\}.\]
The proof of Lemma \ref{Lemma:Cancellation_Blanchfield} shows that this is a metaboliser.
For each $(p,r) \in R$ we can therefore choose a suitable $q$ and so glue the chain complexes:
\[(\ol{\ol{V}}_{(p,r)},\ol{\ol{\delta\theta}}_{(p,r)}) := (V_{(p,q)} \cup_{N_q^\dag} \ol{V}_{(q,r)}, \delta\theta_{(p,q)} \cup_{\theta_q^\dag} \ol{\delta\theta}_{(q,r)}),\]
to create an algebraic cobordism for each $(p,r) \in R$.  Easy Mayer-Vietoris arguments show that the inclusions $N_p \to \ol{\ol{V}}_{(p,r)}$ and $N_r^\ddag \to \ol{\ol{V}}_{(p,r)}$ induce isomorphisms on first $\Q$-homology, and that
\[\xi_p \oplus \xi_r^\ddag \colon H \oplus H^\ddag \toiso H_1(\Q[\Z] \otimes_{\Q\G} N_p) \oplus H_1(\Q[\Z] \otimes_{\Q\G} N_r^\ddag)\]
restricts to an isomorphism
\[R \toiso \ker \big(H_1(\Q[\Z] \otimes_{\Q\G} N_p) \oplus H_1(\Q[\Z] \otimes_{\Q\G} N_r^\ddag) \to H_1(\Q[\Z] \otimes_{\Q\G} \ol{\ol{V}}_{(p,r)})\big).\]  Since $\K \otimes_{\Q\G} N_q^\dag \simeq 0$, the elements of $L^4_S(\K)$ add and we still have the zero element of $L^4_S(\K)$ as required.
\end{proof}

\begin{proposition}\label{Prop:COTobstr_well_defined}
The map $\C/\mathcal{F}_{(1.5)} \to \mc{COT}_{(\C/1.5)}$ in Definition \ref{defn:COTobstructionset_2} is well--defined.
\end{proposition}
\begin{proof}
To see that the map is well--defined, we show that if $K \, \sharp \, -K^\dag$ is $(1.5)$-solvable, then the image of $K$ is equivalent to the image of ${K^\dag}$ in $\mc{COT}_{(\C/1.5)}$.  Let $W$ be a $(1.5)$-solution for $K \, \sharp \, -K^\dag$, and let $$P := \ker(H_1(M_K;\Q[\Z]) \oplus H_1(M_{K^\dag};\Q[\Z]) \to H_1(W;\Q[\Z])),$$
noting that $$H_1(M_K;\Q[\Z]) \oplus H_1(M_{K^\dag};\Q[\Z]) \toiso H_1(M_{K \, \sharp \, -K^\dag};\Q[\Z]).$$
We define, for all $(p,q) \in P$, $V_{(p,q)} := C_*(W,M_{K \, \sharp \, -K^\dag};\Q\G)$ to be the chain complex of $W$ relative to $M_{K \, \sharp \, -K^\dag}$.

Then $\K \otimes_{\Q\G} V_{(p,q)}$ represents an element of $L^4_S(\K)$ as in Definition \ref{defn:localisationexactsequence}.  Since $W$ is a $(1.5)$-solution, as in Theorem \ref{Thm:COTmaintheorem}, we have $B=0$.  That is, the intersection form of $V_{(p,q)}$ is hyperbolic as required.

Applying the algebraic Poincar\'{e} thickening (Definition \ref{Defn:algThomcomplexandthickening}) yields a symmetric Poincar\'{e} pair
$C_*(M_{K \, \sharp \, -K^\dag};\Q\G)_{(p,q)} \to V^{4-*}_{(p,q)}.$
Now note that $$C_*(M_{K \, \sharp \, -K^\dag};\Q\G)_{(p,q)} \simeq C_*(X_K \cup S^1 \times S^1 \times I \cup X_{K^\dag};\Q\G)_{(p,q)}.$$
By gluing the chain complex $C_*(S^1 \times D^2 \times I;\Q\G)$ to $V_{(p,q)}^{4-*}$ along $C_*(S^1 \times S^1 \times I;\Q\G)$, we obtain a symmetric Poincar\'{e} pair \[(C_*(M_K;\Q\G)_{p} \oplus C_*(M_{K^\dag};\Q\G)_q \to \widehat{V}_{(p,q)},(\widehat{\delta\theta}_{(p,q)},\theta_p \oplus -\theta_q^\dag)).\]
This gluing does not change the element of $L^4_S(\K)$ produced, since $C_*(S^1 \times D^2 \times I;\K) \simeq 0.$  We therefore indeed have that $K$ and $K^\dag$ map to equivalent elements in $\mc{COT}_{(\C/1.5)}$, as claimed.
\end{proof}

\section[Extracting the Cochran-Orr-Teichner obstructions]{Extracting the Cochran-Orr-Teichner Concordance Obstructions}\label{Chapter:extractingCOTobstructions}\label{Section:extractingCOT_subsection}

In this section we define a map $\mathcal{AC}_2 \to \mc{COT}_{(\C/1.5)}$ and show that it is a morphism of pointed sets.  Recall that $\G := \Z \ltimes \Q(t)/\Q[t,t^{-1}]$.  A map $\C/\mathcal{F}_{(1.5)} \to \mc{COT}_{(\C/1.5)}$ was implicitly defined in Section \ref{Chapter:COTobstructionthy}.  We will prove the following theorem:

\begin{theorem}\label{Thm:zeroAC2_goes_to_zeroL4QG}
A triple in $\mathcal{AC}_2$ which is second order algebraically concordant to the triple of the unknot has zero \COT metabelian obstruction; i.e. it maps to $U$ in $\mc{COT}_{(\C/1.5)}$.  See Theorem \ref{Thm:betterstatementzeroAC2_to_zeroL4QG} for a more general and precise statement.
\end{theorem}

We can summarise the results of this section in the following diagram:
\[\xymatrix{
\C/\mathcal{F}_{(1.5)} \ar @{-->} [dr] \ar[r] & \ac2 \ar @{-->} [d]  \\
 & \mathcal{COT}_{(\C/1.5)}.
}\]
Recall that we use dotted arrows for morphisms of pointed sets.

To define the map $\ac2 \to \mc{COT}_{(\C/1.5)}$, as in Section \ref{Chapter:extracting_1st_order_obstructions}, we begin by taking an element $(H,\Y,\xi) \in \ac2$, and forming the algebraic equivalent of the zero surgery $M_K$.  We construct the symmetric \emph{Poincar\'{e}} complex:
\[(N,\theta) := ((Y \oplus (\zh \otimes_{\Z[\Z]} Y^U)) \cup_{E \oplus (\zh \otimes_{\Z[\Z]} E^U)} E, (\Phi \oplus 0) \cup_{\phi \oplus -\phi^U} 0).\]
By defining representations $\Z \ltimes H \to \G$, we will obtain elements of $L^4(\Q\G,\Q\G-\{0\})$.  Recall that $L^4(\Q\G,\Q\G-\{0\})$ is the group of 3-dimensional symmetric Poincar\'{e} chain complexes over $\Q\G$ which become contractible when we tensor over the Ore localisation (Definition \ref{Defn:OreLocalisation}) $\K$ of $\Q\G$ with respect to $\Q\G - \{0\}$.  The group $L^4(\Q\G,\Q\G-\{0\})$ fits into the localisation exact sequence:
\[L^4(\Q\G) \to L^4(\K) \to L^4(\Q\G,\Q\G-\{0\}) \to L^3(\Q\G).\]
The reduced $L^{(2)}$-signature (\cite[Section~5]{COT}) obstruct the vanishing of an element of $L^0(\K)/\im(L^0(\Q\G))$.  After the proof of Theorem \ref{Thm:zeroAC2_goes_to_zeroL4QG}, we will describe how to define these signatures purely in terms of the algebraic objects in $\ac2$.  By making use of a result of Higson-Kasparov \cite{HigsonKasparov} which applies to PTFA groups, we do not need to appeal to geometric 4-manifolds to calculate the Von Neumann $\rho$-invariants.

In order to define a representation $\rho \colon \Z \ltimes H \to \G$, first we choose a $p \in H$, and then define:
\[\rho \colon (n,h) \mapsto (n,\Bl(p,h)) \in \G,\]
where $\Bl$ is the Blanchfield pairing, which is defined on $H$ as follows.

Compose $\xi$ with the rationalisation map, to get:
\[\xi \colon H \xrightarrow{\simeq} H_1(\Z[\Z] \otimes_{\zh} N) \rightarrowtail H_1(\Q[\Z] \otimes_{\zh} N).\]
The second map is injective by Theorem \ref{Thm:Levinemodule} (b): $H$ is $\Z$-torsion free.  In this section we abuse notation and also refer to this composition of $\xi$ with the rationalisation map as $\xi$.

We define $\Bl \colon H \times H \to \Q(t)/\Q[t,t^{-1}]$ by:
\[\Bl(p,h) := \Bl(\xi(p),\xi(h)).\]



\begin{proposition}\label{Prop:NinL4QG}
The chain complex:
$(\Q\G \otimes_{\zh} N, \Id \otimes \theta) $
defines an element of $L^4(\Q\G,\Q\Gamma - \{0\})$.  That is, $\K \otimes_{\Q\G} \Q\G \otimes_{\zh} N$ is contractible.
\end{proposition}

\begin{proof}
First note that $\G$ is a PTFA group (Definition \ref{defn:PTFA}), since $[\G,\G] = \Q(t)/\Q[t,t^{-1}]$; therefore $[\G,\G]$ is abelian and $\G/[\G,\G] \cong \Z.$  The fact that $\G$ is PTFA means that, by \cite[Proposition 2.5]{COT}, the Ore localisation of $\Q\G$ with respect to non-zero elements $\Q\G -\{0\}$ exists.
The proof follows that of \cite[Proposition~2.11]{COT} closely, but in terms of chain complexes.  The chain complex of the circle $C_*(S^1;\Q[\Z])$ is given by
$\Q[\Z] \xrightarrow{t-1} \Q[\Z].$
Tensor with $\Q\G$ over $\Q[\Z]$ using the homomorphism $\rho \circ (f_-)_*$, where we have to define $(f_-)_* \colon \Z \to \Z \ltimes H$.  Recall that $f_-$ is a chain map in our symmetric Poincar\'{e} triad $\Y$ (Definition \ref{Defn:algebraicsetofchaincomplexes}), and so we define $(f_-)_*$ to be the corresponding homomorphism of groups: there is, as ever, a symbiosis between the group elements and the 1-chains of the complex.  The homomorphism $(f_-)_* \colon \Z \to \Z \ltimes H$ sends $t \mapsto (1,h_1)$, where $h_1$ is, as in Definition \ref{Defn:algebraicsetofchaincomplexes}, the element of $H$ which makes $f_-$ a chain map.  Thus, passing from $C_*(S^1;\Q[\Z])$ to $C_*(S^1;\Q\G)$, we obtain:
\[\Q\G \otimes_{\Q[\Z]} \Q[\Z] \cong \Q\G \xrightarrow{(\rho\circ (f_-)_*(t) - 1)} \Q\G \otimes_{\Q[\Z]} \Q[\Z] \cong \Q\G.\]
The chain map \[1 \otimes f_- \colon C_*(S^1;\Q\G) = \Q\G \otimes_{\zh} D_- \to \Q\G \otimes_{\zh} Y \to \Q\G \otimes_{\zh} N,\] is 1-connected on rational homology.  Therefore, by the long exact sequence of a pair,
\[H_k(\Q \otimes_{\Q\G} \mathscr{C}(1 \otimes f_- \colon C_*(S^1;\Q\G) \to \Q\G \otimes_{\zh} N)) \cong 0\]
for $k = 0, 1$.
We apply Proposition \ref{Prop:COT2.10chainversion}, with $n=1$ and $C_* = \mathscr{C}(1 \otimes f_-)$, to show that:
\[H_k(\K \otimes_{\Q\G} \mathscr{C}(1 \otimes f_- \colon C_*(S^1;\Q\G) \to \Q\G \otimes_{\zh} N)) \cong 0\]
for $k = 0,1$.  This implies, again by the long exact sequence of a pair, that there is an isomorphism $H_0(S^1;\K) \cong H_0(\K \otimes_{\zh} N)$
and a surjection
$H_1(S^1;\K) \twoheadrightarrow H_1(\K \otimes_{\zh} N).$
As in the proof of \cite[Proposition~2.11]{COT}, $t$ maps to a non-trivial element \[\rho\circ (f_-)_*(t) = \rho(1,h_1) = (1,\Bl(p,h_1)) \in \Gamma.\]    Therefore $\rho\circ (f_-)_*(t) -1 \neq 0 \in \Q \G$ is invertible in $\K$, so $H_*(S^1;\K) \cong 0$.  This then implies that $H_k(\K \otimes_{\zh} N) \cong 0$
for $k=0,1$.

The proof that $\Q\G \otimes_{\zh} N$ is acyclic over $\K$ is then finished by applying Poincar\'{e} duality and universal coefficients.  The latter theorem is straight-forward since $\K$ is a skew-field, so we see that:
\[H_k(\K \otimes_{\Q\G} (\Q\G \otimes_{\zh} N)) \cong 0\]
for $k=2,3$ as a consequence of the corresponding isomorphisms for $k=0,1$.  A projective module chain complex is contractible if and only if its homology modules vanish \cite[Proposition~3.14~(iv)]{Ranicki}, which completes the proof.
\end{proof}

\begin{remark}
We can always define, for any representation which maps $g_1$ to a non-trivial element of $\G$, a map $\ac2 \to L^4(\Q\G,\Q\G - \{0\}).$  However, we will only show that it has the desired property: namely that it maps $0 \in \ac2$ to $0 \in L^4(\Q\G,\Q\G - \{0\})$, in the case that $\xi(p) \in P$ (recall that $p$ was part of the definition of a representation $\rho \colon \Z \ltimes H \to \Gamma$), for at least one of the submodules $P \subseteq H_1(\Q[\Z] \otimes_{\zh} N)$ such that $P=P^{\bot}$.

This contingent vanishing for the \COT obstruction theory  is encoded in the definition of $\mc{COT}_{(\C/1.5)}$: see Definition \ref{defn:COTobstructionset_2}.  We have a two stage definition of the metabelian \COT obstruction set, since we need the Blanchfield form to define the elements and to restrict the allowable null--bordisms; whereas an element of the group $\ac2$ is defined in a single stage from the geometry, via a handle decomposition of the knot exterior, and the allowable null--bordisms are restricted by the consistency square.  Both stages of the \COT obstruction can be extracted from the single stage element of $\ac2$.
\end{remark}


\begin{definition}\label{Defn:map_AC2_to_COT_(C/1.5)}
We define the map $\ac2 \to \mc{COT}_{(\C/1.5)}$ by mapping a triple $(H,\Y,\xi)$ to
\[\bigsqcup_{p \in \Q \otimes_{\Z} H} \, \{(\Q\G \otimes_{\zh} N, \Id \otimes \theta)_p,\xi_p\}, \]
with each $(\Q\G \otimes_{\zh} N)_p$ defined using \[\ba{rcl} \rho \colon \Z \ltimes H &\to & \G \\ (n,h) &\mapsto & (n,\Bl(p,h))\ea\]
and $\xi_p$ given by the composition
\begin{eqnarray*} \xi_p \colon \Q \otimes_{\Z} H &\xrightarrow{\Id \otimes \xi}& \Q \otimes_{\Z} H_1(\Z[\Z] \otimes_{\zh} Y) \toiso H_1(\Q[\Z] \otimes_{\zh} Y) \\ &\toiso& H_1(\Q[\Z] \otimes_{\zh} N) \toiso H_1(\Q[\Z] \otimes_{\Q\G} (\Q\G \otimes_{\zh} N)_p). \end{eqnarray*}
The maps labelled as isomorphisms in this composition are given by the universal coefficient theorem, a Mayer-Vietoris sequence, and a simple chain level isomorphism for the final identification.
\qed\end{definition}

We prove a more general statement than that of Theorem \ref{Thm:zeroAC2_goes_to_zeroL4QG}.  The purpose of this generalisation is to show that the map of pointed sets of Definition \ref{Defn:map_AC2_to_COT_(C/1.5)} is well--defined.  Theorem \ref{Thm:zeroAC2_goes_to_zeroL4QG} is a corollary of Theorem \ref{Thm:betterstatementzeroAC2_to_zeroL4QG} by taking $(H^\dag,\Y^\dag,\xi^\dag) = (\{0\},\mathcal{Y}^U,\Id_{\{0\}})$.

\begin{theorem}\label{Thm:betterstatementzeroAC2_to_zeroL4QG}
Let $(H,\Y,\xi) \sim (H^\dag,\Y^\dag,\xi^\dag) \in \ac2$  be equivalent triples.  Then
\[\bigsqcup_{p \in H} \, \{(\Q\G \otimes_{\zh} N)_p,\xi_p\} \sim \bigsqcup_{q \in H^\dag} \, \{(\Q\G \otimes_{\zhdag} N^\dag)_q,\xi^\dag_q\}  \in \mc{COT}_{(\C/1.5)}.\]
That is, there exists a metaboliser \[P = P^{\bot} \subseteq (\Q \otimes_{\Z} H) \oplus (\Q \otimes_{\Z} H^\dag)\] for the rational Blanchfield form $$\Bl \oplus -\Bl^\dag \colon (\Q \otimes_{\Z} H) \oplus (\Q \otimes_{\Z} H^\dag) \times (\Q \otimes_{\Z} H) \oplus (\Q \otimes_{\Z} H^\dag) \to \Q(t)/\Q[t,t^{-1}],$$ such that, for any $(p,q) \in (\Q \otimes_{\Z} H) \oplus (\Q \otimes_{\Z} H^\dag)$, the corresponding element
\[((\Q\G \otimes_{\zh} N)_p,\theta_p) \oplus ((\Q\G \otimes_{\zh} N^\dag)_q,-\theta^\dag_q) \in L^4(\Q\G,\Q\G-\{0\}),\]
bounds a 4-dimensional symmetric Poincar\'{e} pair
\[(j_p \oplus j_q^{\dag} \colon (\Q\G \otimes_{\zh} N)_p \oplus (\Q\G \otimes_{\zh} N^\dag)_q \to V_{(p,q)},(\delta \theta_{(p,q)}, \theta_p \oplus -\theta^{\dag}_q))\]
over $\Q\G$ such that
$$H_1(\Q \otimes_{\Q\G} (\Q\G \otimes_{\zh} N)_p) \toiso H_1(\Q \otimes_{\Q\G} V_{(p,q)}) \xleftarrow{\simeq} H_1(\Q \otimes_{\Q\G} (\Q\G \otimes_{\zh} N^\dag)_q),$$
such that the isomorphism
\begin{multline*}\xi_p \oplus \xi_q^\dag \colon (\Q \otimes_{\Z} H) \oplus (\Q \otimes_{\Z} H^\dag) \toiso \\ H_1(\Q[\Z] \otimes_{\Q\G} (\Q\G \otimes_{\zh} N)_p) \oplus H_1(\Q[\Z] \otimes_{\Q\G} (\Q\G \otimes_{\zhdag} N^\dag)_q)\end{multline*}
restricts to an isomorphism
\[P \toiso \ker \big(H_1(\Q[\Z] \otimes_{\zh} N) \oplus H_1(\Q[\Z] \otimes_{\zhdag} N^\dag) \to H_1(\Q[\Z] \otimes_{\Q\G} V_{(p,q)})\big),\]
and such that the algebraic Thom complex (Definition \ref{Defn:algThomcomplexandthickening}), taken over the Ore localisation, is algebraically null-cobordant in $L^4_S(\K) \cong L^0_S(\K)$:
\[[(\K \otimes_{\Q\G} \mathscr{C}((j_p \oplus j_q^\dag)),\Id \otimes \delta \theta_{(p,q)}/(\theta_{p}\oplus-\theta^\dag_q))] = [0] \in L^4(\K).\]
\end{theorem}

\begin{proof}
By the hypothesis we have a symmetric Poincar\'{e} triad over $\zhd$:
\[\xymatrix @R+0.5cm @C+0.8cm {\ar @{} [dr] |{\stackrel{(\gamma,\g^{\dag})}{\sim}} (E,\phi) \oplus (E^{\dag},-\phi^{\dag}) \ar[r]^-{(\Id, \Id \otimes \varpi_{E^{\dag}})} \ar[d]_{\left( \begin{array}{cc} \eta & 0 \\ 0 & \eta^{\dag}\end{array}\right)} & (E,0) \ar[d]^{\delta}
\\ (Y,\Phi) \oplus (Y^{\dag},-\Phi^{\dag}) \ar[r]^-{(j,j^{\dag})} & (V,\Theta),
}\]
with isomorphisms
\[H_*(\Z \otimes_{\zh} Y) \toiso H_*(\Z \otimes_{\zhd} V) \xleftarrow{\simeq} H_*(\Z \otimes_{\zhdag} Y^\dag),\]
and a commutative square
\[\xymatrix @R+0.5cm @C+1cm{H \oplus H^{\dag} \ar[r]^{(j_{\flat},j_{\flat}^{\dag})} \ar[d]^{\left(\begin{array}{cc} \xi & 0 \\ 0 & \xi^{\dag} \end{array} \right)} & H' \ar[d]^{\xi'} \\
H_1(\Z[\Z] \otimes_{\zh} Y) \oplus H_1(\Z[\Z] \otimes_{\zh} Y^{\dag}) \ar[r]^-{\Id_{\Z[\Z]} \otimes (j_*,j^{\dag}_*)} & H_1(\Z[\Z] \otimes_{\zhd} V).
}\]
Corresponding to the manifold triad
\[\xymatrix @C+1cm{S^1 \times S^1 \sqcup S^1 \times S^1 \ar[r] \ar[d] & S^1 \times S^1 \times I \ar[d] \\ S^1 \times D^2 \sqcup S^1 \times D^2 \ar[r] & S^1 \times D^2 \times I,}\]
we have a symmetric Poincar\'{e} triad.
\[\xymatrix @R+0.5cm @C+1cm { (E^U,-\phi^U) \oplus (E^{U},\phi^U) \ar[r]^-{(\Id, \Id)} \ar[d]_{\left( \begin{array}{cc} \eta^U & 0 \\ 0 &  \eta^{U}\end{array}\right)} & (E^U,0) \ar[d]^{\delta^U}
\\ (Y^U,0) \oplus (Y^U,0) \ar[r]^-{(j^U,j^U)} & (Y^U,0).
}\]
With this triad tensored up over $\zhd$ sending $t \mapsto g_1$ as usual, we glue the two triads together as follows:
\[\xymatrix @R+0.4cm @C+1cm{(Y^U,0) \oplus (Y^U,0) \ar[r]^-{(j^U,j^U)} & (Y^U,0) \\
(E^U,-\phi^U) \oplus (E^U,\phi^U) \ar[u]^{\left( \begin{array}{cc} \eta^U & 0 \\ 0 &  \eta^{U}\end{array}\right)} \ar[r]^-{(\Id,\Id)} & (E^U,0) \ar[u]^{\delta^U} \\
(E,\phi) \oplus (E^\dag,-\phi^\dag) \ar @{} [dr] |{\stackrel{(\gamma,\g^{\dag})}{\sim}} \ar[u]_{\cong}^{\left( \begin{array}{cc} \varpi_E & 0 \\ 0 &  \varpi_{E^\dag}\end{array}\right)} \ar[r]^-{(\Id, \Id \otimes \varpi_{E^{\dag}})} \ar[d]_-{\left( \begin{array}{cc} \eta & 0 \\ 0 & \eta^{\dag}\end{array}\right)} & (E,0) \ar[u]^{\cong}_{\varpi_{E}} \ar[d]^{\delta} \\
(Y,\Phi) \oplus (Y^{\dag},-\Phi^{\dag}) \ar[r]^-{(j,j^{\dag})} & (V,\Theta),}\]
to obtain a symmetric Poincar\'{e} pair over $\zhd$:
\[((i,i^\dag)\colon N \oplus N^\dag \to \widehat{V} := V \cup_E Y^U,(\widehat{\Theta}:=\Theta \cup 0,\theta \oplus - \theta^{\dag})).\]
We can define $P$, by Theorem \ref{thm:COT4.4chainversion}, to be
\begin{multline*} P := \ker((\Q \otimes_{\Z} H) \oplus (\Q \otimes_{\Z} H^\dag) \to H_1(\Q[\Z] \otimes_{\zh} N) \oplus H_1(\Q[\Z] \otimes N^\dag) \\ \to  H_1(\Q[\Z] \otimes_{\zhd} \widehat{V})).\end{multline*}
Now, for all $(p,q) \in P$, the representation
\begin{multline*} (\Bl\oplus -\Bl^\dag)((\xi(p),\xi^\dag(q)),\bullet) \colon H_1(\Q[\Z] \otimes_{\zh} N) \oplus H_1(\Q[\Z] \otimes_{\zh} N^\dag) \to \frac{\Q(t)}{\Q[t,t^{-1}]},\end{multline*}
extends, by \cite[Theorem~3.6]{COT}, to a representation $H_1(\Q[\Z] \otimes_{\zhd} \widehat{V}) \to \Q(t)/\Q[t,t^{-1}].$  This holds since the proof of \cite[Theorem~3.6]{COT} is entirely homological algebra, so carries over to the chain complex situation without the need for additional arguments.  We therefore have an extension:
\[\xymatrix @R+0.4cm @C+1cm{H \oplus H^\dag \ar[r]^-{(j_{\flat},j^\dag_{\flat})} \ar[d]^{\cong}_{\left(\begin{array}{cc} \xi & 0 \\ 0 & \xi^{\dag} \end{array} \right)} & H' \ar[d]^{\cong}_{\xi'} \\
H_1(\Z[\Z] \otimes_{\zh} N) \oplus H_1(\Z[\Z] \otimes_{\zhdag} N^\dag) \ar[r]^-{\Id_{\Z[\Z]} \otimes (i,i^\dag)} \ar @{>->} [d] & H_1(\Z[\Z] \otimes_{\zhd} \widehat{V}) \ar @{>->} [d] \\
H_1(\Q[\Z] \otimes_{\zh} N) \oplus H_1(\Q[\Z] \otimes_{\zhdag} N^\dag) \ar[r]^-{\Id_{\Q[\Z]} \otimes (i,i^\dag)} \ar[dr] |{(\Bl\oplus\-\Bl^\dag)((\xi(p),\xi^\dag(q)),\bullet)} & H_1(\Q[\Z] \otimes_{\zhd} \widehat{V}) \ar[d] \\
 & \Q(t)/\Q[t,t^{-1}].
}\]
Noting that, from the Mayer-Vietoris sequence for $\widehat{V} = V \cup_E Y^U$, there is an isomorphism \[H_1(\Z[\Z] \otimes_{\zhd} V) \toiso H_1(\Z[\Z] \otimes_{\zhd} \widehat{V}),\] the top square commutes by the consistency condition.  We therefore have an extension of representations:\vspace{-0.2cm}
\[\xymatrix @C+1cm{\Z \ltimes (H\oplus H^\dag) \ar[r]^{(\Id_{\Z},(j_{\flat},j_{\flat}^\dag))} \ar[dr]_{\rho} & \Z \ltimes H' \ar[d]^{\wt{\rho}} \\
 & \G.
}\]
The element $$((\Q\G \otimes_{\zh} N)_p,\theta_p) \oplus ((\Q\G \otimes_{\zhdag} N^\dag)_q,-\theta_q^\dag)  \in L^4(\Q\G,\Q\G -\{0\})$$ therefore lies, by virtue of the existence of $\Q\G \otimes_{\zhd} \widehat{V}_{(p,q)}$, in
$\ker (L^4(\Q\G,\Q\G -\{0\}) \to L^3(\Q\G)).$
As in the $L$-theory localisation sequence (Definition \ref{defn:localisationexactsequence}), we therefore have the element:
\[(\ol{V}_{(p,q)},\ol{\Theta}_{(p,q)}) := ((\K \otimes_{\zhd} \mathscr{C}((i,i^\dag)))_{(p,q)},\Theta_{(p,q)}/(\theta_p \oplus - \theta_q^\dag)) \in L^4_S(\K),\]
whose boundary is $$((\Q\G \otimes_{\zh} N)_p,\theta_p) \oplus ((\Q\G \otimes_{\zhdag} N^\dag)_q,-\theta_q^\dag)  \in L^4(\Q\G,\Q\G -\{0\}).$$  Since $2$ is invertible in $\K$, we can do algebraic surgery below the middle dimension \cite[Part~I,~Proposition~4.4]{Ranicki3}, on $\ol{V}_{(p,q)}$, to obtain a non-singular Hermitian form:
\[(\lambda \colon H^2(\ol{V}_{(p,q)}) \times H^2(\ol{V}_{(p,q)}) \to \K) \in L^0_S(\K) \cong L^4_S(\K),\]
whose image in
$L^0_S(\K)/L^0(\Q\G)$
detects the class of $\Q\G \otimes_{\zh} N \in L^4(\Q\G,\Q\G -\{0\})$.
Once again, we will apply Proposition \ref{Prop:COT2.10chainversion}.  Since $j$ and $j^\dag$ induce isomorphisms on $\Z$-homology, and therefore on $\Q$-homology, we have that the chain map
\[\Id \otimes i \colon \Q \otimes_{\Q\G} (\Q\G \otimes_{\zh} N)_p \to \Q \otimes_{\Q\G} (\Q\G \otimes_{\zhd} \widehat{V}_{(p,q)})\]
induces isomorphisms
$ i_* \colon H_k(\Q \otimes_{\zh} N) \toiso H_k(\Q \otimes_{\zhd} \widehat{V})$
for all $k$, by a straight--forward Mayer-Vietoris argument.  Therefore $H_k(\Q \otimes_{\zhd} \mathscr{C}(i)) \cong 0$ for all $k$ by the long exact sequence of a pair.  Applying Proposition \ref{Prop:COT2.10chainversion}, we therefore have that
$H_k((\K \otimes_{\zhd} \mathscr{C}(i))_{(p,q)}) \cong 0$
for all $k$.  The long exact sequence in $\K$-homology associated to the short exact sequence
\[0\to (\K \otimes_{\zhd}\mathscr{C}(i))_{(p,q)} \to (\K \otimes_{\zhd} \mathscr{C}((i,i^\dag)))_{(p,q)} \to S(\K\otimes_{\zhdag} N^\dag_q) \to 0\]
implies, noting that $H_*(\K \otimes_{\zhdag} N^\dag_q) \cong 0$, that
\[H_k(\K \otimes_{\zhd} \mathscr{C}((i,i^\dag))_{(p,q)}) = H_k(\ol{V}_{(p,q)}) \cong 0\]
for all $k$.  In particular, since $H_2(\ol{V}_{(p,q)}) \cong H^2(\ol{V}_{(p,q)}) \cong 0,$ we see that the image of $\ol{V}_{(p,q)}$ in $L^0_S(\K)$, which is the intersection form $\lambda$, is trivially hyperbolic and represents the zero class of $L^0_S(\K)$.  This completes the proof that
\begin{multline*} \bigsqcup_{p \in H} \, \{(\Q\G \otimes_{\zh} N,\Id \otimes \theta)_p,\xi_p\} \sim \bigsqcup_{q \in H^\dag} \, \{(\Q\G \otimes_{\zhdag} N^\dag,\Id \otimes \theta^\dag)_q,\xi^\dag_q\} \in \mc{COT}_{(\C/1.5)}.\end{multline*}
\end{proof}

Finally, we have a non-triviality result, which shows that we can extract the $L^{(2)}$-signatures from $\ac2$.  In order to obstruct the equivalence of triples $(H,\Y,\xi) \sim (H^\dag,\Y^\dag,\xi^\dag) \in \ac2$, we just need, by Proposition \ref{prop:inverseswork}, to be able to obstruct an equivalence $(H,\Y,\xi) \sim (\{0\},\Y^U,\Id_{\{0\}})$.  To achieve this, as in Definition \ref{defn:COTobstructionset_2}, we need to obstruct the existence of a 4-dimensional symmetric Poincar\'{e} pair over $\Q\G$
$(j \colon (\Q\G \otimes_{\zh} N)_p \to V_p,(\Theta_p,\theta_p)),$
for at least one $p \neq 0$, with $\xi(p) \in P$, for each metaboliser $P = P^\bot \subseteq H_1(\Q[\Z] \otimes_{\zh} N)$ of the Blanchfield form, where $V_p$ satisfies that
\[\xi(p) \in \ker(j_* \colon H_1(\Q[\Z] \otimes_{\zh} N_p) \to H_1(\Q[\Z] \otimes_{\Q\G} V_p)),\]
that
$j_*\colon H_1(\Q \otimes_{\zh} N) \xrightarrow{\simeq} H_1(\Q \otimes_{\Q\G} V_p)$
is an isomorphism, and that
$[\K \otimes_{\Q\G} \mathscr{C}(j)] = [0] \in L^4_S(\K).$
We do this by taking $L^{(2)}$-signatures of the middle dimensional pairings on putative such $V_p$, to obstruct the Witt class in $L^4_S(\K) \cong L^0_S(\K)$ from vanishing.  First, we have a notion of algebraic $(1)$-solvability.

\begin{definition}\label{Defn:alg(1)solvable}
We say that an element $(H,\Y,\xi) \in \ac2$ with image $0 \in \mathcal{AC}_1$ is \emph{algebraically $(1)$-solvable} if the following holds. There exists a metaboliser $P = P^{\bot} \subseteq H_1(\Q[\Z] \otimes_{\zh} N)$ for the rational Blanchfield form such that for any $p \in H$ such that $\xi(p) \in P$, we obtain an element:
\[\Q\G \otimes_{\zh} N_p \in \ker(L^4(\Q\G,\Q\G -\{0\}) \to L^3(\Q\G)),\]
via a symmetric Poincar\'{e} pair over $\Q\G$: \[(j \colon \Q\G \otimes_{\zh} N_p \to V_p, (\Theta_p,\theta_p)),\] with
\[P = \ker(j_* \colon H_1(\Q[\Z] \otimes_{\zh} N) \to H_1(\Q[\Z] \otimes_{\Q\G} V_p)),\]
and such that:
\[j_*\colon H_1(\Q \otimes_{\zh} N) \xrightarrow{\simeq} H_1(\Q \otimes_{\Q\G} V_p)\]
is an isomorphism.  We call each such $(j \colon \Q\G \otimes_{\zh} N_p \to V_p, (\Theta_p,\theta_p))$ an \emph{algebraic $(1)$-solution}.
\qed\end{definition}

\begin{theorem}\label{Thm:extractingL2signatures}
Suppose that $(H,\Y,\xi) \in \ac2$ is algebraically $(1)$-solvable with algebraic $(1)$-solution $(V_p,\Theta_p)$ and $\xi(p) \in P$.  Then since:
\[\ker(L^4(\Q\G,\Q\G -\{0\}) \to L^3(\Q\G)) \cong L^4(\K)/L^4(\Q\G) \cong L^0(\K)/L^0(\Q\G),\]
we can apply the $L^{(2)}$-signature homomorphism (see \cite[Section~5]{COT}):
$\sigma^{(2)} \colon L^0(\K) \to \R,$
to the intersection form:
\[\lambda_{\K} \colon H_2(\K \otimes_{\Q\G} V_p) \times H_2(\K \otimes_{\Q\G} V_p) \to \K.\]
We can also calculate the signature $\sigma(\lambda_{\Q})$ of the ordinary intersection form:
\[\lambda_{\Q} \colon H_2(\Q \otimes_{\Q\G} V_p) \times H_2(\Q \otimes_{\Q\G} V_p) \to \Q,\]
and so calculate the reduced $L^{(2)}$-signature $\wt{\sigma}^{(2)}(V_p) = \sigma^{(2)}(\lambda_{\K}) - \sigma(\lambda_{\Q}).$  This is independent, for fixed $p$, of changes in the choice of chain complex $V_p$.
\end{theorem}

\begin{remark}
Provided we check that the reduced $L^{(2)}$-signature does not vanish, for each metaboliser $P$ of the rational Blanchfield form with respect to which $(H,\Y,\xi)$ is algebraically $(1)$-solvable, and for each $P$, for at least one $p \in P \setminus \{0\}$, then we have a \emph{chain--complex--Von--Neumann $\rho$--invariant} obstruction.  This obstructs the image of the element $(H,\Y,\xi)$ in $\mathcal{COT}_{(\C/1.5)}$ from being $U$, and therefore obstructs $(H,\Y,\xi)$ from being second order algebraically slice.

We do not require any references to 4-manifolds, other than for pedagogic reasons, to extract the \COT $L^{(2)}$-signature metabelian concordance obstructions from the triple of a $(1)$-solvable knot, or indeed for any algebraically $(1)$-solvable triple in $\ac2$.  This result relies strongly on the reason for the invariance of the reduced $L^{(2)}$-signatures which is least emphasised in the paper of \COT \cite{COT}.  This is the result of Higson-Kasparov \cite{HigsonKasparov} that the analytic assembly map is onto for PTFA groups.  The reader is encouraged to look at \cite[Proposition~5.12]{COT}, where it is shown that the surjectivity of the assembly map implies that the $L^{(2)}$-signature and the ordinary signature coincide on the image of $L^0(\Q\G)$.  The key point is that this result does not depend on manifolds for its statement; it is a purely algebraic result (although the proof of \cite[Proposition~5.12]{COT} uses Atiyah's $L^{(2)}$-Index theorem).

The Higson-Kasparov result does not hold for groups with torsion, a fact made use of in e.g. \cite{chaorr}.  Homology cobordism invariants which use representations to torsion groups appear to be using deeper manifold structure than is captured by symmetric Poincar\'{e} complexes alone.
\end{remark}

\begin{proof}[Proof of Theorem \ref{Thm:extractingL2signatures}]
For this proof we omit the $p$ subscripts from the notation; it is to be understood that tensor products with $\Q\G$ depend on a choice of representation.  Given a pair
$(j \colon \Q\G \otimes_{\zh} N \to V, (\Theta,\theta)),$
which exhibits $(H,\Y,\xi)$ as being algebraically $(1)$-solvable, we again take the element:
$(\K \otimes_{\Q\G} \mathscr{C}(j),\Theta/\theta) \in L^4(\K),$
and look at its image
$\lambda_{\K} \in L^0(\K).$
We can calculate an intersection form $\lambda_{\K}$ on $H^2(\K \otimes_{\Q\G} \mathscr{C}(j))$, as in \cite[page~19]{Ranicki2}, by taking \[x,y \in (\K\ \otimes_{\Q\G} \mathscr{C}(j))^2 \cong \Hom_{\K}((\K \otimes_{\Q\G} \mathscr{C}(j))_2,\K),\] and calculating:\vspace{-0.2cm}
\[y' = (\Theta/\theta)_0(y) \in (\K\ \otimes_{\Q\G} \mathscr{C}(j))_2.\]
Then $\lambda_{\K} (x,y) := y'(x) = \ol{x(y')} \in \K.$
This uses, as in the definition of $\Bl$ in Proposition \ref{Prop:chainlevelBlanchfield}, the identification of $(\K\ \otimes_{\Q\G} \mathscr{C}(j))_2$ with its double dual.
By taking the chain complex $\Q \otimes_{\Q\G} \mathscr{C}(j)$ we can also calculate the intersection form $\lambda_{\Q} \in L^0(\Q)$, with an analogous method.  To see that the intersection form on $H^2(\Q \otimes_{\Q\G} \mathscr{C}(j))$ is non-singular, consider the following long exact sequence of the pair; we claim that the maps labelled as $j^*$ and $\kappa$ are isomorphisms.
\[\xymatrix @C-0.3cm {H^1(\Q \otimes_{\Q\G} V) \ar[r]^-{\cong}_-{j^*} & H^1(\Q \otimes_{\zh} N) \ar[r]^{0} & H^2(\Q \otimes_{\Q\G} \mathscr{C}(j)) \ar[r]^-{\cong}_-{\kappa} & H^2(\Q \otimes_{\Q\G} V). }\]
The intersection form is given by the composition:
\begin{multline*}\lambda_{\Q} \colon H^2(\Q \otimes_{\Q\G} \mathscr{C}(j)) \xrightarrow{\kappa} H^2(\Q \otimes_{\Q\G} V) \toiso H_2(\Q \otimes_{\Q\G} \mathscr{C}(j))\\ \toiso \Hom_{\Q}(H^2(\Q \otimes_{\Q\G} \mathscr{C}(j)),\Q),\end{multline*}
given by the map $\kappa$ from the long exact sequence of a pair, followed by a Poincar\'{e} duality isomorphism induced by the symmetric structure, and a universal coefficient theorem isomorphism.  To show that $\lambda_\Q$ is non-singular we therefore need to show that $\kappa$ is an isomorphism.  The assumption that there is an isomorphism
$j_* \colon H_1(\Q \otimes_{\zh} N) \toiso H_1(\Q \otimes_{\Q\G} V)$
on rational first homology implies that, as claimed, there is also an isomorphism $j^* \colon H^1(\Q \otimes_{\Q\G} V) \toiso H^1(\Q \otimes_{\zh} N)$ on rational cohomology, by the universal coefficient theorem (the relevant $\Ext$ groups vanish with rational coefficients).  Therefore, by exactness, the map $\kappa \colon H^2(\Q \otimes_{\Q\G} \mathscr{C}(j)) \to H^2(\Q \otimes_{\Q\G} V)$
is injective. Over $\Q$, for dimension reasons, it must therefore, as marked on the diagram, be an isomorphism; the dimensions must be equal since the second and third maps in the composition which gives $\lambda_\Q$ show that $H^2(\Q \otimes_{\Q\G} V) \cong \Hom_{\Q}(H^2(\Q \otimes_{\Q\G} \mathscr{C}(j)),\Q),$ and the dimensions over $\Q$ of $\Hom_{\Q}(H^2(\Q \otimes_{\Q\G} \mathscr{C}(j)),\Q)$ and of $H^2(\Q \otimes_{\Q\G} \mathscr{C}(j))$ coincide.  Therefore the intersection form $\lambda_\Q$ is non-singular as claimed.

The reduced $L^{(2)}$-signature $\wt{\sigma}^{(2)}(V) = \sigma^{(2)}(\lambda_{\K}) - \sigma(\lambda_{\Q})$
detects non-trivial elements in the group $L^0_S(\K)/L^0(\Q\G)$.  This will follow from \cite[Proposition~5.12]{COT}, which uses a result of Higson-Kasparov \cite{HigsonKasparov} on the analytic assembly map for PTFA groups such as $\G$, and says that the $L^{(2)}$-signature agrees with the ordinary signature on the image of $L^0(\Q\G)$.  We claim that a non-zero reduced $L^{(2)}$-signature, for all possible metabolisers $P=P^{\bot}$ of the rational Blanchfield form, implies that $(H,\Y,\xi)$ is not second order algebraically slice.  To see this, we need to show that, for a fixed representation $\rho$, the reduced $L^{(2)}$-signature does not depend on the choice of chain complex $V$.

We first note, by the proof of Theorem \ref{Thm:betterstatementzeroAC2_to_zeroL4QG}, that a change in $(H,\Y,\xi)$ to an equivalent element in $\ac2$ produces an algebraic concordance which we can glue onto $V$ as in Proposition \ref{prop:equivrelation}, which neither changes the second homology of $V$ with $\K$ nor with $\Q$ coefficients, so does not change the corresponding signatures.

To show that the reduced $L^{(2)}$-signature does not depend on the choice of $V$, suppose that we have two algebraic $(1)$-solutions, that is two 4-dimensional symmetric Poincar\'{e} pairs over $\Q\G$:
\[(j \colon \Q\G \otimes_{\zh} N \to V, (\Theta,\theta)) \text{ and }
(j^{\diamondsuit} \colon \Q\G \otimes_{\zh} N \to V^{\diamondsuit}, (\Theta^{\diamondsuit},\theta)),\]
such that $p = p^{\diamondsuit} \in H$.  Use the union construction to form the symmetric Poincar\'{e} complex:
\[(V \cup_{\Q\G \otimes N} V^{\diamondsuit},\Theta \cup_{\theta} -\Theta^{\diamondsuit}) \in L^4(\Q\G).\]
Over $\K$, $\Q\G \otimes_{\zh} N$ is contractible, so that:
\[(V \cup_{\Q\G \otimes N} V^{\diamondsuit},\Theta \cup_{\theta} -\Theta^{\diamondsuit}) \simeq (V \oplus V^{\diamondsuit},\Theta \oplus -\Theta^{\diamondsuit}) = (V,\Theta) - (V^{\diamondsuit},\Theta^{\diamondsuit}) \in L^4_S(\K).\]
Therefore
$(V,\Theta) - (V^{\diamondsuit},\Theta^{\diamondsuit}) = 0 \in L^4(\K)/L^4(\Q\G),$
which means that the images in $L^0_S(\K)$ satisfy
$\lambda_{\K} - \lambda_{\K}^{\diamondsuit} = 0 \in L^0_S(\K)/L^0(\Q\G)$.
If $\lambda_{\K} - \lambda_{\K}^{\diamondsuit} \in L^0(\Q\G)$, then by \cite[Proposition~5.12]{COT}:
$$\sigma^{(2)}(\lambda_{\K} - \lambda_{\K}^{\diamondsuit}) = \sigma(\Q \otimes_{\Q\G} V \cup_{\Q\G \otimes N} V^{\diamondsuit},\Id_{\Q} \otimes (\Theta \cup_{\theta} -\Theta^{\diamondsuit})) = \sigma(\lambda_{\Q}) - \sigma(\lambda_{\Q}^{\diamondsuit}),$$
where the last equality is by Novikov Additivity.  Novikov Additivity also holds for $\sigma^{(2)}$: see \cite[Lemma~5.9.3]{COT}, so that:
\[\sigma^{(2)}(\lambda_{\K}) - \sigma^{(2)}(\lambda_{\K}^{\diamondsuit}) = \sigma(\lambda_{\Q}) - \sigma(\lambda_{\Q}^{\diamondsuit})\]
and therefore $\wt{\sigma}^{(2)}(V) = \wt{\sigma}^{(2)}(V^{\diamondsuit})$
as claimed.
\end{proof}

This completes the proof of Theorem \ref{maintheorem}.

\begin{remark}
The results of Kim \cite{Kim04}, \COT \cite{COT2} and Cochran--Harvey--Leidy \cite{cohale, cohale3_primary, cohale2}, which use Cheeger-Gromov Von Neumann $\rho$--invariants to show the existence of infinitely many linearly independent injections of $\Z$ and of $\Z_2$ into $\mathcal{F}_{(1)}/\mathcal{F}_{(1.5)}$, can also be applied, so that we can use the chain-complex-Von-Neumann $\rho$-invariant of Theorem \ref{Thm:extractingL2signatures} to show the existence of infinitely many injections of $\Z$ and $\Z_2$ into $\ker(\ac2 \to \mathcal{AC}_1)$, which in particular implies the claim in Corollary \ref{corol:infinite_rank_kernel}.
\end{remark}

\bibliographystyle{annotate}
\bibliography{markbib1}

\end{document}